\documentclass[letterpaper]{amsart}
\usepackage{ifpdf}
\ifpdf
 \usepackage[implicit,naturalnames,hyperfootnotes=false,final]{hyperref}
 \usepackage[pdftex]{color}
\else
 \usepackage{color}
 \gdef\texorpdfstring#1#2{#1}
\fi
\usepackage{lmodern}
\usepackage{xcolor}
\usepackage{enumitem}
\usepackage{amsmath}
\usepackage{amssymb}
\usepackage{amsthm}
\usepackage{bbm}
\usepackage{amsfonts}
\usepackage{mathrsfs}
\usepackage{calligra}
\usepackage{dsfont}
\usepackage{nicefrac}
\usepackage{esint}
\usepackage{mathtools}
\usepackage{xparse}
\usepackage{calc}
\usepackage[nostandardtensors]{simpletensors}
\usepackage{partial}
\usepackage{bracket-hack}

\NewDocumentCommand\regsphere{sO{\kappa}D<>{2|\mass|}D(){\Aradius}mO{p}}
{\mathcal R^{\hspace{-.05em}#6}_{\;#2\IfBooleanTF{#1}\relax{,#4}}(#3,#5)}
\ifx\undef\NotDefinedCommand\gdef\undef#1{\let#1\relax}\fi
\newtheoremstyle{mytheorem}{3pt}{}{\itshape}{}{\bfseries}{\nopagebreak\newline}{.5em}{}
\newtheoremstyle{mydefinition}{3pt}{}{}{}{\bfseries}{\nopagebreak\newline}{.5em}{}
\gdef\mytheorem{\theoremstyle{mytheorem}}
\gdef\mydefinition{\theoremstyle{mydefinition}}
\gdef\mytheoremcounter{theorem}
\mytheorem
\newtheorem{theorem}{Theorem}[section]
\newtheorem{corollary}[\mytheoremcounter]{Corollary}
\newtheorem{lemma}[\mytheoremcounter]{Lemma}

\newtheorem{proposition}[\mytheoremcounter]{Proposition}

\mydefinition
\newtheorem{definition}[\mytheoremcounter]{Definition}

\theoremstyle{remark}
\newtheorem{remark}[\mytheoremcounter]{Remark}

\gdef\ii{I}\gdef\ij{J}\gdef\ik{K}\gdef\il{L}
\gdef\oi{i}\gdef\oj{j}\gdef\ok{k}

\NewDocumentCommand\outsymbol{om}{\overline{\IfValueTF{#1}{#1{#2}}{#2}}}
\NewDocumentCommand\unisymbol{om}{\widehat{\IfValueTF{#1}{#1{#2}}{#2}}}

\makeatletter
\newdimen\middle@width
\newcommand*\phantomas[3][c]{\ifmmode\makebox[\widthof{$#2$}][#1]{$#3$}\else\makebox[\widthof{#2}][#1]{#3}\fi}
\NewDocumentCommand\tracefree{sm}
{\IfBooleanTF{#1}{\fixedheighttracefree{#2}}{\mathring{#2}}}
\NewDocumentCommand\fixedheighttracefree{om}
{\IfValueTF{#1}%
 {\setlength\middle@width{\widthof{\ensuremath{#1#2}}/3}%
  #1#2\hskip-\middle@width\phantomas\relax{\tracefree{\ }}\hskip\middle@width}%
 {\mathchoice{\fixedheighttracefree[\displaystyle]{#2}}{\fixedheighttracefree[\textstyle]{#2}}{\fixedheighttracefree[\scriptstyle]{#2}}{\fixedheighttracefree[\scriptscriptstyle]{#2}}}}
\NewDocumentCommand\mean{om}
{\IfValueTF{#1}%
 {\setlength\middle@width{\widthof{\ensuremath{#1#2}}}%
  \phantomas{#1#2}{{#1\hspace{-.1em}\textbf{\small$\boldsymbol\backslash$}\!}}%
  \hskip-\middle@width#1#2}%
 {\mathchoice{\mean[\displaystyle]{#2}}{\mean[\textstyle]{#2}}{\mean[\scriptstyle]{#2}}{\mean[\scriptscriptstyle]{#2}}}}
\gdef\meanH{\hspace{-.08em}\mean{\hspace{.08em}\H}}

\makeatother
\NewTensor[\outsymbol]\mass m
\let\oldPhi\Phi\RenewTensor\Phi\oldPhi
\let\oldphi\varphi\RenewTensor\varphi\oldphi
\NewTensor[\outsymbol]\outPhi\oldPhi
\let\oldPsi\Psi\RenewTensor\Psi<\!>\oldPsi
\NewTensor[\outsymbol]\outPsi\oldPsi
\NewTensor[\outsymbol]\outx x
\NewTensor[\outsymbol]\outy y
\NewTensor[\outsymbol]\outp p
\gdef\outsigma{\Hradius}

\gdef\schwarzs{\mathcal S}
\gdef\schwarzoutg{\outg[\schwarzs]}
\gdef\euclideane{e}
\gdef\eukg{\g[\euclideane\hspace{.05em}]}
\gdef\euknu{\nu[\euclideane\hspace{.05em}]}
\gdef\eukoutg{\outg[\euclideane]}
\gdef\euklevi{\levi[\euclideane\hspace{.05em}]}
\gdef\sphlevi{\levi[\sphg*\hspace{.05em}]}
\gdef\rad{|\outx|}
\gdef\radoutsigma{|\outx_\outsigma|}
\gdef\volume#1{\left|#1\right|}
\gdef\lieD#1#2{\mathfrak L_{#1}#2}
\undef\c\NewDocumentCommand\c{sG{c}}{\IfBooleanTF{#1}{#2}{\IndexSymbol*{#2}}}
\gdef\cSob{\c{c_S}}
\NewDocumentCommand\Cof{G{C}d<>d()}
{\IfValueTF{#3}{\Cof{#1}<\IfValueTF{#2}{#2,}\relax{#3}>}{\CofSnd{#1\IfValueTF{#2}{[#2]}\relax}}}
\NewDocumentCommand\CofSnd{G{C}d<>o}
{\IfValueTF{#3}{\CofSnd{#1}<\IfValueTF{#2}{#2,}\relax{#3}>}{\mathop{#1\IfValueTF{#2}{(#2)}\relax}}}

\NewTensor\M[\hspace{-.05em}]<\hspace{-.05em}>\Sigma
\NewTensor[\outsymbol]\outM[\hspace{-.025em}]{\textrm M}
\NewTensor[\unisymbol]\uniM[\hspace{-.025em}]{\textrm M}
\NewTensor[\outsymbol]\outexp{\text{\normalfont exp}}
\NewTensor[\unisymbol]\uniexp{\text{\normalfont exp}}
\NewTensor*\graphf[\!]f
\NewTensor\graphF F
\NewTensor\sphg{\Omega}
\NewTensor*\conformalf[\hspace{-.05em}]<\hspace{-.05em}>v
\NewDocumentCommand\conffst{O{\relax}md<>}
{\IfValueTF{#3}{\confthd{#1}{#2}{#3}}{\confsnd[#1]{#2}}}
\NewDocumentCommand\confsnd{O{\relax}mt_}
{\IfBooleanTF{#3}{\confthd{#1}{#2}}{#2}}
\gdef\confthd#1#2#3{#2_{#1#3}}
\gdef\conf{\conffst u}
\gdef\conft{\conffst v}
\gdef\confth{\conffst w}
\gdef\confmug{\conffst\mu}
\NewDocumentCommand\confgauss{o}
{\IfNoValueTF{#1}\relax{{}^{#1\!}}\conffst[\hspace{.1em}]{\gaussfont\gausssymbol}}
\gdef\cmu{\c_\mu}
\NewTensor*\rnu[\hspace{-.05em}]u
\NewTensor*\meanrnu[\hspace{-.05em}]{\mean u}
\NewTensor\rbeta[\!]\beta
\NewTensor[\MakeSymbol<>\vec]*\centerz[\!]<\hspace{-.05em}>z
\undef\metric\NewTensor[\newmathcal]\metric[\hspace{-.15em}]<\hspace{.05em}>g
\let\g\metric
\NewTensor[\newmathcal]\h[\!\!]<\hspace{.05em}>h
\NewTensor[{\outsymbol[\newmathcal]}]\outmetric<\hspace{.05em}>g
\let\outg\outmetric
\NewTensor[{\unisymbol[\newmathcal]}]\unimetric[\!\!]<\hspace{.05em}>g
\let\unig\unimetric
\NewTensor[\newmathcal]\riem R
\let\rc\riem
\NewTensor[{\outsymbol[\newmathcal]}]\outriem R
\let\outrc\outriem
\NewTensor[\normalfont\textrm]\ricci{Ric}
\let\ric\ricci
\let\Ric\ricci
\NewTensor[{\outsymbol[\normalfont\textrm]}]\outricci{Ric}
\let\outric\outricci
\NewTensor[{\unisymbol[\normalfont\textrm]}]\uniricci{Ric}
\let\uniric\uniricci
\NewTensor[\newmathcal]\scalar[\!]S
\NewTensor[{\outsymbol[\newmathcal]}]\outscalar S
\let\outsc\outscalar
\NewTensor[{\unisymbol[\newmathcal]}]\uniscalar S
\let\unisc\uniscalar
\gdef\gausssymbol{K}
\gdef\gaussfont{\newmathcal}
\NewTensor[\gaussfont]\gauss<\!>\gausssymbol
\NewTensor[\textrm]\II k

\let\zFund\k
\NewTensor[\tracefree*]\IItrf[\!]{\textrm k}

\let\zFundtrf\ktrf
\NewTensor[{\outsymbol[\textrm]}]\outII[\hspace{-.05em}]k

\let\outzFund\outk
\NewTensor[\mathcal]\mc[\!]<\hspace{-.05em}> H
\let\H\mc
\NewTensor[{\outsymbol[\mathcal]}]\outmc[\!]<\hspace{-.05em}>H
\let\outH\outmc
\let\oldnu\nu
\NewTensor*\normal[\hspace{-.05em}]\oldnu
\let\nu\normal
\NewTensor[\outsymbol]\timenormal[\!]\vartheta
\let\tv\timenormal

\NewTensor*\laplace[\!]\Delta
\NewTensor[\outsymbol]*\outlaplace[\!]\Delta
\NewTensor\Hess[\!]{\text{Hess}}
\NewTensor\Hesstrf[\!]{\text H\tracefree{\text{es}}\text s}
\NewTensor[\outsymbol]\outHess[\!]{\text{Hess}}
\undef\div\NewTensor\div[\!]{\text{div}}
\NewTensor[\outsymbol]\outdiv[\!]{\text{div}}
\NewTensor*\tr[\!]{\text{tr}}
\NewTensor[\outsymbol]*\outtr{\text{tr}}
\NewTensor\levi<{\MakeSymbol<>{\!}}>{\MakeSymbol<\Gamma>\nabla}
\NewTensor[\outsymbol]\outlevi<{\MakeSymbol<>{\!}}>{\MakeSymbol<\Gamma>\nabla}
\NewTensor*\mug[\!]\mu
\NewTensor[\outsymbol]*\outmug[\!]\mu
\NewTensor[\outsymbol]\outmomden[\!]J
\NewTensor[\outsymbol]\outenden[\hspace{-.05em}]\varrho

\NewTensor[\textrm]*\jacobiext[\hspace{-.05em}]<\hspace{-.05em}> L
\NewDocumentCommand\trzd{smm}{{#2}\odot{#3}}
\NewDocumentCommand\trtr{smm}
{\def\testa{\string#2}\def\testb{\string#3}%
 \ifx\testa\testb\IfDisplaystyleTF{\left|}|#2\IfDisplaystyleTF{\right|}|_{\g*}^2\else\tr(\trzd{#2}{#3})\fi}

\NewDocumentCommand\outc{G{c}}{\IndexSymbol(\outsymbol)*{#1}}\let\oc\outc
\NewTensor\outve\ve
\gdef\outtr{\IndexSymbol(\outsymbol)[\!]{\text{tr}}}
\NewTensor*\eflap[\!]f
\NewTensor*\ewlap[\!]\lambda
\NewTensor*\efjac[\!]{\mathcal f}
\NewTensor*\ewjac[\!]<\!>\gamma
\NewTensor*\funcg[\!]g
\NewTensor*\funch[\!]h
\gdef\trans#1{{#1}^{\text t}}
\gdef\sphtrans#1{{#1}^{{}_{\sphg*}\text t}}
\gdef\deform#1{{#1}^{\text d}}

\gdef\Hradius{\sigma}
\gdef\Aradius{r}
\gdef\rradius{R}

\NewTensor\mHaw{{\normalfont m_{\text{H}}}}

\NewTensor[\MakeSymbol<\outsymbol>\vec]\impuls P
\NewTensor\impulsf[\hspace{-.15em}]{\trans{\bar P}}
\undef\d
\NewDocumentCommand\d{s}{\IfBooleanTF{#1}\relax{\mathop{}\!}\mathrm d}
\newcommand\pullback[1]{{#1}^*}
\newcommand\pushforward[1]{{#1}_*}

\DeclareMathOperator\id{id}

\DeclareMathOperator\lin{lin}
\DeclareMathOperator\graph{graph}

\newcommand\R{\mathds{R}}
\newcommand\N{\mathds{N}}
\newcommand\X{\mathfrak{X}}
\newcommand\Lp{{\normalfont\textrm L}}
\newcommand\Wkp{{\normalfont\textrm W}}
\newcommand\Hk{{\normalfont\textrm H}}
\newcommand\Ck{{\normalfont\textrm C}}

\newcommand\sphere{\mathds S}
\newcommand\ve{\varepsilon}

\usepackage{urwchancal}
\DeclareMathAlphabet{\mathcal}{OT1}{pzc}{m}{n}
\let\newmathcal\mathcal

\usepackage{auto-eqlabel}
 \setlist[description]{font=\normalfont\itshape\space}
\begin{document}
\title[Geometric charac.\ of asympt.\ flatness \& linear momentum]{Geometric characterizations of\\asymptotic flatness and linear momentum\\in general relativity}
 \author{Christopher Nerz}
 \address{Mathematisches Institut\\Eberhard Karls Universit\"at T\"ubingen\\Auf der Morgenstelle 10\\72076~T\"ubingen\\Germany\smallskip}
 \email{christopher.nerz@math.uni-tuebingen.de}
 \date\today
\begin{abstract}\noindent
In 1996, Huisken-Yau proved that every three-dimensional Riemannian manifold can be uniquely foliated near infinity by stable closed surfaces of constant mean curvature (CMC) if it is asymptotically equal to the (spatial) Schwarz\-schild solution. 
Later, their decay assumptions were weakened by Metzger, Huang, Eichmair-Metzger, and the author. In this work, we prove the reverse implication, i.\,e.\ any three-dimensional Riemannian manifold is asymptotically flat if it possesses a CMC-\emph{cover} satisfying certain geometric curvature estimates, a uniqueness property, a \emph{weak} foliation property, and each surface has weakly controlled instability. With the author's previous result that every asymptotically flat manifold possesses a CMC-foliation, we conclude that asymptotic flatness is characterized by existence of such a CMC-cover.
Additionally, we use this characterization to give a geometric (i.\,e.\ coordinate-free) definition of a \linebreak[1]\hbox{\emph{{\normalfont(\hspace{-.1em}}CMC-{\normalfont\hspace{.05em})}linear}} \emph{momentum} and prove its compatibility with the linear momentum defined by Arnowitt-Deser-Misner.
\end{abstract}\maketitle
\let\sc\scalar
\section{Introduction}
Surfaces of constant mean curvature (CMC) were for the first time used in mathematical general relativity by Christodoulou-Yau who studied quasi-local mass of asymptotically flat manifolds \cite{christodoulou71some}. In 1996, Huisken-Yau proved the existence of a unique foliation by stable CMC-surfaces \cite{huisken_yau_foliation}. They considered Riemannian manifolds $(\outM,\outg*,\outx)$ which are asymptotically equal to the (spatial) Schwarzschild solution, i.\,e.\ they assumed existence of a coordinate system $\outx:\outM\setminus\outsymbol L\to\R^3\setminus \overline{B_1(0)}$ mapping the manifold (outside of some compact set $\outsymbol L$) to the Euclidean space (outside the closed unit ball) such that the push forward $\outx_*\outg*$ of the metric $\outg*$ is asymptotically equal to the (spatial) Schwarzschild metric. More precisely, they assumed that the $k$\mbox-th derivatives of the difference $\outg_{ij}-\schwarzoutg_{ij}$ of the metric $\outg*$ and the Schwarzschild metric $\schwarzoutg*:=(1+\nicefrac{\mass}{2\vert x\vert})^4\,\eukoutg*$ decays in these coordinates like $\vert\outx\vert^{-2-k}$ for every $k\le4$, where the mass $\mass$ was assumed to be positive and $\eukoutg*$ denotes the Euclidean metric. This is abbreviate by $\outg*=\schwarzoutg*+\mathcal O_4(\vert x\vert^ {-2})$. Later, these decay assumptions were weakened by Metzger, Huang, Eichmair-Metzger, and the author \cite{metzger2007foliations,Lan_Hsuan_Huang__Foliations_by_Stable_Spheres_with_Constant_Mean_Curvature,metzger_eichmair_2012_unique,nerz2015CMCfoliation}: It is sufficient to assume asymptotic flatness to ensure existence of a CMC-foliation (and its uniqueness in a well-defined class of surfaces). Here, being asymptotically flat means $\outg*=\eukoutg*+\mathcal O_2(\vert x\vert^{{-}\frac12-\outve})$ with $\outsc=\mathcal O_0(\vert x\vert^{-3-\ve})$, where $\outsc$ the scalar curvature of $\outg*$. Further properties of this foliation were studied by Huisken-Yau, Corvino-Wu, Eichmair-Metzger, the author, and others \cite{huisken_yau_foliation,Corvino__On_the_center_of_mass_of_isolated_systems,metzger_eichmair_2012_unique,nerz2013timeevolutionofCMC,nerz2015CMCfoliation}.\pagebreak[2]\smallskip

Inspired by an idea to use the CMC-foliation to define a unique\footnote{up to Euclidean isometries} coordinate system $\outsymbol y:\outM\setminus\outsymbol K\to\interval{\Hradius_0}\infty\times\sphere^2$ which Huisken explained to the author, we prove that asymptotic flatness does not only \emph{imply} the existence and (local) uniqueness of a CMC-\emph{cover} $\{\M<\Hradius>\}_{\Hradius\ge\Hradius_0}$, but is \emph{characterized} by it. A simple version of the more general versions (s.~Section~\ref{section_main_theorem}) is the following:
\begin{corollary}[{Simple version of Theorem~\ref{Suff_ass_for_asymp_flat}}]
Let $(\outM,\outg*)$ be a three-dimensional Riemannian manifold without boundary and $\outve\in\interval0{\frac12}$ be a constant. There exists a coordinate system $\outx:\outM\setminus\outsymbol L\to\R^3\setminus\overline{B_1(0)}$ outside a compact set $\outsymbol L\subseteq\outM$ such that $(\outM,\outg*,\outx)$ is $\Ck^2_{\frac12+\outve}$-asymp\-to\-tic\-ally flat with strictly positive ADM-mass if and only if there are constants $\c>0$, $M>0$, $\Hradius_0\ge\Hradius_0'=\Cof{\Hradius_0'}[\outve][\c][M]$, and a family $\mathcal M:=\{\M<\Hradius>\}_{\Hradius>\Hradius_0'}$ of spheres in $\outM$ such that
\begin{enumerate}[nosep,label={\normalfont(\alph{enumi})}]
\item $\mathcal M$ is \emph{locally unique}, i.\,e.~if $\M'\hookrightarrow\outM$ is a CMC-surface and a graph on $\M\in\mathcal M$ which is $\Wkp^{2,p}$-close to $\M$, then $\M'\in\mathcal M$; \label{Simple_version__ass_fst}
\item $\outM\setminus\bigcup_\Hradius\M<\Hradius>$ is relatively compact;
\item $\M<\Hradius>$ is stable as surface of constant mean curvature $\H<\Hradius>\equiv\nicefrac{{-}2}\Hradius$ for every $\Hradius>\Hradius_0$;
\item the Hawking masses $\mHaw(\M<\Hradius>)$ of the elements $\M<\Hradius>\in\mathcal M$ are bounded away from zero and infinity, i.\,e.~$\vert\mHaw(\M<\Hradius>)\vert\in\interval{M^{-1}}M$ for every $\Hradius>\Hradius_0$;
\item $\sup_{\M<\Hradius>}(\Hradius^{\frac52+\outve}\,\vert\outric*\vert_{\outg*} + \Hradius^{3+\outve}\,\vert\outsc\vert) \le \c$ for every $\Hradius>\Hradius_0$; \label{Simple_version__ass_lst}
\item the elements of $\mathcal M$ are pairwise disjoint; \label{Simple_version__ass}
\end{enumerate}
In this setting, $\mathcal M$ is a smooth foliation and the ADM-mass $\mass$ of $(\outM,\outg*)$ satisfies $\mass=\lim_{\Hradius\to\infty}\mHaw(\M<\Hradius>)$.
\end{corollary}
To simplify the proof of this version of the theorem, we use the theorem by Bandue-Kasue-Nakajima: A manifold is asymptotically flat if it satisfies specific decay assumptions on the curvatures and a volume growth estimate, \cite[Thm~1.1]{BandueKasueNakajima1989}.\footnote{The authors thanks the referees of the \textsc{Journal of Functional Analysis} for bringing his attention to this article.} Note that the assumption that the elements of $\mathcal M$ are pairwise disjoint can be replaced by a more analytic one, see Section~\ref{Foliation_property}. As a side result, we get again a uniqueness result for the CMC-foliation, see Corollary~\ref{Uniqueness_of_CMC_foliation}.\smallskip

This theorem can be generalized to the setting, where the leaves of the CMC-cover are not stable but only have \emph{controlled instability}, i.\,e.~instead of assuming that the smallest eigenvalue of the stability operator is positive, we only have to assume that it is bounded from bellow by ${-}\Hradius^{{-}\frac{3-\outve}2}$.\footnote{In this setting, the ADM-mass can be negative.} We furthermore give a corresponding characterization of $\Wkp^{3,p}_{{-}\tau}$-asymptotic flatness (Theorem~\ref{Charac_of_Sobolev_asymp_flat}) for $p\in\interval2\infty$ and $\tau\ge\frac12$, i.\,e.\ for asymptotic flatness in a Sobolev sense as defined by Bartnik \cite{bartnik1986mass}. Note that in the Sobolev setting, we do not impose pointwise assumptions on the Ricci curvature as we allow the Sobolev exponent $p$ to be less than the dimension $n=3$. In particular, we can not apply Bandue-Kasue-Nakajima's theorem, \cite[Thm~1.1]{BandueKasueNakajima1989}. Additionally, we explain a local version of this characterization (Remark~\ref{ALocalVersion}). \smallskip

Furthermore, the above characterizations of asymptotic flatness can be used to define other quantities without using coordinates. Exemplary, we explain this for the linear momentum (Section~\ref{Charac_quan}): We define the CMC-linear momentum as a function $\impulsf$ on the initial data set $(\outM,\outg*,\outzFund*,\outmomden*,\outenden*)$ possessing a CMC-foliation. We prove that this function is well-defined outside of a compact set and that it characterizes the linear momentum calculated with respect to any asymptotically flat coordinate system as it is defined by \textbf Arnowitt-\textbf Deser-\textbf Misner \cite{arnowitt1961coordinate}.
This means that the ADM-linear momentum can be interpreted as a coordinate expression of a geometric quantity: the CMC-linear momentum.\smallskip

Note that Bandue-Kasue-Nakajima assume weaker (pointwise) decay assumptions on the curvatures -- in our notation $\ve\in\interval-{\frac12}{\frac12}$ instead of $\outve\in\interval0{\frac12}$ -- and a volume growth estimates instead of the existence of a CMC-foliation to conclude asymptotic flatness, \cite{BandueKasueNakajima1989}.\footnote{The authors thanks the referees of the \textsc{Journal of Functional Analysis} for bringing his attention to this article.} However differently to our approach, they cannot characterize mass, linear momentum, and other quantities directly geometrically. Furthermore, it is non-trivial that our assumptions imply their volume growth estimates and there is no Sobolev version of their theorem.

There is a third geometric characterization of asymptotically flat manifolds by Reiris for stationary solutions of the Einstein equations, more precisely he assumes that $(\outM,\outg*)$ corresponds to a stationary solution of the Einstein-equations and satisfies topological a~priori assumptions, \cite{Reiris2010stationarysolutionsI,Reiris2010stationarysolutionsII}.\medskip

As a technical step in the proof which seems interesting for itself, we prove in Appendix \ref{Regularity_Calabi_energy} that every metric $\g*$ on the two-dimensional sphere possesses a \lq good\rq\ conformal parametrization if it has a $\Lp^p$-almost constant Gau\ss\ curvature. This generalizes the well-known corresponding result for Gau\ss\ curvature pointwise bounded away from zero and infinity, see for example \cite[Chap.~2]{christodoulou1993global}. Here, we state a weaker $\Lp^2$-version of our new result.
\begin{corollary}[{Simple version of Theorem~\ref{Lp_Reg_gauss_curv}}]
There exist constants $C<\infty$ and $\ve>0$ with the following property: If a metric $\g*$ of the Euclidean unit sphere $\sphere^2$ satisfies $\Vert\confgauss-1\Vert_{\Lp^2(\M)} \le \ve$ and $\confmug(\sphere^2)\in \interval*\pi*{7\pi}$, then there exists a conformal parametrization $\varphi:\sphere^2\to\sphere^2$ with $\varphi^*\g*=\exp(2\,\conf)\,\sphg*$ and $\Vert\conf\Vert_{\Hk^2(\sphere^2,\sphg)} \le C\,\Vert\confgauss-1\Vert_{\Lp^2(\sphere^2,\g*)}$, where $\confgauss$ and $\confmug$ are the Gau\ss\ curvature and the measure on the sphere $\sphere^2$ with respect to $\g*$, respectively, and $\sphg*$ denotes the standard metric of $\sphere^2$.
\end{corollary} \bigskip

\textbf{Acknowledgment.}
The author wishes to express gratitude to Gerhard Huisken for suggesting the topic of CMC-surfaces in asymptotically flat manifolds and for many inspiring discussions. Further thanks are owed to Simon Brendle for suggesting the scaling argument in Appendix~\ref{Regularity_Calabi_energy}. Finally, thanks goes to Carla Cederbaum for exchanging interesting ideas and thoughts about CMC-foliations, to Mattias Dahl for bringing up the question whether it is necessary to \emph{a priori} assume the existence of a \emph{smooth foliation} rather than the one of a cover, and to the referees of the \textsc{Journal of Functional Analysis} for their helpful comments on the first draft of this article.\bigskip

\section*{Structure of the paper}
In Section~\ref{Ass_and_notation}, we explain basic notations and definitions -- these were also used in \cite{nerz2015CMCfoliation}. We give main regularity arguments in Section~\ref{Regularity_of_the_hypersurfaces}, where we use Appendix~\ref{Regularity_Calabi_energy}. The main result (and a local version of it) is stated and proven in Section~\ref{section_main_theorem}. In Section~\ref{Foliation_property}, we show some alternative assumptions for the main theorem. Furthermore, we give a coordinate-free definition of linear momentum and compare it with the linear momentum defined by Arnowitt-Deser-Misner (ADM) in Section~\ref{Charac_quan}. Finally, we prove existence of a \lq good\rq\ conformal parametrization for surfaces having $\Lp^p$-almost constant Gau\ss\ curvature in Appendix~\ref{Regularity_Calabi_energy}.\pagebreak[3]

\section{Assumptions and notation}\label{Ass_and_notation}
In order to study foliations (near infinity) of three-dimensional Riemannian manifolds by two-dimensional spheres, we will have to deal with different manifolds (of different or the same dimension) and different metrics on these manifolds, simultaneously. To distinguish between them, all three-dimensional quantities like the surrounding manifold $(\outM,\outg*)$, its Ricci and scalar curvature $\outric$ and $\outsc$, and all other derived quantities carry a bar, while all two-dimensional quantities like the CMC leaf $(\M,\g*)$, the trace free part $\zFundtrf*$ of its second fundamental form $\zFund*$, its scalar and mean curvature $\sc$ and $\H:=\tr\zFund$, its outer unit normal $\nu$, and all other derived quantities do not. Furthermore, we stress that the sign convention used for the second fundamental form results in a \emph{negative} mean curvature of the Euclidean uni sphere.

If different two-dimensional manifolds or metrics are involved, then the lower left index will always denote the mean curvature index $\Hradius$ of the current leaf $\M<\Hradius>$, i.\,e.\ the leaf with mean curvature $\H<\Hradius>\equiv\nicefrac{{-}2}\Hradius$. We abuse notation and suppress this index, whenever it is clear from the context which metric we refer to.
Furthermore, quantities carry the upper left index $\euclideane$ and $\sphg*$ if they are calculated with respect to the Euclidean metric $\eukoutg*$ and the standard metric $\sphg<\Hradius>$ of the Euclidean sphere $\sphere^2_\Hradius(0)$, correspondingly.

Here, we interpret the second fundamental form and the normal vector of a hypersurface as quantities on the hypersurfaces (and thus as two-dimensional). For example, if $\M<\Hradius>$ is a hypersurface in $\outM$, then $\nu<\Hradius>$ denotes its normal (and \emph{not} $\IndexSymbol{\outsymbol\nu}<\Hradius>$). The same is true for the \lq lapse function\rq\ and the \lq shift vector\rq\ of hypersurfaces arising as a leaf of a given deformation or foliation.\pagebreak[1]

Furthermore, we use upper case latin indices $\ii$, $\ij$, $\ik$, and $\il$ for the two-dimensional range $\lbrace2,3\rbrace$ and lower case latin indices $\oi$, $\oj$, and $\ok$ for the three-dimensional range $\lbrace 1,2,3\rbrace$. The Einstein summation convention is used accordingly.
\pagebreak[3]\medskip

Now, we give the main definition used within this work. Let us begin by recalling the Hawking mass \cite{hawking2003gravitational}.
\begin{definition}[Hawking mass]
Let $(\outM,\outg*)$ be a three-dimensional Riemannian manifold. For any closed hypersurface $\M\hookrightarrow(\outM,\outg*)$ the \emph{Hawking-mass} is defined by
\[ \mHaw(\M) := \sqrt{\frac{\volume{\M}}{16\pi}}(1-\frac1{16\pi}\int\H^2\d\mug), \]
where $\H$ and $\mug$ denote the mean curvature and measure induced on $\M$, respectively.\pagebreak[3]
\end{definition}
As there are different definitions of \lq asymptotically flat\rq, we now define the decay assumptions used in this paper.
\begin{definition}[\texorpdfstring{$\Ck^2_{\frac12+\outve}$}{C-2}-asymptotically flat Riemannian manifolds]\label{asymp_flat_RM}
Let $\outve\in\interval0{\nicefrac12}$ be a constant and let $(\outM,\outg*)$ be a smooth Riemannian manifold. The tuple $(\outM,\outg*,\outx)$ is called \emph{$\Ck^2_{\frac12+\outve}$-asymptotically flat Riemannian manifold} if $\outx:\outM\setminus\overline L\to\R^3\setminus\overline{B_1(0)}$ is a smooth chart of $\outM$ outside a compact set $\overline L\subseteq\outM$ such that
\begin{equation*}
 \vert\outg_{ij}-\eukoutg_{ij}\vert + \rad\,\vert\outlevi_{ij}^k\vert + \rad^2\,\vert\outric_{ij}\vert + \rad^{\frac52}\,\vert\outsc\vert \le \frac\oc{\rad^{\frac12+\outve}} \qquad\forall\,i,j,k\in\{1,2,3\} \labeleq{decay_g}
\end{equation*}
holds for some constant $\oc\ge0$, where $\eukoutg*$ denotes the Euclidean metric.
\end{definition}

As we will also use Sobolev version of asymptotic flatness, we recall the corresponding definitions by Bartnik, \cite{bartnik1986mass}.
\begin{definition}[$\Wkp^{k,p}_\eta$-asymptotically flat Riemannian manifolds {\cite{bartnik1986mass}}]
Let $(\outM,\outg*)$ be a Riemannian manifold and $\outx:\outM\setminus\outsymbol L\to\R^n\setminus\overline{B_1(0)}$ be a chart of $\outM$ outside some compact set $\outsymbol L\subseteq\outM$. If $(\pullback\outx\outg*-\eukoutg*)\in\Lp^\infty_{{-}\eta}(\outM)$, $(\pullback\outx\outg*-\eukoutg*)\in\Wkp^{k,p}_{{-}\eta}(\outM)$, and $\sc\in\Lp^1(\outM)$ for some $\eta>0$, $k\ge1$, and $p>n$, then $(\outM,\outg*,\outx)$ is called $\Wkp^{k,p}_{\eta}$-asymptotically flat. Here, the weighted Lebesgue norm of a function/tensor $\outsymbol T$ on $\outM$ (with respect to $\outx$) and a constant $\eta\in\R$ is defined by\footnote{We suppress the chart $\outx$ and the compact set $\outsymbol L$.}
\[ \Vert\outsymbol T\Vert_{\Lp^p_\eta(\outM)}
	 := \Vert\outsymbol  T\Vert_{\Wkp^{0,p}_\eta(\outM)}
	 := \begin{cases}\ 
		\displaystyle(\int_{\outM\setminus\outsymbol L} \rad^{{-}3}\,(\rad^{{-}\eta}\,\vert\outsymbol T\vert_{\outg*})^p \d\outmug)^{\frac1p}
			&:\quad p\in\interval*1\infty, \\\ 
		\displaystyle\mathop{\text{\normalfont ess\,sup}}_{\outM\setminus\outsymbol L}\vert \rad^{{-}\eta}\,\outsymbol T\vert_{\outg*}
			&:\quad p=\infty
	 \end{cases} \]
and correspondingly the $k^{\text{th}}$-Sobolev norm is defined by
\[ \Vert\outsymbol  T\Vert_{\Wkp^{k+1,p}_\eta(\outM)}
		:= \Vert\outsymbol  T\Vert_{\Lp^p_{\eta}(\outM)} + \Vert\outlevi*\outsymbol T\Vert_{\Wkp^{k,p}_{\eta-1}(\outM)}. \]
If $\Vert\outsymbol  T\Vert_{\Lp^p_\eta(\outM)}<\infty$\vspace{-.25em} and $\Vert\outsymbol  T\Vert_{\Wkp^{k,p}_\eta(\outM)}<\infty$, then $\outsymbol T$ is called element of $\Lp^p_\eta(\outM)$ and $\Wkp^{k,p}_\eta(\outM)$, respectively.
\end{definition}
Note that the usual Sobolev inequalities are satisfied for these Sobolev norms, \cite{bartnik1986mass}. In particular, every $\Wkp^{3,p}_\eta$-asymptotically flat manifold is $\Wkp^{2,q}_\eta$-asymptoti\-cally flat, where $q=\nicefrac{np}{(n-p)}>n$ if $p\in\interval{\nicefrac n2}n$.

\begin{definition}[ADM-mass]
Arnowitt-Deser-Misner defined the \emph{{\normalfont(\hspace{-.05em}}ADM-{\normalfont\hspace{.05em})}mass} of an asymptotically flat $n$-dimen\-sional Riemannian manifold $(\outM,\outg*,\outx)$ by
\begin{equation*} \mass_{\text{ADM}} := \lim_{\rradius\to\infty} \frac1{2(n-1)\,\omega_{n-1}}\sum_{\oj=1}^3 \int_{\sphere^2_\rradius(0)}(\partial[\outx]_\oj@{\outg_{\oi\oj}}-\partial[\outx]_\oi@{\outg_{\oj\oj}})\,\nu<\rradius>^\oi \d\mug<\rradius> \label{Definition_of_mass_ADM}, \end{equation*}
where $\omega_{n-1}$, $\nu<\rradius>$, and $\mug<\rradius>$ denote the Euclidean area of the $n{-}1$-dimensional Euclidean unit sphere, the outer unit normal, and the area measure of $\sphere^2_\rradius(0)\hookrightarrow(\outM,\outg*)$, respectively, \cite{arnowitt1961coordinate}. This mass is independent of the $\Wkp^{2,p}_{\frac12}$-asymptotically flat coordinate system if $p\in\interval n\infty$, \cite{bartnik1986mass}.
\end{definition}

In the literature, the ADM-mass is characterized using the Ricci curvature:
\begin{equation*} \mass := \lim_{\rradius\to\infty} \frac{{-}\rradius}{(n-1)(n-2)\,\omega_{n-1}} \int_{\sphere^2_\rradius(0)} \outric*(\nu<\rradius>,\nu<\rradius>) - \frac\outsc2 \d\mug<\rradius>,\pagebreak[3] \labeleq{Definition_of_mass} \end{equation*}
see Ashtekar-Hansen's, Chru\'sciel's, and Schoen's articles \cite{ashtekar1978unified,schoen1988existence,chrusciel1986remark}.\pagebreak[1]
Miao-Tam recently gave a proof of this characterization $\mass_{\text{ADM}}=\mass$ for any $\Ck^2_{\frac12+\outve}$-asymptotically flat manifold, \cite{miao2013evaluation}, and their proof can directly applied to $\Wkp^{3,p}_{\frac12}$-asymptotically flat manifolds, where $p>2$.\pagebreak[1] We recall that this mass is (in three-dimensions) also characterized by
 \begin{equation*} \mass = \lim_{\rradius\to\infty} \mHaw(\sphere^2_\rradius(0)) \labeleq{Definition_of_mass_dash}. \end{equation*}
This can be seen by a direct calculation using the Gau\ss\ equation, the Gau\ss-Codazzi equation, and the decay assumptions on metric and curvatures.\pagebreak[3]

We specify the definitions of Lebesgue and Sobolev norms on compact Riemannian manifolds which we will use throughout this article.
\begin{definition}[Lesbesgue and Sobolev norms]
If $(\M,\g*)$ is a compact Riemannian manifold without boundary, then the \emph{Lebesgue norms} are  defined by
\[ \Vert T\Vert_{\Lp^p(\M)} := (\int_{\M} \vert T\vert_{\g*}^p \d\mug)^{\frac1p}\quad\forall\,p\in\interval*1\infty, \qquad \Vert T\Vert_{\Lp^\infty(\M)} := \mathop{\text{ess\,sup}}\limits_{\M}\,\vert T\vert_{\g*}, \]
where\pagebreak[1] $T$ is any measurable function (or tensor field) on $\M$. Correspondingly, $\Lp^p(\M)$ is defined to be the set of all measurable functions (or tensor fields) on $\M$ for which the $\Lp^p$-norm is finite. If $\Aradius:=(\nicefrac{\volume{\M}}{\omega_n})^{\nicefrac1n}$ denotes the \emph{area radius} of $\M$, where $n$ is the dimension of $\M$ and $\omega_n$ denotes the Euclidean surface area of the $n$-dimensional unit sphere, then the \emph{Sobolev norms} are defined by
\[ \Vert T\Vert_{\Wkp^{k+1,p}(\M)} := \Vert T\Vert_{\Lp^p(\M)} + \Aradius\,\Vert\levi*T\Vert_{\Wkp^{k,p}(\M)}, \qquad \Vert T\Vert_{\Wkp^{0,p}(\M)} := \Vert T\Vert_{\Lp^p(\M)}, \]
where\pagebreak[1] $k\in\N_{\ge0}$, $p\in\interval*1*\infty$ and $T$ is any measurable function (or tensor field) on $\M$ for which the $k$-th (weak) derivative exists. Correspondingly, $\Wkp^{k,p}(\M)$ is the set of all such functions (or tensors fields) for which the $\Wkp^{k,p}(\M)$-norm is finite. Furthermore, $\Hk^k(\M)$ denotes $\Wkp^{k,2}(\M)$ for any $k\ge1$ and $\Hk(\M):=\Hk^1(\M)$.\pagebreak[3]
\end{definition}
Now, we can define the class of surfaces which we will use in the following.
\begin{definition}[Regular spheres]\label{Regular_sphere}
Let $(\outM,\outg*)$ be a three-dimensional Riemannian manifold. A hypersurface $\M\hookrightarrow(\outM,\outg*)$ is called \emph{regular sphere with area radius $\Aradius:=\sqrt{\nicefrac{\volume{\M}}{4\pi}}$ and constants $\kappa\in\interval1*2$, $M>0$, $p\in\interval2*\infty$ and $\c\ge0$}, in symbols $\M\in\regsphere<M>{\c}$, if $\M$ is a topological sphere satisfying $\vert\mHaw(\M)\vert\in\interval*{M^{-1}}*M$ and has constant mean curvature $\H\equiv\meanH$ such that
\begin{equation*}\labeleq{Regular_sphere_Ric}
	\Vert \vert\outric*\vert_{\outg*} \Vert_{\Lp^p(\M)} + \vert\meanH\vert^{\frac12} \Vert \outsc \Vert_{\Lp^p(\M)}  \le \c\,\vert\meanH\vert^{\kappa+1-\frac2p}, \end{equation*}
where $\nicefrac2p:=0$ if $p=\infty$ and where $\Hradius:=\nicefrac{{-}2}{\meanH}$ denotes the \emph{mean curvature radius of $\M$}. Furthermore, $\regsphere*<M>{\c}[p]$ denotes the radius independent family, i.\,e.\ $\regsphere*<M>{\c}[p]:=\bigcup_\Aradius\regsphere<M>(\Aradius){\c}[p]$.
\end{definition}
\begin{remark}[Almost constant mean curvature]
As we only need to study surfaces of constant mean curvature in this article, we prove everything just for these surfaces. However a straightforward adjustment of the proofs, prove all results from Subsection~\ref{Conformal_parametrization} for surfaces satisfying
$\Vert \H - \meanH \Vert_{\Wkp^{1,p}(\M)} \le \c_1\,\vert\meanH\vert^{\kappa-\frac2p}$
for some constant $\meanH$ instead of $\H\equiv\meanH$.\footnote{For example, this could be of interest when studying surfaces of \emph{constant expansion} $\H+\tr\,\outzFund$ or $\H-\tr\,\outzFund$, where $\outzFund$ is a given, asymptotically vanishing tensor of $\outM$ which is motivated by the second fundamental form of $\outM$ in a surrounding spacetime, see \cite{metzger2007foliations,nerz2015CEfoliation}. In this setting, we call $\Hradius:={-}2\meanH^{-1}$ \emph{approximated mean curvature radius of $\M$}.}
\end{remark}
\begin{remark}[Lorentzian version]
The main motivation of asymptotically flat manifolds are spacelike hypersurfaces in a four-dimensional Lorentzian manifold solving the Einstein equations. However, we use in Definition~\ref{Regular_sphere} (and in the rest of this work) only the Riemannian data of such a time-slice. Here, we give alternative assumptions on the surfaces using data of the surrounding four-dimensional Lorentzian manifold:\pagebreak[1]

If $(\outM,\outg*)$ is a spacelike hypersurface within a four-dimensional Lorentzian manifold $(\uniM,\unig*)$ solving the Einstein equations $8\pi\,\unisymbol T = \uniric* - \frac12\,\unisc\,\unig*$ for some energy-mo\-mentum tensor $\unisymbol T$, then the assumption on $\vert\outric\vert_{\outg*}$ in Definition~\ref{Regular_sphere} holds for a closed hypersurface $\M\hookrightarrow\outM$ and $p=\infty$ if
\[ \sup_{\M}\,\vert\unisymbol T\vert_{\outg*}\le \frac{\c_2}4\,\vert\meanH\vert^{\kappa+1}, \quad
		\sup_{\M}\,\vert\outzFund\vert_{\outg*}\le \frac{\c_2}4\,\vert\meanH\vert^{\frac{\kappa+1}2}, \quad
		\sup_{\M}\,\vert\lieD{\,\tv}\,\outzFund\vert_{\outg} \le \frac{\c_2}4\,\vert\meanH\vert^{\kappa+1}, \]
where $\meanH$ is as in Definition~\ref{Regular_sphere}, $\outzFund*$ denotes the second fundamental form of $(\outM,\outg*)\hookrightarrow(\uniM,\unig*)$, and $\lieD{\tv}\,\outzFund*$ denotes the Lie-derivative of $\outzFund$ in direction of a unit normal field $\tv$ on $\outM$.\footnote{More exactly, $\lieD{\,\tv\;}\outzFund*$ is the Lie-derivative of $\outzFund*'$, where $\outzFund*'(\uniexp_{\bar p}(\tau\,\tv_{\bar p}))$ is the second fundamental form of $\graph^{\tv}\tau:=\{\uniexp_{\bar p}(\tau\,\tv_{\bar p}) :\bar p\in\outM\}\hookrightarrow(\uniM,\unig*)$ in $\uniexp_{\bar p}(\tau\,\tv_{\bar p})$ which is well-defined in a $\uniM$-neighborhood of $\outM$. Here, $\uniexp_{\bar p}$ denotes the exponential map of $\uniM$ in a point $\bar p\in\outM$. In particular, this Lie-derivative is well-defined.
} We can equivalently rewrite this assumption using a given foliation of space-time $\uniM$ by spacelike hypersurfaces $\outM[t]$ or for $p\in\interval2\infty$.
\end{remark}
\begin{definition}[{$\alpha$-controlled instability of a regular sphere \cite{cederbaum2015communication}}]\label{controlled_instability}
Let $(\outM,\outg*)$ be a three-dimensional Riemannian manifold and $\alpha\in\R$. A regular sphere $\M\in\regsphere*<M>{\c}[p]$ has \emph{$\alpha$-controlled instability} if the smallest eigenvalue of the (negative) stability operator ${-}\jacobiext f={-}\laplace f - (\trtr\zFund\zFund+\outric(\nu,\nu))\,f$ on mean value free functions is larger than $\alpha$, i.\,e.
\begin{equation*} \int_{\M} \vert\levi f\vert^2 \d\mug \ge \int_{\M}(\outric(\nu,\nu)+\trtr\zFund\zFund+\alpha)(f-\fint_{\M} f\d\mug)^2 \d\mug \qquad\forall\,f\in\Ck^2(\M). \labeleq{stability} \end{equation*}
\end{definition}
\begin{remark}[Scaling]
Note that due to a rescaling argument, the assumption that a surface has ${-}\meanH^2$-controlled instability is a natural assumption. In particular, the assumption that a surface has ${-}\vert\H\vert^{{-}\frac{3-\outve}2}$ as in Theorems~\ref{Suff_ass_for_asymp_flat} and~\ref{Charac_of_Sobolev_asymp_flat} seems to be a very weak assumption. In particular, it is far weaker than to assume stability of the surfaces.
\end{remark}
\begin{remark}
A surface has $0$-con\-trolled instability if and only if the surface is (non-strictly) stable as CMC-surface and it has $\ve$-controlled instability (for some $\ve>0$) if and only if the surface is strictly stable. However, a surface with ${-}\alpha$-controlled instability is (not necessarily) stable. For example, the leaves $\M<\sigma>$ of the CMC-foliation of an asymptotically flat space with negative ADM-mass $\mass\relax<0$ are non-stable but each of them has $\frac{\mass}2\,|\meanH|^3$-controlled instability if its mean curvature radius $\Hradius:={-}\frac2{\H}$ is sufficiently large \cite{nerz2015CMCfoliation}. As this example proves, we cannot assume that the CMC-leaves are stable, but only that their instability is sufficiently controlled. Note that we assume a far weaker control of the instability (in Theorems~\ref{Suff_ass_for_asymp_flat} and~\ref{Charac_of_Sobolev_asymp_flat}) than we get a posterior.
\end{remark}
\begin{remark}[An alternative assumption]\label{an_alternative_assumption}
Due to our other assumptions, we can in the following replace the assumption that a surfaces has $\alpha$-controlled instability by the assumption that the Sobolev inequality on $\M$ holds for $\trtr\zFundtrf\zFundtrf$, i.\,e. if \label{ass_sobolev_for_k2}
\begin{equation*} \Vert\trtr\zFundtrf\zFundtrf\Vert_{\Lp^2(\M)} \le \c\,(\int\vert\H\vert\,\trtr\zFundtrf\zFundtrf+\vert\levi\trtr\zFundtrf\zFundtrf\vert\d\mug),\labeleq{sobolev_for_k2} \end{equation*}
where the constant $\c=\Cof[\c_1][\c_2]$ only depends on $\c_1$ and $\c_2$. It is also sufficient if \eqref{sobolev_for_k2} holds for $\max\{0,\trtr\zFundtrf\zFundtrf-\vert\meanH\vert^{2+\delta}\}$ instead of $\trtr\zFundtrf\zFundtrf$, where $\delta>0$ can be chosen arbitrary. By Lemma~\ref{Bootstrap_for_trace_free_second_fundamental_form__Lp2}, Hoffman-Spruck's result implies that \ref{ass_sobolev_for_k2} holds if $p=\infty$, the inequality on the Ricci curvature form \eqref{Regular_sphere_Ric} is satisfied in a neighborhood $U=\{\outp\in\outM : \text{dist}(p,\M)\le 1\}$ of $\M$, and the injectivity radius of $\outM$ (restricted to $\M$) is at least $\Hradius^{\frac{3-\kappa+\delta}2}$  for some fixed $\delta>0$ \cite{hoffman1974sobolev}.
\end{remark}

For notation convenience, we use the following abbreviated form for the contraction of two tensor fields.
\begin{definition}[Tensor contraction]
Let $(\M,\g*)$ be a Riemannian manifold. The \emph{traced tensor product} of a $(0,k)$ tensor field $S$ and a $(0,l)$ tensor field $T$ on $(\M,\g*)$ with $k,l>0$ is defined by
\[ (\trzd ST)_{I_1\dots I_{k-1}\!J_1\dots J_{l-1}} := S_{I_1\dots I_{k-1} K}\,\g^{KL}\,T_{LJ_1\dots J_{l-1}}. \]
This definition is independent of the chosen coordinates. Furthermore, $\trzd{\trzd ST}U$ is well-defined if $T$ is a $(0,k)$ tensor field with $k\ge2$, i.\,e.\ $\trzd{(\trzd ST)}U=\trzd S{(\trzd TU)}$ for such a $T$.\pagebreak[3]
\end{definition}

Finally, we infinitesimally characterize foliations in the following by their lapse functions and their shift vectors.
\begin{definition}[Lapse functions, shift vectors]
Let $\theta>0$ and $\Hradius_0\in\R$ be constants, $I\supseteq\interval{\Hradius_0-\theta}{\Hradius_0+\theta}$ be an interval, and $(\outM,\outg*)$ be a Riemannian manifold. A smooth map $\Phi:I\times\M\to\outM$ is called \emph{deformation} of the closed hypersurface $\M=:\M<\Hradius_0>=\Phi(\Hradius_0,\M)\subseteq\outM$ if $\Phi<\Hradius>(\cdot):=\Phi(\Hradius,\cdot)$ is a diffeomorphism onto its image $\M<\Hradius>:=\Phi<\Hradius>(\M)$ and $\Phi<\Hradius_0>\equiv\id_{\M}$. The decomposition of $\spartial*_\Hradius\Phi$ into its normal and tangential parts can be written as
\[ \partial[\Hradius]@\Phi = {\rnu<\Hradius>}\,{\nu<\Hradius>} + {\rbeta<\Hradius>}, \]
where $\nu<\Hradius>$ is the outer unit normal to $\M<\Hradius>$. The function $\rnu<\Hradius>:\M<\Hradius>\to\R$ is called \emph{lapse function} and the vector field $\rbeta<\Hradius>\in\X(\M<\Hradius>)$ is called \emph{shift} of $\Phi$. If $\Phi$ is a diffeomorphism onto its image, then it is called a (local) \emph{foliation}.\pagebreak[3]
\end{definition}

\section{Regularity of the hypersurfaces}\label{Regularity_of_the_hypersurfaces}
In this section, we prove the regularity results for the hypersurfaces used within this work. The author proved that regular spheres $\M<\Aradius>\in\regsphere[\kappa]{\c}[p]$ with $\kappa>\frac32$ and $p>2$ within a $\Ck^2_{\frac12+\outve}$-asymptotically flat Riemannian manifold (with $\outve>0$) are asymptotically pointwise umbilic (as $\Aradius\to\infty$) \cite{nerz2015CMCfoliation}. Using DeLellis-M\"uller's result \cite{DeLellisMueller_OptimalRigidityEstimates}, this implies existence of conformal parametrizations $\varphi<\Aradius>:\sphere^2\to\M<\Aradius>$ such that the corresponding conformal factor $\conformalf<\Aradius>\in\Wkp^{2,p}(\sphere^2)$ is asymptotically constant, i.\,e.\ $\conformalf<\Aradius>\to1$\linebreak[1] in $\Wkp^{2,p}(\sphere^2)$ for $\Aradius\to\infty$, where $p\in\interval*2\infty$ is arbitrary, $\varphi<\Aradius>^*\g<\Aradius>*=\Aradius^2\,e^{2\,\conformalf<\Aradius>}\,\sphg*$, and where $\sphg*$ denotes the standard metric of the Euclidean unit sphere. In the setting of a surrounding manifold $(\outM,\outg*,\outx)$ asymptotically equal to the (spatial) Schwarzschild solution, a similiar result was previously proven by Huisken-Yau, Metzger, and others \cite{huisken_yau_foliation,metzger2007foliations}. In Subsection~\ref{Conformal_parametrization}, we prove the same result for regular spheres with $|\H|^{\frac32}$-controlled instability in an arbitrary three-dimensional Riemannian manifolds. To do so, we have to replace the crucial tool in the above argument, DeLellis-M\"uller's result, by the arguments in Appendix~\ref{Regularity_Calabi_energy}. There, we prove that metric $\g<n>*$ on the Euclidean sphere converge (after conformal reparametrization) to the standard metric of the Euclidean sphere if the Gau\ss\ curvatures $\gauss<n>$ converges in $\Lp^p(\sphere^2,\sphg*)$ to $1$. Note that the same result is well-known if the Gau\ss\ curvatures are \emph{pointwise} bounded away from zero and infinity, see for example \cite{christodoulou1993global}. In Subsection~\ref{Stability_operator_eigenvalues}, we then cite results and arguments from \cite{nerz2015CMCfoliation} proving that the Eigenvalues of the stability operator are (asymptotically) controlled.

\subsection{Conformal parametrization and umbilicness}\label{Conformal_parametrization}
We start by proving that any regular sphere $\M$ with sufficiently controlled instability has a \lq good\rq\ conformal parametrization, i.\,e.\ the corresponding conformal factor is almost constant. To do so, we make the following three steps:
\begin{description}
\item[Lemma~\ref{Bootstrap_for_trace_free_second_fundamental_form__Lp2}] prove that the sphere is in a $\Lp^2$-sense almost umbilic and that the area of $\M$ is (approximately) $4\pi\Hradius^2$;
\item[Lemma~\ref{Bootstrap_for_trace_free_second_fundamental_form__Lp4}] prove that the sphere is in a $\Lp^4$-sense almost umbilic and therefore the scalar curvature is in a $\Lp^2$-sense asymptotically constant;
\item[Proposition~\ref{Regularity_of_the_spheres}] prove that the surfaces is in a $\Wkp^{1,p}$-sense almost umbilic and that a \lq good\rq\ parametrization exists -- this uses Theorem~\ref{Lp_Reg_gauss_curv} from Appendix~\ref{Regularity_Calabi_energy} and Corollary~\ref{Bootstrap_for_trace_free_second_fundamental_form_cited} (beeing {\cite[Prop~2.1]{nerz2015CMCfoliation}}).\pagebreak[3]
\end{description}
Let us first look at the last step and rewrite the cited proposition \cite[Prop~2.1]{nerz2015CMCfoliation} in the notation used within this article.
\begin{corollary}[$\Lp^\infty$-estimates on the second fundamental form]\label{Bootstrap_for_trace_free_second_fundamental_form_cited}
Let $(\M,\g*)\in\regsphere<M>{\c}$ be a hypersurface with mean curvature radius $\Hradius$ of a three-dimensional Riemannian manifold $(\outM,\outg*)$, where $\kappa\in\interval1*2$, $p\in\interval2*\infty$, $\c\ge0$, and $M>0$ are constants. Assume there exists a finite Sobolev constant $\cSob\ge0$, i.\,e.
\begin{equation*} \Vert f\Vert_{\Lp^2(\M)} \le \frac\cSob\Aradius\Vert f\Vert_{\Wkp^{1,1}(\M)}  \qquad\forall\,f\in\Ck^1(\M), \labeleq{Sobolev_inequality} \end{equation*}
where $\Aradius$ denotes the area radius, i.\,e.~$4\pi\,\Aradius^2:=\volume{\M}$. There are constants $\Hradius_0=\Cof{\Hradius_0}[\kappa][p][\c][M][\cSob]$ and $C=\Cof[\kappa][p][\c][M][\cSob]$ such that $\Hradius\ge\Hradius_0$, $\vert\Hradius-\Aradius\vert\le \nicefrac{\Hradius^{\frac{3-\kappa}2}}C$, and $\Vert\zFundtrf\Vert_{\Lp^2(\M)}\le\nicefrac 1C$ implies
\begin{equation*}
 \Vert\hspace{.05em}\zFundtrf\Vert_{\Lp^\infty(\M)} \le \frac C{\Hradius^\kappa}, \qquad
 \Vert\hspace{.05em}\zFundtrf*\Vert_{\Hk(\M)}\le \frac C{\Hradius^{\kappa-1}} \labeleq{Bootstrap_for_trace_free_second_fundamental_form_cited__k} \end{equation*}
\end{corollary}
Now, let us begin by a simple $\Lp^2$-estimate for the second fundamental form and prove that all preliminaries of Corollary~\ref{Bootstrap_for_trace_free_second_fundamental_form_cited} except the Sobolev inequality \eqref{Sobolev_inequality} are satisfied if $\Hradius$ is sufficiently large, too.
\begin{lemma}[\texorpdfstring{$\Lp^2$}{Lp-2}-estimates for the second fundamental form]\label{Bootstrap_for_trace_free_second_fundamental_form__Lp2}
Let $(\M,\g*)\in\regsphere<M>{\c}$ be a hypersurface of a three-dimensional Riemannian manifold $(\outM,\outg*)$, where $\kappa\in\interval1*2$, $p\in\interval2*\infty$, $\c\ge0$, and $M>0$ are constants. There are two constants $\Hradius_0=\Cof{\Hradius_0}[\kappa][p][\c][M]$ and $C=\Cof[\kappa][p][\c][M]$ such that if $\Hradius\ge\Hradius_0$, then
\begin{equation*} \vert\Hradius-\Aradius\vert \le C, \qquad\Vert\zFundtrf\Vert_{\Lp^2(\M)} \le \frac C{\Hradius^{\frac{\kappa-1}2}} \labeleq{Bootstrap_for_trace_free_second_fundamental_form__Lp2_eq}. \end{equation*}
\end{lemma}
\begin{proof}
With the assumption on the Hawking mass, we see that
\begin{equation*}
 \vert1-\frac{r^2}{\Hradius^2}\vert
	= \vert 1 - \frac1{16\pi}\int\H^2\d\mug\vert
	= \frac{2\,\vert\mHaw\vert}r \in \interval{\frac 2{r\,M}}{\frac{2\,M}r}
\end{equation*}
which implies $\vert\Aradius-\Hradius\vert<C$. In particular, we get
$\vert\int_{\M}\H^2 \d\mug - 16\pi\vert \le C$
and conclude the claim by the Gau\ss\ equation and the Gau\ss-Bonnet theorem.
\end{proof}
Now, we strengthen this to $\Lp^4$-estimates for the second fundamental form. This generalizes the results proven by Metzger in \cite[Prop.~3.3]{metzger2007foliations} and by the author in \cite[Prop.~2.1]{nerz2015CMCfoliation} where a corresponding result was proven under the assumption of asymptotic flatness of the surrounding manifold (in order to use DeLellis-M\"uller's result \cite{DeLellisMueller_OptimalRigidityEstimates}). Here, we have to assume that the sphere has controlled instability.
\begin{lemma}[\texorpdfstring{$\Lp^4$}{Lp-4}-estimates for the second fundamental form]\label{Bootstrap_for_trace_free_second_fundamental_form__Lp4}
Let $(\M,\g*)\in\regsphere<M>{\c_1}$ be a regular hypersurface with ${-}\c_2\,\vert\meanH\vert^\beta$-controlled instability of a three-dimensional Riemannian manifold $(\outM,\outg*)$, where $\kappa\in\interval1*2$, $\beta\in\interval1*\kappa$, $p\in\interval2*\infty$, $\c_1\ge0$, $\c_2\in\R$, $M>0$ are constants. There are constants $\Hradius_0=\Cof{\Hradius_0}[\kappa][p][\c_1][\c_2][M]$ and $C=\Cof[\kappa][p][\c_1][\c_2][M]$ such that
\begin{equation*} \labeleq{Bootstrap_for_trace_free_second_fundamental_form__Lp4_eq} \Vert\zFundtrf\Vert_{\Lp^2(\M)}^2 + \Hradius^2\,\Vert\zFundtrf\Vert_{\Lp^4(\M)}^4 + \Hradius^2\,\Vert\levi*\zFundtrf\Vert_{\Lp^2(\M)}^2 \le \frac C{\Hradius^{\kappa+\beta-1}} \end{equation*}
if $\Hradius>\Hradius_0$.\pagebreak[2]
\end{lemma}
\begin{proof}
Equivalent to \cite[Prop.~3.3]{metzger2007foliations} and \cite[Prop.~2.1]{nerz2015CMCfoliation}, we first integrate $\trtr{\laplace\zFundtrf}{\zFundtrf}$ and then integrate it by parts. However, instead of using the Simon's identity for $\laplace\zFund$, we use the following (equivalent but simpler) formula
\[ \levi*\div\zFund = \div_2(\levi*\zFund) - \trzd\ric\zFund + \rc_{\cdot\ik\il\cdot}\,\zFund^{\ik\il}. \]
which can be proven equivalent to the Simon's identity using normal coordinates and the Codazzi equation.\footnote{Actually, this is true for every hypersurface $\M\hookrightarrow(\outM,\outg*)$ in any Riemannian manifolds $(\outM,\outg*)$.}
The Codazzi equation implies
\begin{equation*}
 \levi*\div\zFund = \laplace\zFund - \div_1(\outrc_{\cdot\cdot\cdot\nu}) - \trzd\ric\zFund + \rc_{\cdot\ii\ij\cdot}\,\zFund^{\ii\ij}. \labeleq{alternative_Simons_identity}
\end{equation*}
In particular using $\dim\M=2$ and \eqref{alternative_Simons_identity}, a integration by parts proves
\begin{align*}
 \Vert\levi*\zFundtrf\Vert_{\Lp^2(\M)}^2
 ={}& \int\trzd{\div\zFund}{\div\zFundtrf}\d\mug - \int\outrc(\levi\zFundtrf,\nu)\d\mug \\
		&	- \frac12\int_{\M}\sc(\trtr\zFundtrf\zFundtrf + (\g_{\ii\ik}\g_{\ij\il}-\g_{\ii\il}\g_{\ij\ik})\,\zFund^{\ij\ik}\,\zFund^{\ii\il})\d\mug \\
 ={}& \frac12\Vert\levi*\H\Vert_{\Lp^2(\M)}^2 - 2\int\outric(\nu,\levi*\H)\d\mug + 2\,\Vert\outric_\nu\Vert_{\Lp^2(\M)}^2
			- \int_{\M}\sc\trtr\zFundtrf\zFundtrf\d\mug \\
 ={}& \frac12\Vert\levi*\H\Vert_{\Lp^2(\M)}^2 - 2\int\outric(\nu,\levi*\H)\d\mug + 2\,\Vert\outric_\nu\Vert_{\Lp^2(\M)}^2 \\
		& - \int_{\M} (\outsc-2\,\outric(\nu,\nu) + \frac{\H^2}2 - \trtr\zFundtrf\zFundtrf)\trtr\zFundtrf\zFundtrf\d\mug,
\end{align*}
where we used the Gau\ss\ equation in the last step. Now, we use $\meanH\equiv\frac{{-}2}\Hradius$, the decay assumptions for $\outric$, and \eqref{Bootstrap_for_trace_free_second_fundamental_form__Lp2_eq} to conclude
\begin{equation*}\labeleq{Regularity_of_the_spheres_L4_and_H2}
 \vert\Vert\levi*\zFundtrf\Vert_{\Lp^2(\M)}^2 + \frac2{\Hradius^2}\Vert\zFundtrf\Vert_{\Lp^2(\M)}^2 - \Vert\zFundtrf\Vert_{\Lp^4(\M)}^4\vert
	 \le \frac C{\Hradius^{2\kappa}} + \frac C{\Hradius^{\kappa}}\Vert\zFundtrf\Vert_{\Lp^4(\M)}^2.\smallskip \end{equation*}
	
Without loss of generality, we assume $\c_3\ge0$ and we define $\alpha:=\c_2\,\vert\meanH\vert^\beta$. The rest of the proof is analogue to Huisken-Yau's proof of \cite[Prop.~5.3]{huisken_yau_foliation} which again is based on a paper by Schoen-Simon-Yau \cite{schoen1975curvature}, where we replace their assumption that the surface is stable by our weaker assumption that the sphere has controlled instability. We recall the proof nevertheless for the readers convenience. Due to the assumed ${-}\alpha$-controlled instability assumption, we have (choosing $f:=\vert\zFundtrf\vert_{\g*}$ and $\mean f:=\fint f\d\mug=\volume{\M}^{-1}\,\Vert\zFundtrf\Vert_{\Lp^1(\M)}$)
\begin{align*}
 \Vert\zFundtrf\Vert_{\Lp^4(\M)}^4
 \le{}& \int(\trtr\zFundtrf\zFundtrf - \alpha + \outric(\nu,\nu))(f-\mean f)^2 \d\mug
				+ \Vert\outric\Vert_{\Lp^2(\M)}\,\Vert(f-\mean f)\Vert_{\Lp^4(\M)}^2 \\
			& + \alpha \Vert f-\mean f\Vert_{\Lp^2(\M)}^2
				+ \mean f^2\,\Vert\zFundtrf\Vert_{\Lp^2(\M)}^2
				+ 2\,\mean f\,\Vert\zFundtrf\Vert_{\Lp^3(\M)}^3 \\
 \le{}& \Vert\levi*\vert\zFundtrf\vert_{\g*}\Vert_{\Lp^2(\M)}^2
				+ \frac14\Vert\zFundtrf\Vert_{\Lp^4(\M)}^4
				+ \frac C{\Hradius^{2\,\kappa}}
				+ \frac C{\Hradius^\kappa}\,\mean f^2 \\
			&	+ \frac C{\Hradius^\beta}\,\Vert\zFundtrf\Vert_{\Lp^2(\M)}^2
				+ C\,\Hradius^{2-\beta}\,\mean f^2
				+ \mean f^2\,\Vert\zFundtrf\Vert_{\Lp^2(\M)}^2
				+ C\,\mean f^4.
\end{align*}
As we know $\mean f:=\fint\vert\zFundtrf\vert_{\g*}\d\mug\le \frac C\Hradius\,\Vert\zFundtrf\Vert_{\Lp^2(\M)}$, Lemma~\ref{Bootstrap_for_trace_free_second_fundamental_form__Lp2} implies
\begin{equation*}\labeleq{Regularity_of_the_spheres_L4_and_H2_2}
 \Vert\zFundtrf\Vert_{\Lp^4(\M)}^4
 \le \frac43\,\Vert\levi*\vert\zFundtrf\vert_{\g*}\Vert_{\Lp^2(\M)}^2
		+ \frac C{\Hradius^{\kappa+\beta-1}}.
\end{equation*}
Using \cite[(1.28)]{schoen1975curvature}, we see
\begin{equation*}\labeleq{Regularity_of_the_spheres_Schoen-Simon-Yau} \frac{16}{9}\,\trtr{\levi*\vert\zFundtrf\vert_{\outg*}}{\levi*\vert\zFundtrf\vert_{\outg*}} \le \trtr{\levi*\zFundtrf}{\levi*\zFundtrf} + C\,\vert\outric\vert_{\outg*}^2, \end{equation*}
where we note that Schoen-Simon-Yau did not use their minimality condition to prove the above inequality.\footnote{In fact, they prove this inequality by a brilliant algebraic argument in a suitable chosen chart.} Combining \eqref{Regularity_of_the_spheres_L4_and_H2}, \eqref{Regularity_of_the_spheres_L4_and_H2_2}, and \eqref{Regularity_of_the_spheres_Schoen-Simon-Yau}, we conclude
\[ \Vert\levi*\zFundtrf\Vert_{\Lp^2(\M)}^2 + \frac2{\Hradius^2}\,\Vert\zFundtrf\Vert_{\Lp^2(\M)}^2
		\le \frac C{\Hradius^{2\kappa}} + \frac54\Vert\zFundtrf\Vert_{\Lp^4(\M)}^4
		\le \frac{15}{16}\,\Vert\levi*\zFundtrf\Vert_{\Lp^2(\M)}^2 + \frac C{\Hradius^{\kappa+\beta-1}}. \]
Therefore, \eqref{Regularity_of_the_spheres_L4_and_H2} implies the claim.
\end{proof}

\begin{proof}[Proof of Remark~\ref{an_alternative_assumption}]
Assume that \eqref{sobolev_for_k2} holds (instead of the control of the instability). We prove \eqref{Regularity_of_the_spheres_L4_and_H2} as in the above proof of Lemma~\ref{Bootstrap_for_trace_free_second_fundamental_form__Lp4}. Now, \eqref{sobolev_for_k2} implies
\[ \Vert\zFundtrf\Vert_{\Lp^4(\M)}^2
		= \Vert\trtr\zFundtrf\zFundtrf\Vert_{\Lp^2(\M)}
		\le \frac C\Hradius\,\Vert\trtr\zFundtrf\zFundtrf\Vert_{\Wkp^{1,1}(\M)}
		\le \frac C\Hradius\,\Vert\zFundtrf\Vert_{\H(\M)}\,\Vert\zFundtrf\Vert_{\Lp^2(\M)}, \]
where we used $\vert\levi\trtr\zFundtrf\zFundtrf\vert=2\,\vert\zFundtrf\vert\,\vert\levi*\zFundtrf\vert$ $\mug$-almost everywhere. Therefore, Lemma~\ref{Bootstrap_for_trace_free_second_fundamental_form__Lp2} and~\eqref{Regularity_of_the_spheres_L4_and_H2} imply
\[ \Vert\levi*\zFundtrf\Vert_{\Lp^2(\M)}^2 + \frac2{\Hradius^2}\Vert\zFundtrf\Vert_{\Lp^2(\M)}^2
		\le \frac C{\Hradius^{2\kappa}} + \Vert\zFundtrf\Vert_{\Lp^4(\M)}^4
		\le \frac C{\Hradius^{2\kappa}} + \frac C{\Hradius^{\kappa+1}}\,\Vert\zFundtrf\Vert_{\Hk(\M)}^2. \]
As we assumed $\kappa>1$, we have proven \eqref{Bootstrap_for_trace_free_second_fundamental_form__Lp4_eq} in this setting, too.
\end{proof}
\begin{proposition}[Regularity of the spheres]\label{Regularity_of_the_spheres}
Let $q<\infty$ be a constant and $(\M,\g*)\in\regsphere<M>{\c_1}$ be a hypersurface with ${-}\c_2\,\vert\meanH\vert^\beta$-controlled instability within a three-dimensional Riemannian manifold $(\outM,\outg*)$, where $\kappa\in\interval{\frac32}*2$, $\beta>3-\kappa$, $p\in\interval2\infty$, $\c_1\ge0$, $\c_2\in\R$, and $M>0$ are constants. There exist two constants $\Hradius_0=\Cof{\Hradius_0}[\kappa][\beta][p][\c_1][\c_2][M]$ and $C=\Cof[\kappa][\beta][p][\c_1][\c_2][M][q]$ and a conformal parametrization $\varphi:\sphere^2\to\M$ with corresponding conformal factor $\conformalf\in\Hk^2(\sphere^2)$, i.\,e.\ $\varphi^*\g*=\exp(2\,\conformalf)\,\Hradius^2\,\sphg*$, such that
\begin{equation*}
 \Vert\conformalf\Vert_{\Wkp^{2,p}(\sphere^2,\Hradius^2\,\sphg*)} \le \frac C{\Hradius^{\kappa+1-\frac2p}}, \qquad
 \Vert\zFundtrf\Vert_{\Wkp^{1,q}(\M)} \le \frac C{\Hradius^{\kappa-\frac2q}} \labeleq{Regularity_of_the_spheres__k}
\end{equation*}
if $\Hradius>\Hradius_0$, where $\sphg*$ denotes the standard metric of the Euclidean unit sphere.\pagebreak[2]
\end{proposition}
\begin{proof}
Without loss of generality $\beta\le\kappa$. Let $\psi:\sphere^2\to\M$ be an arbitrary smooth parametrization and define $\h:=\Hradius^{{-}2}\;\psi^*\g$. By the Gau\ss\ equation, we know
\begin{align*}
 \Vert\sc[h]-2\Vert_{\Lp^2(\M)}^2
 ={}& \int(\sc[\h]-2)^2\d\mug[\h]
 =		\Hradius^2\int(\sc-\frac{\H^2}2)^2\d\mug \\
 ={}&	\Hradius^2\int(\outsc - 2\,\outric(\nu,\nu) - \trtr\zFundtrf\zFundtrf)^2\d\mug
 \le  \frac C{\Hradius^{\kappa+\beta-3}} \xrightarrow{\Hradius\to\infty} 0,
\end{align*}
where we used $\H\equiv\frac{{-}2}\Hradius$ in the first step and the assumptions on $\outric$ and the Lemmata~\ref{Bootstrap_for_trace_free_second_fundamental_form__Lp2} and~\ref{Bootstrap_for_trace_free_second_fundamental_form__Lp4} in the last one. By Theorem~\ref{Lp_Reg_gauss_curv}, we can replace $\psi$ by a conformal parametrization such that $\h=\exp(2\,\conformalf)\,\sphg*$ with $\Vert\conformalf\Vert_{\H^2(\sphere^2)}\le C\,\Hradius^{\frac{\kappa+\beta-3}2}\ll1$, where we assumed $\Hradius$ to be sufficiently large. In particular, the Sobolev inequality holds for a Sobolev constant independent of $\Hradius$. Thus, we can use Corollary~\ref{Bootstrap_for_trace_free_second_fundamental_form_cited} and Lemma~\ref{Bootstrap_for_trace_free_second_fundamental_form__Lp2} to conclude \eqref{Bootstrap_for_trace_free_second_fundamental_form_cited__k}. Therefore, repeating the above calculation gives
\[ \Vert\sc[h]-2\Vert_{\Lp^p(\sphere^2)} \le \frac C{\Hradius^{\kappa+1-\frac2p}}. \]
Hence, we can again use Theorem~\ref{Lp_Reg_gauss_curv} to conclude that there exists a conformal parametrization such that its conformal factor $\conformalf$ satisfies the first inequality in~\eqref{Regularity_of_the_spheres__k}. Recalling the Simon's identity
\begin{align*}\labeleq{Simons-identity}
 \laplace\zFund*
	={}& \Hess\,\H - \levi*\outric_{\nu} + \div_2\outrc_{\cdot\cdot\cdot\nu} + \frac{\H^2}2\zFundtrf + \H\,\trzd\zFundtrf\zFundtrf - \trtr\zFundtrf\zFundtrf\,\zFund* \\
		& - \trzd{(\tr_{23}\outrc)}\zFundtrf + \outrc_{\cdot I\!J\cdot}\,\zFundtrf^{I\!J} 
\end{align*}
from \cite{simons1968minimal,schoen1975curvature}, we can use the regularity of the weak Laplace operator, see the results of Christodoulou-Klainerman \cite[Cor.~2.3.1.2]{christodoulou1993global} or the combination of the results of Giaquinta-Martinazzi \cite[Thm~7.1]{giaquinta2005introduction} and Adams-Fournier \cite[Thm~3.9]{adams2003sobolev} to conclude the second inequality in \eqref{Regularity_of_the_spheres__k}.\pagebreak[3]
\end{proof}

\subsection{The stability operator}\label{Stability_operator_eigenvalues}
Now, we recall results from \cite{nerz2015CMCfoliation}, which we can use in this setting. As first step, we note that the eigenvalues of the stability operator of $\M<\Hradius>$ are of order $\Hradius^{-2}$ except for three eigenvalues of order $\Hradius^{-3}$. As we will see in Proposition~\ref{Stability}, the corresponding partition of $\Hk^2(\M)$ (respectively $\Lp^2(\M)$) is (asymptotically) given as follows.
\begin{definition}[Translational, deformational, and rescaling part]
Let $(\M,\g*)\in\regsphere<M>{\c_1}$ be a regular sphere with ${-}\c_2\,\vert\meanH\vert^\beta$-controlled instability of a three-dimensional Riemannian manifold $(\outM,\outg*)$, where $\kappa\in\interval1*2$, $\beta>3-\kappa$, $p\in\interval2*\infty$, $\c_1\ge0$, $\c_2\in\R$, and $M>0$ are constants. The \emph{translational part} $\trans f$ of a function $f\in\Lp^2(\M)$ is the $\Lp^2(\M)$-orthogonal projection of $f$ on the linear span of eigenfunctions of the (negative) Laplace with eigenvalue $\lambda$ satisfying $\vert\lambda-\nicefrac2{\Hradius^2}\vert\le\nicefrac1{\Hradius^2}$, i.\,e.
\begin{equation*}\labeleq{definition_trans_f} \trans f := \sum_{\vert\ewlap_i-\nicefrac2{\Hradius^2}\vert\le\nicefrac1{\Hradius^2}} \eflap_i\,\int_{\M} f\,\eflap_i \d\mug \qquad\forall\,f\in\Lp^2(\M), \end{equation*}
where $\{\eflap_i\}_{i\in\N}$ is any complete orthogonal system of $\Lp^2(\M)$ consisting of eigenfunctions of the (negative) Laplace operator with corresponding eigenvalue $\ewlap_i$, i.\,e.\ ${-}\laplace\eflap_i=\ewlap_i\,\eflap_i$ with $\ewlap_i\le\ewlap_{i+1}$. The \emph{rescaling part} $\mean f$ and \emph{deformational part} $\deform f$ of such a function $f\in\Lp^2(\M)$ is defined by $\mean f:=\volume{\M}^{-1}\int f\d\mug$ and $\deform f := f-\trans f - \mean f$, respectively.
\end{definition}
The author explained in \cite[Prop.~4.5]{nerz2015CMCfoliation} reasons for calling this terms \emph{translational}, \emph{rescaling}, and \emph{deformational part}. Note that $\trans{\Lp^2(\M)}:=\{\trans f\;{:}\;f\in\Lp^2(\M)\}$ is three-dimensional due to Proposition~\ref{Regularity_of_the_spheres}.\pagebreak[3]

Now, we can cite the announced stability proposition which is one of the central tools for the proof of the main theorem.
\begin{proposition}[Stability (equivalent to {\cite[Prop~2.7]{nerz2015CMCfoliation}})]\label{Stability}
Let $q\in\interval2\infty$ be a constant and $(\M,\g*)\in\regsphere<M>{\c_1}$ be a hypersurface with constant mean curvature and ${-}\c_2\,\vert\meanH\vert^\beta$-controlled instability in a three-dimen\-sio\-nal Riemannian manifold $(\outM,\outg*)$, where $\kappa\in\interval1*2$, $\beta>3-\kappa$, $p\in\interval2*\infty$, $M>0$, $\c_1$, and $\c_2$ are constants. There are two constants $\Hradius_0=\Cof{\Hradius_0}[\kappa][\beta][p][\c_1][\c_2][M]$ and $C=\Cof[\kappa][p][\beta][\c_1][\c_2][M][q]$ such that
\begin{alignat*}4
 \vert \int_{\M} \jacobiext*\trans g\,\trans h\d\mug - \frac{6\mHaw}{\Hradius^3}\int_{\M}\trans g\,\trans h\d\mug \vert \le{}& \frac C{\Hradius^{3+\ve}}\Vert\trans g\Vert_{\Lp^2(\M)}\,\Vert\trans h\Vert_{\Lp^2(\M)} &\quad&\forall\,g,h\in\Lp^2(\M), \\
 \Hradius\,\vert\mean g\vert + \Vert \deform g\Vert_{\Lp^2(\M)} \le{}& \frac{3\,\Hradius^2}2 \Vert \jacobiext*(\mean g+\deform g)\Vert_{\Lp^2(\M)} &&\forall\,g\in\Hk^2(\M)
\end{alignat*}
if $\Hradius>\Hradius_0$. If $\efjac\in\Hk^2(\M)$ is a eigenfunction of ${-}\jacobiext*$ with corresponding eigenvalue $\ewjac$ and $\Hradius>\Hradius_0$, then
\[ \vert\,\ewjac\vert \ge \frac3{2\Hradius^2} \quad\ \text{or}\quad\ \Hradius\,\vert\,\mean{\efjac}\vert+\Vert\,\deform{\efjac}\Vert_{\Hk^2(\M)} \le \frac C{\Hradius^{\frac12+\outve}}\Vert\,\efjac\Vert_{\Hk^2(\M)},\ \; \vert\,\ewjac-\frac{6\mHaw}{\Hradius^3}\vert \le \frac C{\Hradius^{3+\outve}}. \]
Furthermore, the corresponding $\Wkp^{3,q}$-inequalities
\begin{equation*}
 \Vert\trans g\Vert_{\Wkp^{3,p}(\M)} \le (\frac{\Hradius^3}{6\,\mHaw}+C\,\Hradius^{3-\outve})\,\Vert\jacobiext*g\Vert_{\Lp^p(\M)}, \ \;
   \Vert\Hesstrf\,\trans g\Vert_{\Wkp^{1,p}(\M)} \le C\,\Hradius^{\frac12-\outve}\,\Vert\jacobiext*g\Vert_{\Lp^p(\M)}
	\labeleq{Stability_eq}
\end{equation*}
and $\Wkp^{2,q}$-inequalities
\[ \Hradius^{\frac2q}\,\vert\mean g\vert + \Vert\deform g\Vert_{\Wkp^{2,q}(\M)} \le C_q\,\Hradius^2\,\Vert\jacobiext*g\Vert_{\Lp^q(\M)} \qquad\forall\,q\in\interval1*p \]
hold for every function $g\in\Wkp^{2,q}(\M)$ if $\Hradius>\Hradius_0$.\pagebreak[3]
\end{proposition}
Note that the inequalities of this proposition are proven using the conformal parametrization from Proposition~\ref{Regularity_of_the_spheres}, but nevertheless these inequalities are geometric ones, i.\,e.~they do not depend on a parametrization of $\M$.
\begin{proof}
Replacing \cite[Prop.~2.4]{nerz2015CMCfoliation} by Proposition~\ref{Regularity_of_the_spheres}, we mimic the proof of \cite[Prop~2.7]{nerz2015CMCfoliation} (resp.~\cite[Lemma~2.5]{nerz2015CMCfoliation}) and this proves everything except~\eqref{Stability_eq} but including $\Vert\trans g\Vert_{\Wkp^{2,p}(\M)} \le (\nicefrac{\Hradius^3}{(6\,\mHaw)}+C\,\Hradius^{3-\outve})\,\Vert\jacobiext*g\Vert_{\Lp^p(\M)}$.

Now, we recall the commutator identity $\laplace\levi*\trans g-\levi*\laplace\trans g=\nicefrac\sc2\;\levi*\trans g$ and see that the regularity of the Laplace operator implies
\[ \Vert\trans g\Vert_{\Wkp^{3,p}(\M)} \le C\,\Vert\trans g\Vert_{\Wkp^{2,p}(\M)} \le (\frac{\Hradius^3}{6\,\mHaw}+C\,\Hradius^{3-\outve})\,\Vert\jacobiext*g\Vert_{\Lp^p(\M)}. \]
Therefore with respect to the conformal parametrization $\varphi$ from Proposition~\ref{Regularity_of_the_spheres}, we have
\[ \Vert\laplace[\sphg*]\,(\trans g\circ\varphi) - \frac2{\Hradius^2}(\trans g\circ\varphi)\Vert_{\Wkp^{1,p}(\sphere^2)}
		\le \frac C{\Hradius^{\frac52+\outve}}\Vert\trans g\Vert_{\Lp^p(\M)} \]
Denoting the $\Lp^2(\sphere^2,\sphg*)$-orthogonal projection of $\varphi^*\trans g:=\trans g\circ\varphi$ on the eigenfunctions of the $\sphg*$-Laplace operator with eigenvalue $\nicefrac2{\Hradius^2}$ by $(\varphi^*\trans g)\!^{{}^{\sphg*}t}$, this implies
\[ \Hradius^2\Vert\Hesstrf[\sphg*\,]\,(\varphi^*\trans g)\Vert_{\Wkp^{1,p}(\sphere^2,\sphg*)} \le \Vert \varphi^*\trans g - (\varphi^*\trans g)\!^{{}^{\sphg*}t} \Vert_{\Wkp^{3,p}(\sphere^2,\sphg*)} \le \frac C{\Hradius^{\frac12+\outve}}\Vert\trans g\Vert_{\Lp^p(\M)}, \]
where we used $\Hesstrf[\sphg*\,]\,h\!^{{}^{\sphg*}t}\equiv0$ for any $h\in\Lp^2(\sphere^2,\sphg*)$. Thus, we finally get
\[ \Vert\Hesstrf\,\trans g\Vert_{\Wkp^{1,p}(\M)} \le \frac C{\Hradius^{\frac52+\outve}}\,\Vert\trans g\Vert_{\Lp^p(\M)}. \qedhere \]
%
\end{proof}%
In particular, we get a strong control of the instability for regular spheres with controlled instability.
\begin{corollary}[Bootstrap for the control of the instability]
If $\M$ satisfies the assumption of Prop.~\ref{Stability} with sufficiently large $\Hradius$, then $\M$ has $(\nicefrac{6\,\mHaw}{\Hradius^3}+\nicefrac C{\Hradius^{3+\ve}})$-controlled instability. In particular, it is strictly stable if the Hawking mass is bounded away from zero (independently from $\Hradius$), if $\mHaw\ge M^{-1}>0$.
\end{corollary}\pagebreak[3]

\section{The main theorem}\label{section_main_theorem}
In this section, we state and prove the main theorem (in multiple versions). As the proof is quite long, we divide this section in multiple subsection. In subsection~\ref{TheMainTheorem}, we state the theorem first and explain afterwords the definitions used -- note that we give alternative assumption for the main theorem in Section~\ref{Foliation_property}. The main technical lemmata are stated and proven in Subsection~\ref{TechnicalLemmata}. Finally, we give the proof of all versions of the main theorem in Subsection~\ref{ProofOfTheMainTheorem}.
\subsection{The main theorem}\label{TheMainTheorem}
First, let us state the characterization of \emph{pointwise} asymptotic flatness.
\begin{theorem}[Characterization of asymptotic flatness]\label{Suff_ass_for_asymp_flat}
Let $(\outM,\outg*)$ be a three-dimensional Riemannian manifold without boundary and $\outve\in\interval0{\frac12}$ be a constant. There exists a coordinate system $\outx:\outM\setminus\outsymbol L\to\R^3\setminus\overline{B_1(0)}$ outside a compact set $\outsymbol L\subseteq\outM$ such that $(\outM,\outg*,\outx)$ is $\Ck^2_{\frac12+\outve}$-asymp\-to\-tic\-ally flat with ADM-mass $\mass\neq0$ if and only if there are constants $\c>0$, $\Hradius_0\ge\Hradius_0'=\Cof{\Hradius_0'}[\outve][\c][\vert\mass\vert]$ and a family $\mathcal M:=\{\M<\Hradius>\}_{\Hradius>\Hradius_0'}\subseteq\regsphere*[\frac32+\ve]<\c\,\vert\mass\vert>\c[\infty]$ of regular sphere in $\outM$ such that
\begin{enumerate}[nosep,label={\normalfont(\alph{enumi})}]
\item $\mathcal M$ is \emph{locally unique} {\normalfont(}\emph{locally complete}{\normalfont)}; \label{thm_ass__fst}
\item $\outM\setminus\bigcup_\Hradius\M<\Hradius>$ is relatively compact;
\item $\M<\Hradius>$ has mean curvature radius $\Hradius$, i.\,e.\ $\M<\Hradius>$ has constant mean curvature $\H<\Hradius>\equiv\nicefrac{{-}2}\Hradius$;
\item $\M<\Hradius>$ has ${-}\c\,\Hradius^{{-}\frac32}$-controlled instability;\label{thm_ass__lst}
\item the elements of $\mathcal M$ are pairwise disjoint. \label{thm_ass__disjoint}
\end{enumerate}
In this setting, $\mHaw(\M<\Hradius>)\to\mass$ and $\mathcal M$ is a smooth foliation.
\end{theorem}
\begin{remark}[The Hawking masses]
Note that we do \emph{not} assume that the Hawking masses of the CMC-surfaces converge, but only that they are bounded away from zero and infinity.
\end{remark}
\begin{remark}[A local version of the alternative version]\label{ALocalVersion}
Note that we can replace the cover $\{\M<\Hradius>\}_{\Hradius>\Hradius_0}$ of the manifold $\outM$ by a cover $\{\M<\Hradius>\}_{\Hradius\in\interval{\Hradius_0}{\Hradius_1}}$ for a finite $\Hradius_1\in\interval{\Hradius_0}\infty$ and get the corresponding local result (for both theorems), i.\,e.~a chart of $\bigcup_{\Hradius\in\interval{\Hradius'}{\Hradius_1}}\M<\Hradius>$, where $\Hradius':=\max\{\Hradius_0,\Hradius_0'\}$ and $\Hradius_0'=\Cof{\Hradius_0'}[\c][M][\outve][p]$ does neither depend on $\Hradius_0$ nor on $\Hradius_1$.
\end{remark}
For the proof of this theorem, we use the following theorem by Bandue-Kasue-Nakajima that a manifold satisfying specific decay assumptions on the curvatures and a volume growth estimate is already asymptotically flat, \cite[Thm~1.1]{BandueKasueNakajima1989}.\footnote{The authors thanks the referees of the \textsc{Journal of Functional Analysis} for bringing his attention to this article.} Let us recall this theorem.
\begin{theorem}[{\cite[Thm~1.1]{BandueKasueNakajima1989}} for \texorpdfstring{$\dim\outM{=}3$}{dimension three}]\label{BandueKasueNakajima1989}
Let $o\in M$ be a point within a three-dimensional Riemannian manifold $(\outM,\outg*)$ and $\oc>0$ and $\outve\in\interval02$ be constants with $\outve\neq1$. If $(\outM,\outg*)$ has only one end and
\[ \vert\outric\vert_{\outg*}\le\frac{\oc}{\textrm d(o,\,{\cdot}\,)^{2+\outve}}, \qquad
		\outsymbol\mu(B_r(o)) \ge \frac{r^n}{\oc}, \]
then $\outM$ is $\Ck^2_{\outve}$-asymptotically flat.
\end{theorem}
\begin{remark}[Pairwise disjoint]
In Section~\ref{Foliation_property}, we give possible assumptions with any of which we can replace Assumption~\ref{thm_ass__disjoint}, i.\,e.~three assumption such that each of them implies that the surfaces are pairwise disjoint.
\end{remark}
Furthermore, we give an alternative version of Theorem \ref{Suff_ass_for_asymp_flat} in the notation of weighted Sobolev spaces. For this theorem (and its proof), we assume that the reader is familiar with Bartnik's article \cite{bartnik1986mass}. In particular, with the Sobolev inequalities and regularity of the Laplace operator in weighted Sobolev spaces, \cite[Sect.~1]{bartnik1986mass}, and the existence and regularity of harmonic coordinates \cite[Sect.~3]{bartnik1986mass}.
\begin{subequations}\label{assumptions}
\begin{theorem}[Characterization of \texorpdfstring{$\Wkp^{3,p}_{\nicefrac12}$}{Sobolev}-asymptotic flatness]\label{Charac_of_Sobolev_asymp_flat}
Let $(\outM,\outg*)$ be a three-dimensional Riemannian manifold without boundary and $\outsc\in\Lp^1(\outM)$, $p\in\interval2\infty$, $\eta\ge\nicefrac12$ with $\eta\notin\N$, and $k\in\N_{\ge0}$ be constants. There exists a coordinate system $\outx:\outM\setminus\outsymbol L\to\R^3\setminus\overline{B_1(0)}$ outside a compact set $\outsymbol L\subseteq\outM$ such that $(\pushforward\outx\outg*-\eukoutg)\in\Wkp^{3+k,p}_{{-}\eta}(\R^3\setminus\overline{B_1(0)})$, $\rad^3\,\vert\outsc\vert\to0$ as $\rad\to\infty$ and $(\outM,\outg*,\outx)$ has non-vanishing ADM-mass if and only if there are constants $\c>0$, $M>0$, and $\Hradius_0>\Hradius_0'=\Cof{\Hradius_0'}[\outve][\c][\mass][M]$ and a family $\mathcal M:=\{\M<\Hradius>\}_{\Hradius>\Hradius_0'}$ of spherical hypersurfaces in $\outM$ such that the assumptions~\ref{thm_ass__fst}--\ref{thm_ass__disjoint} are satisfied for $\mathcal M$ and
\begin{align}
 \Vert\vert\outric*\vert_{\outg*}\Vert_{\Lp^p(\M<\Hradius>)}
	\le{}& \frac{\outc(\Hradius)}{\Hradius^{\frac52-\frac2p}}, &
 \Hradius^{{-}3}\,\sum_{l=0}^{k+1}\int_{\M<\Hradius>}(\Hradius^{2+\eta+l}\vert\outlevi*^l\outric*\vert_{\outg*})^p\d\mug
	\le{}& \outc'(\Hradius), \labeleq*{Ass_on_the_ricci_curv} \\
 \Vert\outsc\Vert_{\Lp^\infty(\M<\Hradius>)}
	\le{}& \frac{\outc(\Hradius)}{\Hradius^3}, &
 \vert\mHaw(\M<\Hradius>)\vert \in \interval{M^{-1}}M \hphantom{\ \le{}}&
		\labeleq*{Ass_on_the_scalar_curv}
\end{align}
is satisfied for each $\Hradius>\Hradius_0$, where $\outc'\in\Lp^1(\R)$ and $\outc(\Hradius)=\int_\Hradius^\infty\outc'(s)\d s\to0$ for $\Hradius\to\infty$ and where $\outlevi*\outric*$ denote the three-dimensional covariant derivative of the Ricci curvature of $\outM$. In this setting, the ADM-mass $\mass$ satisfies $\mass=\lim_{\Hradius\to\infty}\mHaw(\M<\Hradius>)$.
\end{theorem}
\begin{remark}[Alternative characterization]
In Section~\ref{Foliation_property}, we will see that we can equivalent characterize wether a Riemannian manifold $(\outM,\outg*)$ is $\Wkp^{3+k,p}_\eta$-asymptot\-ically flat with $\outsc\in\Wkp^{1+l,1}_3(\R^3\setminus\overline{B_1(0)})$ and $l\in\N_{\ge0}$ by the existence of a cover $\mathcal M$ such that the assumptions~\ref{thm_ass__fst}--\ref{thm_ass__lst} and~\eqref{assumptions} are satisfied and
\begin{equation*}
 \Vert\outsc\Vert_{\Lp^\infty(\M<\Hradius>)}
	\le \frac{\outc(\Hradius)}{\Hradius^3}, \qquad
 \sum_{i=0}^{l+1} \Hradius^i\,\Vert\outlevi*^i\hspace{.05em}\outsc\Vert_{\Lp^1(\M<\Hradius>)}
	\le \outc'(\Hradius). \labeleq*{Ass_on_the_scalar_curv__strong}
\end{equation*}
Note that we do \emph{not} assume \ref{thm_ass__disjoint}, i.\,e.~that the surfaces are pairwise disjoint, and that \eqref{Ass_on_the_scalar_curv__strong} is a natural assumption as \eqref{Ass_on_the_ricci_curv} is the corresponding inequality for the Ricci curvature (if $l=k$).
\end{remark}
\end{subequations}
%
%
%
%
Now, we explain the terminology used above. 
\begin{definition}[Locally uniqueness (completeness)]\label{locally_unique_CMC-family}
Let $\mathcal M\subseteq\regsphere*<M>\c$ be a family of regular spheres with constant mean curvature, where $\kappa\in\interval1*2$, $p\in\interval2*\infty$, $\c\ge0$, and $M>0$ are constants. The family $\mathcal M$ is called \emph{locally unique} (or \emph{locally complete}) if there exists an arbitrary constant $\delta=\Cof{\delta}[\M][p]\relax>0$ for each leaf $\M\in\mathcal M$ such that the implication
\[ \Vert f\Vert_{\Wkp^{2,p}(\M)} \le \delta,\ \H(\graph f)\equiv\H'\in\interval{\H^{-}}{\H^{+}}\qquad\Longrightarrow\qquad \graph f\in\mathcal M \]
holds for every function $f\in\Hk^2(\M)$, where $\H(\graph f)$ denotes the mean curvature of $\graph f:=\{\outexp_p(f(p)\,\nu)\;{:}\;p\in\M\}$ and $\H^{\pm}:={\pm}\sup\{{\pm}\H(\M):\M\in\mathcal M\}$.
\end{definition}
We want to stress that we do \emph{not} assume that any leaf of a CMC-family is a graph of some other leaf of the family. \smallskip\pagebreak[3]

\subsection{Technical lemmata}\label{TechnicalLemmata}
\begin{lemma}[Smoothness of the cover and the lapse function]\label{Lemma__smooth_cover}
Assume that
\begin{enumerate}[series=lemma_assumptions,nosep,label={\normalfont(\Alph{enumi})}]
\item $(\outM,\outg*)$ is a three-dimensional Riemannian manifold without boundary;\label{Ass_outM}\label{Ass_fst}
\item $M>0$, $\outve\in\interval0{\nicefrac12}$, $p\in\interval2*\infty$, $0<\Hradius_0<\Hradius_1\le\infty$, and $\c>0$  are constants;\label{Ass_const}
\item $\mathcal M:=\{\M<\Hradius>\}_{\Hradius\in\interval{\Hradius_0}{\Hradius_1}}\subseteq\regsphere*[\frac32+\ve]<\c\,M>\c$ is a family of regular spheres in $\outM$;\vspace{-.5em}\label{Ass_hypersurfaces}
\item $\M<\Hradius>$ has constant mean curvature $\H<\Hradius>\equiv\nicefrac{{-}2}\Hradius$ and ${-}\c\,\Hradius^{{-}\frac32}$-con\-trolled instability for every $\Hradius\in\interval{\Hradius_0}{\Hradius_1}$.\label{Ass_cmc}
\item $\mathcal M$ is as family of CMC-surfaces locally unique;\label{Ass_unique}\label{Ass_lst}
\end{enumerate}
There exist two constants $\Hradius_0'=\Cof{\Hradius_0'}[\outve][M][p][\c]$ and $C=\Cof[\outve][M][p][\c]$ such that if $\Hradius_0>\Hradius_0'$, then $\mathcal M$ is a $\mathcal C^1$-cover of its image, i.\,e.\ there exists a $\Ck^1$-map $\outPhi:\interval{\Hradius_0}{\Hradius_1}\times\sphere^2\to\outM$ such that $\outPhi(\Hradius,\sphere^2)=\M<\Hradius>$. In this setting, the lapse function $\rnu:=\outg*(\nu,\partial*_\Hradius\outPhi)$ of any such $\Ck^1$-cover satisfies
\begin{subequations}\label{Lemma__smooth_cover_eq}%
\begin{align}
 \Vert \deform{\rnu}\Vert_{\Wkp^{2,p}(\M)} \le{}& \frac C{\Hradius^{\frac12+\outve-\frac2p}}, &
 \vert\meanrnu-1\vert \le{}& \frac C{\Hradius^{\frac12+\outve}}, \\
 \Vert \trans{\rnu} \Vert_{\Wkp^{3,p}(\M)} \le{}& C\,\Hradius^{\frac12+\frac2p-\ve}, &
 \Vert\Hesstrf\,\trans\rnu\Vert_{\Wkp^{1,p}(\M)}
	\le{}& \frac C{\Hradius^{\frac52+\outve}}\Vert\trans\rnu\Vert_{\Lp^p(\M)}.
\end{align}%
\end{subequations}%
If the elements of $\mathcal M$ are pairwise disjoint, then $0<\rnu\relax<3$.
\end{lemma}
It would be sufficient to assume that the family contains $\{\M<\Hradius>\}_{\Hradius\in I}$, where $I$ is a dense subset of $\interval{\Hradius_0}{\Hradius_1}$. However, this is a technical assumption and does not need any additional step in the proof as the uniqueness condition (a posteriori) implies that $I\supseteq\interval{\Hradius_0'}{\Hradius_1}$ for some $\Hradius_0'\ge\Hradius_0$.
\begin{proof}[Proof of Lemma \ref{Lemma__smooth_cover}]
We can assume that $\Hradius_0$ is so large that we can use the Propositions~\ref{Regularity_of_the_spheres} and \ref{Stability} for each spheres $\M<\Hradius>\in\mathcal M$. Fix a sphere $\M<\Hradius>=\M$ and suppress the corresponding index $\Hradius$. We know that the stability operator 
\[ \jacobiext*:\Wkp^{2,q}(\M)\to\Lp^q(\M):f\mapsto\laplace f + (\trtr\zFund\zFund+\outric*(\nu,\nu))f \]
is the Fr\'echet derivative of the mean curvature map
\[ \textsf H : \Wkp^{2,q}(\M) \to \Lp^q(\M) : f \mapsto \H(\graph f) \]
at $f=0$ (for every $q>2$), where $\H(\graph f)$ denotes the mean curvature of the graph of $f$ which we interpret as function on $\M$.\pagebreak[1] By Proposition~\ref{Stability}, the stability operator is invertible. Thus, the inverse function theorem implies that $\textsf H$ is bijective from a $\Wkp^{2,q}(\M)$-neigh\-bor\-hood of $0\in\Wkp^{2,q}(\M)$ to a $\Lp^q(\M)$-neighborhood of $\H\in\Lp^q(\M)$. In particular, there is a $\eta_0>0$ and a curve $\gamma:\interval{\Hradius-\eta_0}{\Hradius+\eta_0}\to\Wkp^{2,q}(\M)$ such that $\textsf H(\gamma(\H+\eta))\equiv\H+\eta$ for any $\vert\eta\vert<\eta_0$. By the uniqueness condition, this means that $\graph\gamma(\eta)=\M<\H+\eta>$ for any $\vert\eta\vert<\eta_0$. In particular, every leaf $\M<\Hradius>$ (with sufficiently large $\Hradius$) is a graph of every other leaf $\M<\Hradius'>$ with small enough $\vert\Hradius'-\Hradius\vert$. Furthermore, this implies that the existence of a constant $\Hradius_0'\ge\Hradius_0$ and a $\Ck^1$-map
$\Phi : \interval{\Hradius_0'}{\Hradius_1} \times\sphere^2 \to \outM$ such that $\Phi(\Hradius,\sphere^2)=\M<\Hradius>$ for any $\Hradius\in\interval{\Hradius_0'}{\Hradius_1}$.

Per Definition of $\Phi$ and $\M<\Hradius>$, we know
\begin{equation*} \jacobiext*\rnu = \partial[\Hradius]@{(\H<\Hradius>)} \equiv \frac2{\Hradius^2} \equiv \frac{\H^2}2 = \trtr\zFund\zFund - \trtr\zFundtrf\zFundtrf. \labeleq{lapse_func_charaction} \end{equation*}
Thus, Proposition~\ref{Regularity_of_the_spheres} implies
\[ \vert\hspace{.05em}\jacobiext*(\rnu-1) + \outric*(\nu,\nu)\vert \le \frac C{\Hradius^{3+\ve}}. \]
Hence, Proposition~\ref{Stability} proves the claim.
\end{proof}

\begin{lemma}[\texorpdfstring{$\Hradius$}{sigma}-derivatives of the metric, the Hawking mass, and \texorpdfstring{$\eflap_i$}{the eigenfunctions of the Laplace operator}]\label{lemma__derivatives}
If the assumptions~\ref{Ass_fst}--\ref{Ass_lst} are satisfied, then there are constants $\Hradius_0'=\Cof{\Hradius'_0}[M][\ve][p][\c]$ and $C=\Cof[M][\ve][p][\c]$ such that $\Hradius_0>\Hradius_0'$ implies
\begin{equation*}
 \vert\partial[\Hradius]@{(\mHaw<\Hradius>)}-2(1+\frac\mHaw\Hradius-\meanrnu)\vert
		\le \frac C{\Hradius^{\frac32+\outve}}, \quad
	 \vert\volume{\M} - 4\pi\Hradius^2(1-\frac{2\,\mHaw}\Hradius) \vert
		\le C\,\Hradius^{\frac12-\outve} \labeleq*{lemma__derivatives__dm}
\end{equation*}
for the Hawking mass $\mHaw<\Hradius>:=\mHaw(\M<\Hradius>)$. In this setting, each $\Ck^1$-map
\begin{equation*}
 \Phi:\interval{\Hradius_0}{\Hradius_1}\times\sphere^2\to\outM:(\Hradius,p)\mapsto\Phi(\Hradius,p)
		\quad\text{with}\quad\partial*_\Hradius\Phi = (\rnu<\Hradius>\,\nu<\Hradius> + \Hradius\,\levi*\trans\rnu)\labeleq{lemma__derivatives__fol_map_char}
\end{equation*}
satisfies
\begin{subequations}\label{lemma__derivatives__d}
\begin{align*}
 \Vert\partial[\Hradius]@{\,\g<\Hradius>*} - \frac2\Hradius\,\g* \Vert_{\Wkp^{1,p}(\M)}
	\le{}& \frac C{\Hradius^{\frac32+\outve}}\Vert\trans\rnu\Vert_{\Lp^p(\M)} + \frac C{\Hradius^{\frac32+\outve-\frac2p}}, \labeleq*{lemma__derivatives__dg}\\
 \Vert\partial[\Hradius]@{\,\zFundtrf<\Hradius>*}\Vert_{\Lp^p(\M)}
	\le{}& \frac C{\Hradius^{\frac52+\outve}}\Vert\trans\rnu\Vert_{\Lp^p(\M)}
		+ \frac C{\Hradius^{\frac52+\outve-\frac2p}}, \labeleq*{lemma__derivatives__dk}\\
 \Vert\partial[\Hradius]@{\,\sc<\Hradius>} + \frac4{\Hradius^3}\Vert_{\Wkp^{{-}1,q}(\M)}
	\le{}& \frac C{\Hradius^{\frac72+\outve}}\Vert\trans\rnu\Vert_{\Lp^p(\M)}
		+ \frac C{\Hradius^{\frac72+\outve-\frac2p}}, \labeleq*{lemma__derivatives__dsc}
\end{align*}
where all quantities of $\M<\Hradius>$ are identified with their pullback along $\Phi<\Hradius>:=\Phi(\Hradius,\cdot)$ and where $\nicefrac1p+\nicefrac1q=1$. In this setting, there exist $\Lp^2(\M<\Hradius>)$-orthogonal functions $\eflap<\Hradius>_i\in\Wkp^{2,p}(\M<\Hradius>)$ with 
\[ \eflap<\Hradius>_i \in \lin\,\lbrace f\in\Lp^2(\M<\Hradius>) \ :\  \laplace<\Hradius>\,f = {-}\lambda\,f,\ \vert\lambda-\frac2{\Hradius^2}\vert\le\frac 1{\Hradius^2} \rbrace,\quad
	\Vert\eflap<\Hradius>_i\Vert_{\Lp^\infty(\M<\Hradius>)} = 1 \]
such that
\begin{equation} \Vert\partial[\Hradius]@{(\eflap<\Hradius>_i)} \Vert_{\Wkp^{2,p}(\M)} \le \frac C{\Hradius^{\frac12+\outve}}\Vert\trans\rnu\Vert_{\Lp^p(\M)}
		+ \frac C{\Hradius^{\frac12+\outve-\frac2p}}. \end{equation}
\end{subequations}
\end{lemma}
\begin{proof}
Per definition of $\Phi$, we know
\[ \partial[\Hradius]@{\g<\Hradius>*} = {-}2\rnu\,\zFund* + \Hradius\,\Hess\,\trans\rnu = (\frac{2\,\rnu}\Hradius + \Hradius\,\laplace\trans\rnu\!)\,\g* - 2\rnu\,\zFundtrf* + \Hradius\,\Hesstrf\,\trans\rnu \]
and
\[ \partial[\Hradius]@{\,\zFundtrf*} = \rnu(\outric*-\frac\sc2\,\g*-2\trzd\zFund\zFund+\H\,\zFund)+\Hess\,\rnu + \Hradius\,\g*(\levi*\trans\rnu,\levi*\zFund*). \]
Therefore our inequalities on $\outric*$, $\rnu$, $\zFundtrf*$, and $\sc$ imply \eqref{lemma__derivatives__dg} and \eqref{lemma__derivatives__dk}. Due integration by parts (in coordinates of a fixed $\M<\Hradius>$), we have
\begin{align*}
 \int \sc<\Hradius>\,f \d\mug<\Hradius>
	={}& \int {-}\levi<\Hradius>_{\!J\!I}^K\,\div<\Hradius>(\g<\Hradius>^{I\!J}\,f\,e_K) + \levi<\Hradius>_{\!KI}^K\,\div<\Hradius>(\g<\Hradius>^{I\!J}\,f\,e_J) \d\mug<\Hradius> \\
		& + \int_{\M<\Hradius>} \g<\Hradius>^{I\!J}(\levi<\Hradius>_{\!K\!L}^L\levi<\Hradius>_{\!J\!I}^L - \levi<\Hradius>_{\!J\!L}^K\levi<\Hradius>_{\!K\!I}^L)\,f \d\mug<\Hradius>.
\end{align*}
Using the inequalities for the derivative of $\partial*_\Hradius\g*$, i.\,e.\ the ones for $\zFundtrf*$ and $\rnu$, we conclude
\[ \vert\partial[\Hradius](\int (\sc<\Hradius>-\frac2{\Hradius^2})\,f \d\mug)\vert \le \frac C{\Hradius^{\frac72+\outve}}\Vert\rnu\Vert_{\Lp^p(\M)} \Vert f\Vert_{\Wkp^{1,q}(\M)}, \]
where we used $\nicefrac12\le\Vert\deform\rnu+\meanrnu\Vert_{\Lp^\infty(\M)}\le 2$ to shorten the notation. As $\partial*_\sigma(\d\mug-\d\mug[\sphg*])={-}\H\,\rnu\d\mug+\H\d\mug[\sphg*]$, this implies \eqref{lemma__derivatives__dsc}.

Now, we note that \eqref{lemma__derivatives__dg} implies
\begin{equation*} \Vert\partial[\Hradius]{(\vphantom{\big|}(\laplace<\Hradius>\,\funcg<\Hradius>)\circ\Phi<\Hradius>)} + \frac2\Hradius (\vphantom{\big|}\laplace<\Hradius>\,\funcg<\Hradius>)\circ\Phi<\Hradius> \Vert_{\Lp^p(\sphere^2,\pullback{\Phi<\Hradius>}\g<\Hradius>)} \le \frac C{\Hradius^{\frac72+\outve}}\Vert\rnu\Vert_{\Lp^\infty(\M)} \Vert g\Vert_{\Wkp^{2,p}(\M)} \labeleq{lemma__derivatives__dlaplace}\end{equation*}
for every function $g\in\Wkp^{2,p}(\sphere^2)$, its $\Phi$-constant expansion $\funcg<\Hradius>:=g\circ\Phi<\Hradius>^{-1}$, where $\Phi<\Hradius>:=\Phi(\Hradius,\,{\cdot}\,)$. In particular, we get
\[ \Vert \partial[\Hradius](\deform{(\funcg<\Hradius>)}\circ\Phi<\Hradius>+\int_{\M<\Hradius>}\frac{\funcg<\Hradius>}{\volume{\M<\Hradius>}}\d\mug<\Hradius>)\Vert_{\Lp^p(\M<\Hradius>)}
		\le \frac C{\Hradius^{\frac12+\outve}}\Vert\rnu\Vert_{\Lp^\infty(\M)}\Vert g\Vert_{\Wkp^{2,p}(\M)} \]
for every function $g=\trans g\in\Wkp^{2,p}(\M)$ and its $\Phi$-constant expansion $\funcg<\Hradius>:=g\circ\Phi<\Hradius>^{-1}$. Now, let $\funcg<\Hradius>_i$ be the $\Phi$-constant expansion of the eigenfunction $\eflap_i=\trans\eflap_i$ of the Laplace operator on $\sphere^2$ with respect to $\g<\Hradius>$ (for one fixed $\Hradius$) and $\Vert\eflap_i\Vert_{\Lp^\infty(\sphere^2)}=1$. We see that the functions $\eflap<\Hradius>_i:=\Vert\trans{\funcg<\Hradius>_i}\Vert_{\Lp^\infty(\M)}^{-1}\;\trans{\funcg<\Hradius>_i}$ satisfy the claim.

For \eqref{lemma__derivatives__dm}, we see the definition of the Hawking mass combined with $\H^2\equiv\nicefrac4{\Hradius^2}$ gives
\[ \mHaw = \sqrt{\frac{\volume{\M}}{16\pi}} ( 1 - \frac1{4\pi\Hradius^2}\volume{\M} ). \]
In particular, $\vert\volume{\M}-4\pi\,\Hradius^2\vert\le C\,\Hradius^{\frac32-\outve}$ implies the second inequality in~\eqref{lemma__derivatives__dm}
and
\begin{align*}
 \partial[\Hradius]@{(\mHaw(\M<\Hradius>))}
	={}& \volume{\M}^{{-}\frac12}\,(16\pi)^{{-}\frac12}(\frac12\partial[\Hradius]@{\volume{\M<\Hradius>}}
		- \frac3{8\pi\Hradius^2}\partial[\Hradius]@{\volume{\M<\Hradius>}}\,\volume{\M<\Hradius>}
		+ \frac1{2\pi\Hradius^3}\volume{\M}^2) \\
	={}& \frac12(\frac{\volume{\M}}{4\pi\Hradius^2})^{\frac12}\,(
		\meanrnu - \frac{3\,\meanrnu}{4\pi\Hradius^2}\,\volume{\M<\Hradius>}
		+ \frac1{2\pi\Hradius^2}\volume{\M})
\end{align*}
implies the rest of~\eqref{lemma__derivatives__dm}.
\end{proof}

\subsection{Proof of the main theorems}\label{ProofOfTheMainTheorem}
\begin{proof}[Proof of Theorem~\ref{Suff_ass_for_asymp_flat}]
If there exists a $\mathcal C^2_{\frac12+\outve}$-asymptotically flat coordinate system $\outx:\outM\setminus\outsymbol K\to\R^3\setminus\overline{B_1(0)}$ of $(\outM,\outg*)$ outside some compact set $\outsymbol K\subseteq\outM$, then \cite[Thm~3.1]{nerz2015CMCfoliation} implies the existence of such a CMC-foliation.

We see $\mathcal M\subseteq\regsphere*[\frac32+\ve]<\c\,\vert\mass\vert>C[p]$ for every $p\in\interval2\infty$ due to Lemma~\ref{Bootstrap_for_trace_free_second_fundamental_form__Lp2} and fix such $p\in\interval2\infty$. Replacing $\outM$ and $\Hradius_0$ by $\bigcup_\Hradius\M<\Hradius>$ and $\Hradius_0'$, respectively, we can assume that $\mathcal M$ is given by a smooth cover $\Phi:\interval{\Hradius_0}\infty\times\sphere^2\to\outM$ of the entire space $\outM$ and that the lapse function of $\Phi$ is strictly positive, see Lemma~\ref{Lemma__smooth_cover}. In particular, we can assume that $\Phi$ is differentiable, bijective, and satisfies $\Phi(\Hradius,\sphere^2)=\M<\Hradius>$, $\Phi(\interval{\Hradius_0}\infty,\sphere^2)=\outM$, and $\partial*_\Hradius\Phi=(\rnu\,\nu+\Hradius\,\levi*\trans\rnu)\circ\Phi$. From now on, we do not distinguish between quantities on $\outM$ and their pullback along $\Phi$.

Now, we prove that $\pullback{\Phi}\outg*$ corresponds (in highest order) to an Euclidean translation of an sphere. Let us first explain what this means. A translation of an Euclidean sphere (in non-constant direction) is a map $\Psi:\interval*{\Hradius_0}\infty\times\sphere^2\to\R^3$ such that 
\begin{equation*}
 \Psi(\Hradius_0,\,{\cdot}\,)=\mathcal i\hspace{.05em}_{\sphere^2_{\Hradius_0}(0)},  \qquad
		\partial*_\Hradius\Psi=\euknu+\Hradius\,\sphlevi*(\rnu<\Hradius>^i\,\euknu_i)
 \labeleq{Charac_of_translation}
\end{equation*}
where $\mathcal i\hspace{.05em}_{\sphere^2_{\Hradius_0}(0)}:\sphere^2\to\R^3$ is the standard embedding of $\sphere^2_{\Hradius_0}(0)$ in the Euclidean space ${}_\Hradius\boldsymbol{\rnu}=(\rnu<\Hradius>^1,\rnu<\Hradius>^2,\rnu<\Hradius>^3)\in\R^3$ is some vector (depending on $\Hradius$). We note that a direct calculation proves
\[ \pullback\Psi\eukg* = (1+\rnu<\Hradius>^i\,\euknu_i)\d\Hradius^2 + \Hradius^2\,\sphg*, \]
for such a $\Psi$ and that $\Psi$ is a diffeomorphism onto its image if $\vert{}_\Hradius\boldsymbol\rnu\vert<1$. Now, the above claim is that there exists such a translation $\Psi$ of an Euclidean sphere such that
\begin{equation*} \vert\pullback{(\Psi\circ\Phi)}\outg*-\eukg*\vert \le \frac C{\Hradius^{\frac12+\outve}}. \labeleq{Phi_is_translation}\end{equation*}
implying
\[ \vert B_r(p) \vert \ge \frac{2\pi}3\,r^3 \qquad\forall\,p\in\M<\Hradius>,\;\Hradius>\Hradius_0. \]
If this inequality is true, then \cite[Thm~1.1]{BandueKasueNakajima1989} (Theorem~\ref{BandueKasueNakajima1989}) proves the claim, i.\,e.~implies that $(\outM,\outg*)$ is asymptotically flat. Therefore, the claim is proven if we prove \eqref{Phi_is_translation} for some translation $\Psi$.

Due to the Lemmata~\ref{Lemma__smooth_cover} and~\ref{lemma__derivatives}, we know
\begin{equation*} \Vert\partial[\Hradius]@{\,\g<\Hradius>*} - \frac2\Hradius\g<\Hradius>*\Vert_{\Wkp^{1,p}(\M<\Hradius>)} + 
	\Hradius^2\,\Vert\partial<\Hradius>@{\,\sc<\Hradius>}+\frac4{\Hradius^3}\Vert_{\Wkp^{{-}1,q}(\M<\Hradius>)}  \le \frac C{\Hradius^{\frac32+\outve-\frac2p}}. \labeleq{proof_first_inequality} \end{equation*}
In particular, $\g<\Hradius>[r]$ converges in $\Wkp^{1,p}(\sphere^2,\sphg*)$ as $\Hradius\to\infty$ to some metric $\g<\infty>[r]*$ of $\sphere^2$ and $\sc<\Hradius>[r]$ converges to $2$ in $\Wkp^{{-1},q}(\sphere^2,\sphg*)$, where $\g<\Hradius>[r]:=\Hradius^{-2}\,\g<\Hradius>$ denotes the rescaled metric of $\M<\Hradius>$. Now, let $\varphi:\R^2\to\sphere^2$ be any conformal map with respect to $\g<\infty>[r]$, i.\,e.\ $\pullback\varphi{\g[r]<\infty>}=\exp(2\,\conformalf)\,\eukg*^2$ for the two-dimensional Euclidean metric $\eukg*^2$ and some $\conformalf\in\mathcal C(\R^2)$. In particular, we know that $\laplace\conformalf = {-}\exp(2\,\conformalf)$ weakly in $\R^2$ and that
\[ \int_{\R^2} \exp(2\,\conformalf) \d\mug = \mug<\infty>[r](\sphere^2) = \lim_{\Hradius\to\infty}\mug<\Hradius>[r](\sphere^2) = \lim_{\Hradius\to\infty}\frac{\mug(\M<\Hradius>)}{\Hradius^2} = 4\pi < \infty. \]
Therefore, Chen-Li's classification theorem \cite[Thm~1]{chen1991} implies $\pullback\psi(\g[r]<\infty>)=\sphg*$ for some diffeomorphism $\psi:\sphere^2\to\sphere^2$, see Appendix~\ref{Regularity_Calabi_energy} for more information. Without loss of generality, we can assume $\psi=\id_{\sphere^2}$. Thus, \eqref{proof_first_inequality} implies
\[ \Vert \g<\Hradius> - \Hradius^2\,\sphg*\Vert_{\Wkp^{1,p}(\sphere^2)} \le \frac C{\Hradius^{\frac12+\outve-\frac2p}}. \]
In particular, we know
\[ \Vert\laplace<\Hradius>\nu^i - \frac2{\Hradius^2}\nu^i\Vert_{\Lp^p(\M)} \le \frac C{\Hradius^{\frac12+\outve-\frac2p}}, \]
where $\nu^i\in\Ck^\infty(\sphere^2)$ are the Euclidean coordinates restricted to the standard embedded Euclidean unit sphere. Thus, we get
\[ \Vert\deform{(\nu^i)}\Vert_{\Wkp^{2,p}(\M)} + \Hradius^{\frac2p}\,\vert\mean\nu^i\vert \le \frac C{\Hradius^{\frac12+\outve-\frac2p}} \]
and Lemma~\ref{Lemma__smooth_cover} implies
\[ \Vert\rnu<\Hradius>-1-\rnu<\Hradius>_i\,\nu^i\Vert_{\Wkp^{2,p}(\M<\Hradius>)} \le \frac C{\Hradius^{\frac12+\outve-\frac2p}} \]
for some vector ${}_\Hradius\boldsymbol{\rnu}=(\rnu<\Hradius>^1,\rnu<\Hradius>^2,\rnu<\Hradius>^3)\in\R^3$ and by the positivity of $\rnu<\Hradius>$, we can assume that $\vert{}_\Hradius\boldsymbol\rnu\vert<1$. All in all, we therefore have proven
\[ \vert\pullback\Phi\outg* - ((1+\rnu^i\,\euknu_i)\d\Hradius^2 + \Hradius^2\,\sphg*) \vert \le \frac C{\Hradius^{\frac12+\outve}} \]
and choosing $\Psi$ with \eqref{Charac_of_translation} for these specific $\rnu<\Hradius>_i$ proves \eqref{Phi_is_translation} for some translation~$\Psi$. As explained above, this proves the claim.
\end{proof}

\begin{proof}[Proof of Theorem~\ref{Charac_of_Sobolev_asymp_flat}]
If there exists a coordinate system $(\outM,\outg*,\outx)$ such that $\pushforward\outx\outg*-\eukoutg*\in\Wkp^{3,q}_{\nicefrac12}(\R\setminus\overline{B_1(0)})$, then \cite[Thm~3.1, Remark~1.2]{nerz2015CMCfoliation} implies the existence of such a CMC-foliation. \medskip

Now, we prove the reverse implication. To stay with the notation used so far, we assume $\outc(\Hradius)=\Hradius^{{-}\outve}$. Although this would be an additional assumption, the proof is also valid in the general setting: if $\outc\in\Ck(\R)$ is arbitrary with $\outc(\Hradius)\to0$ as $\Hradius\to\infty$, then the inequalities used in the following are true for some $\oc''\in\Ck(\R)$ depending arbitrarily on $\oc$, but satisfying $\oc''(\Hradius)\to0$ for $\Hradius\to\infty$.\smallskip

We see $\mathcal M\subseteq\regsphere*[\frac32+\ve]<\c\,\vert\mass\vert>C[p]$\vspace{-.25em} due to Lemma~\ref{Bootstrap_for_trace_free_second_fundamental_form__Lp2}. Replacing $\outM$ and $\Hradius_0$ by $\bigcup_\Hradius\M<\Hradius>$ and $\Hradius_0'$, respectively, we can assume that $\mathcal M$ is given by a smooth cover $\Phi:\interval{\Hradius_0}\infty\times\sphere^2\to\outM$ of the entire space $\outM$ and that the lapse function of $\Phi$ is strictly positive, see Lemmata~\ref{Lemma__smooth_cover} and~\ref{lapse_function_control_III}. In particular, we can assume that $\Phi$ is differentiable, bijective, and satisfies \eqref{lemma__derivatives__fol_map_char}, i.\,e.~$\Phi(\Hradius,\sphere^2)=\M<\Hradius>$, $\Phi(\interval{\Hradius_0}\infty,\sphere^2)=\outM$, and $\partial*_\Hradius\Phi=(\rnu\,\nu+\Hradius\,\levi*\trans\rnu)\circ\Phi$. From now on, we do not distinguish between quantities on $\outM$ and their pullback along $\Phi$.

By the above, we can define a continuously differentiable map $\outsigma:\outM\to\interval{\Hradius_0}\infty$ such that $\outp\in\M<\outsigma(\outp)>$ for every $\outp\in\outM$ and we can equally define a vector field $\nu$ by the characterization $\nu<\Hradius>:=\nu|_{\M<\Hradius>}$ is the unit normal of $\M<\Hradius>$ with $D_{\nu}\outsigma>0$ (outward unit normal). By the inequalities \eqref{Lemma__smooth_cover_eq} on the lapse function $\rnu$, $\nu$ is also continuously differentiable. Hence, the metric $\g<\Hradius>*$ of $\M<\Hradius>$ depends smoothly on $\Hradius$, i.\,e.\ the function 
\[ \g<\Hradius>*(\outsymbol X,\outsymbol Y) := \g<\Hradius>*(\outsymbol X-\outg*(\outsymbol X,\nu<\Hradius>)\,\nu<\Hradius>,\outsymbol Y-\outg*(\outsymbol Y,\nu<\Hradius>)\,\nu<\Hradius>) \]
is at least continuously differentiable in $\outM$ for any smooth vector fields $\outsymbol X,\outsymbol Y\in\X(\outM)$. For the same reasons, we furthermore see that $\Hradius\mapsto\g<\Hradius>*(\outsymbol X,\outsymbol Y)\in\Hk(\M<\Hradius>)$ depends continuously on $\Hradius$. In particular, we can choose differentiable functions $\eflap_i : \outM\to \interval*-1*1$ such that $\eflap<\Hradius>_i:=\eflap_i|_{\M<\Hradius>}$ is as in Lemma~\ref{lemma__derivatives}. Now, we define
\begin{alignat*}4
 \centerz &{}:{}& \outM \to \R^3 : \outp \mapsto{}& (\sqrt{\frac3{4\pi\Hradius^2}}\int_{\Hradius_0}^{\Hradius(\outp)}\int_{\M<\Hradius>}\eflap<\Hradius>_\oi\,\rnu<\Hradius> \d\mug \d\Hradius)_{i=1}^3, \\
 \outx		&{}:{}& \outM \to \R^3 : \outp \mapsto{}& \outsigma(\outp)\,\big(\eflap_1(\outp),\eflap_2(\outp),\eflap_3(\outp)\big) + \centerz\big(\outsigma(\outp)\big)
\labeleq{Suff_ass_for_asymp_flat__outx}
\end{alignat*}
and prove that the latter is an asymptotically flat coordinate system, where $\rnu$ again denotes the lapse function. In the following, we identify $\centerz$ and $\centerz\circ\outsigma^{-1}$.\footnote{Note that we cannot choose $\centerz(\Hradius)$ to converge as $\Hradius\to\infty$ as this would imply that the center of mass in the resulting asymptotically flat coordinate system would be well-defined, which is not necessary true, \cite{cederbaumnerz2015_examples}.}

\textbf{On each CMC-leaf:} First, let us proof that $\outx<\Hradius>:=\outx|_{\M<\Hradius>}$ are coordinates with respect to which the induced metric $\g<\Hradius>*$ is asymptotically to the pullback of the corresponding Euclidean metric $\pullback{\outx<\Hradius>}\outg*$. We note that the estimates for the conformal parametrization in Proposition~\ref{Regularity_of_the_spheres} imply
\[ \vert\fint_{\M<\Hradius>}\frac{\radoutsigma^2}{\outsigma^2}\d\mug - 1\vert
	= \vert \sum_{i=1}^3\frac{\Vert\eflap_i\Vert_{\Lp^2(\M)}^2}{\volume{\M}} - 1\vert
	\le \frac C{\Hradius^{\frac12+\outve}}, \]
where $\radoutsigma:=\vert\outx-\centerz\,\vert$. Furthermore, Lemma~\ref{Regularity_of_the_spheres} implies
\[ \vert\frac{\laplace<\Hradius>\radoutsigma^2}{\Hradius^2}+6\,\frac{\radoutsigma^2-1}{\Hradius^4}\vert
	\le \vert\sum_{i=1}^3({-}\ewlap_i\,\eflap_i^2+\frac1{\Hradius^2}(1-\eflap_i^2)) + 6\,\frac{\radoutsigma^2-1}{\Hradius^4} \vert. \]
This means that $1-\nicefrac{\radoutsigma^2}{\Hradius^2}$ is (asymptotically) an eigenfunction of the (negative) Laplace operator with eigenvalue $\nicefrac6{\Hradius^2}$. Again using Proposition~\ref{Regularity_of_the_spheres}, we see that there are five $\Lp^2(\M)$-orthonormal eigenfunctions $\eflap_4$, $\eflap_5$, $\eflap_6$, $\eflap_7$, $\eflap_8$ of the Laplace operator such that the corresponding eigenvalues $\ewlap_i$ satisfy $\vert\lambda_i-\nicefrac6{\Hradius^2}\vert\le\nicefrac1{\Hradius^2}$ and these satisfy $\vert\lambda_i-\nicefrac6{\Hradius^2}\vert\le\nicefrac C{\Hradius^{\frac52+\outve}}$ for some constant $C$. Again comparing with the corresponding Eigenfunctions of the Euclidean sphere, we see that
\[ \vert\int f\eflap_i\d\mug\vert \le \sum_{k=1}^5\vert\int f g_k \d\mug \vert + \frac C{\Hradius^{\frac12+\outve}}\Vert f\Vert_{\Lp^2(\M)} \quad\forall\,i\in\{4,5,6,7,8\},\,f\in\Lp^2(\M), \]
where $g_1:=\sqrt5\,\eflap_1\,\eflap_2$, $g_2:=\sqrt 5\,\eflap_2\,\eflap_3$, $g_3:=\eflap_3^2-\nicefrac12\,(\eflap_1^2+\eflap_2^2)$, $g_4:=\sqrt 5\,\eflap_1\,\eflap_3$ and $g_5:=\nicefrac12\,(\eflap_1^2-\eflap_2^2)$. By calculating $\int\radoutsigma^2\,g_i\d\mug$, this implies
\[ \Vert\frac{\radoutsigma^2}{\Hradius^2}-1\Vert_{\Lp^2(\M)}\le C\,\Hradius^{-\frac12-\outve}. \]
Using the regularity of the Laplace operator this means
\[ \Vert \frac{\radoutsigma^2}{\sigma^2} - 1 \Vert_{\Wkp^{2,p}(\M)} \le \frac C{\Hradius^{\frac12+\outve-\frac2p}}. \]
With 
\begin{align*}
 \laplace\levi*(\frac{\radoutsigma^2}{\sigma^2})
	={}& \frac\sc2\,\levi*(\frac{\radoutsigma^2}{\sigma^2})+\levi*\laplace(\frac{\radoutsigma^2}{\sigma^2}) \\
	={}& \sum_{i=1}^3(\frac\sc2\,\levi*\eflap_i^2 + \levi*({-}2\ewlap_i\,\eflap_i^2+2\trtr{\levi*\eflap_i}{\levi*\eflap_i})) \\
	={}& \sum_{i=1}^3((\frac\sc2-3\,\ewlap_i)\,(\levi*\eflap_i^2) + 4\,\Hesstrf\,\eflap_i(\levi*\eflap_i,\cdot)),
\end{align*}	
the inequalities for $\sc$, $\Hesstrf\,\eflap_i$, $\ewlap_i$, and the above one for $\radoutsigma^2$, we get
\[ \vert\laplace\levi*(\frac{\radoutsigma^2}{\sigma^2}) \vert 
	\le \vert \sum_{i=1}^3((\frac\sc2-3\frac2{\Hradius^2})\,(\levi*\eflap_i^2) + 4\,\Hesstrf\,\eflap_i(\levi*\eflap_i,\cdot)) \vert + \frac C{\Hradius^{\frac72+\outve}}. \]
Thus, we can strengthen the above inequality to
\[ \Vert \frac{\radoutsigma^2}{\sigma^2} - 1 \Vert_{\Wkp^{3,p}(\M)} \le \frac C{\Hradius^{\frac12+\outve-\frac2p}}. \]
Hence, there is a function $\graphf<\Hradius>\in\Wkp^{3,p}(\sphere^2_\Hradius(\centerz<\Hradius>\,))$ such that $\M<\Hradius>':=\outsymbol x(\M<\Hradius>) = \graph\graphf<\Hradius>$ and
\[ \Vert\graphf<\Hradius>\Vert_{\Wkp^{3,p}(\sphere^2_\Hradius(\centerz<\Hradius>\,))} \le \frac C{\Hradius^{\frac12+\outve-\frac2p}}, \quad
	\Vert\eukoutg*-\pushforward{\graphF<\Hradius>\!\!}\hspace{.05em}\sphg<\Hradius>*\Vert_{\Wkp^{2,p}(\M<\Hradius>')} \le \frac C{\Hradius^{\frac12+\outve-\frac2p}}, \]
where $\graphF$ is the graph function of $f$, $\sphg<\Hradius>*=\Hradius^2\,\sphg*$ is the standard metric on $\sphere^2_\Hradius(\centerz<\Hradius>\,)$, and $\centerz<\Hradius>:=\centerz(\outp)$ for any (and therefore every) $\outp\in\M<\Hradius>$. In particular, $\outx<\Hradius>:=\outx|_{\M<\Hradius>}$ are coordinates of $\M<\Hradius>$. Again, using the estimates for the conformal parametrization in Proposition~\ref{Regularity_of_the_spheres}, we furthermore know 
\[ \vert\g*(\levi*\outx<\Hradius>_i,\levi*\outx<\Hradius>_j) - \sphg<\sigma>*(\euklevi*\outy_i,\euklevi*\outy_j) \vert \le \frac C{\Hradius^{\frac12+\outve}}, \]
where $\outy:\R^3\to\R^3$ are the standard coordinates. This implies
\[ \vert\g_{I\!J}-\eukoutg_{I\!J}\vert \le \frac C{\Hradius^{\frac12+\outve}} \]
and we get
\[ \Vert\g_{I\!J}-\eukoutg_{I\!J}\Vert_{\Wkp^{1,p}(\M)} \le \frac C{\Hradius^{\frac12+\outve-\frac2p}} \]
by the same argument. Doing a similar calculation as above for $\levi*\eflap_i$ instead of $\levi*(\nicefrac{\radoutsigma^2}{\Hradius^2})$, we strengthen this to
\begin{equation*} \Vert\g_{I\!J}-\eukoutg_{I\!J}\Vert_{\Wkp^{2,p}(\M)} \le \frac C{\Hradius^{\frac12+\outve-\frac2p}}. \labeleq{Suff_ass_for_asymp_flat__g-ge} \end{equation*}
Thus, $\outx<\Hradius>$ is a coordinate system of $\M<\Hradius>$ such that $\g<\Hradius>*-\pullback{\outx<\Hradius>}\eukoutg*$ decays suitable fast.\pagebreak[2]\smallskip

\textbf{On each leaf ($\Hradius$-derivative):} Now, we prove that the metrics on the single leaves depends continuously differentiable on $\Hradius$ and that the corresponding derivative decays suitable fast. Using the Lemmata~\ref{lemma__derivatives} and~\ref{lapse_function_control_III} on the foliation $\Phi$, we see
\[ \Vert\partial[\Hradius]@{\graphf<\Hradius>}\Vert_{\Wkp^{2,p}(\sphere^2_\Hradius(\centerz<\Hradius>\,))} \le \frac C{\Hradius^{\frac32+\outve-\frac2p}}, \]
where $\graphf<\Hradius>$ is the function on $\sphere^2_\Hradius(\centerz<\Hradius>\,)$ with $\graph\graphf<\Hradius>=\M<\Hradius>$. Here, we used the map
\[ \sphere^2_\Hradius(\centerz<\Hradius>\,)\to \sphere^2_{\Hradius'}(\centerz<\Hradius'>\,) : p\mapsto \frac{\Hradius'}\Hradius(p - \centerz<\Hradius>\,) + \centerz<\Hradius'> \]
to identify $\sphere^2_\Hradius(\centerz<\Hradius>\,)$ and $\sphere^2_{\Hradius'}(\centerz<\Hradius'>\,)$. Correspondingly, we get
\[ \Vert\partial[\Hradius]@{\g_{I\!J}}-\frac2\Hradius\g_{I\!J}\Vert_{\Wkp^{1,p}(\M)} \le \frac C{\Hradius^{\frac12+\outve-\frac2p}}, \]
where we used the graph function $\graphF<\Hradius>$ of $\graphf$ and the above map to choose one coordinate system for every $\M<\Hradius'>$ with sufficiently small $\vert\Hradius-\Hradius'\vert$.\pagebreak[2]\smallskip

\textbf{Radial direction:} Again using the Lemmata~\ref{lemma__derivatives} and~\ref{lapse_function_control_III} on $\Phi$, we see
\[ \vert D_{\rnu\nu}\outx_i - \eflap_i - 3\fint\eflap<\sigma>_i\,\rnu<\sigma> \d\mug \vert \le \frac C{\Hradius^{\frac12+\outve}}. \]
In particular, we get
\begin{align*}\hspace{4em}&\hspace{-4em}
 \vert\eukoutg*(D_{\rnu\nu}\outx,D_{\rnu\nu}\outx) - (1 + \trans\rnu)^2 \vert \\
 \le{}& \vert\eukoutg*(D_{\rnu\nu}\outx,D_{\rnu\nu}\outx) - (1 + 6\sum_{i=1}^3\eflap_i\fint\eflap<\sigma>_i\,\rnu<\sigma>\d\mug + \sum_{i=1}^3(9\fint\eflap<\sigma>_i\,\rnu<\sigma> \d\mug)^2) \vert \\ &+ \frac C{\Hradius^{\frac12+\outve}} \\
 \le {}& \vert\eukoutg*(D_{\rnu\nu}\outx,D_{\rnu\nu}\outx) - (\eflap_i+3\fint\eflap<\sigma>_i\,\rnu<\sigma> \d\mug)^2 \vert + \frac C{\Hradius^{\frac12+\outve}}
 \le \frac C{\Hradius^{\frac12+\outve}}
\end{align*}
implying
\[ \vert(\pullback\outx\eukoutg*)(\nu,\nu) - 1 \vert \le \frac C{\Hradius^{\frac12+\outve}}. \]
By the same argument, we see that the corresponding result holds for the $\M$-tangential derivative, i.\,e.\
\[ \vert D_X((\pullback\outx\eukoutg*)(\nu,\nu)) \vert \le \frac C{\Hradius^{\frac32+\outve}} \qquad\forall\,X\in\X(\M<\Hradius>), \]
and that the $\eukoutg*$-tangential part of $\pushforward\outx\nu$ decays with $\nicefrac C{\Hradius^{\frac12+\outve}}$ (and correspondingly for the first derivative). In particular $\outx$ is a coordinate system of $\outM$.\pagebreak[3]\smallskip

\textbf{Radial direction ($\Hradius$-derivative):} Thus, left to prove is
\[ \vert D_{\nu}((\pullback\outx\eukoutg*)(\nu,\nu))\vert \le \frac C{\Hradius^{\frac32+\outve}} \]
and that the corresponding result holds for the $\pullback\outx\eukoutg*$-tangential part of $\nu$. However by the above, this is equivalent to
\[ \Vert\partial[\Hradius]@{\deform\rnu}\Vert_{\Wkp^{1,p}(\M)} \le \frac C{\Hradius^{\frac32+\outve}}, \]
which is true due to Lemma~\ref{lapse_function_control_III}. 
Therefore, we have constructed $\Wkp^{1,p}_{\frac12}$- and $\Wkp^{1,\infty}_{\frac12+\outve}$-asymptotically flat coordinates.\footnote{In the general setting, the $\Wkp^{1,\infty}_{\frac12+\outve}$-decay means that $\rad^{\frac12}\,\vert\outg*-\eukoutg\vert\to0$ for $\rad\to\infty$ and correspondingly for the derivative of $\outg*$.} In particular, these coordinates are $\Wkp^{1,q}_{\frac12}$-asymp\-to\-tically flat for every $q\in\interval2\infty$.
\smallskip

\textbf{Final step:} The rest of the proof is analogous to Bartnik's proof of \cite[Prop.~3.3]{bartnik1986mass}. We repeat it nevertheless for the reader's convenience. By the above, we know that $(\outM,\outg,\outx)$ is a structure of infinity, see \cite[Def.~2.1]{bartnik1986mass}. Due to Bartnik's famous result, there exists harmonic coordinates\footnote{In fact, an equivalent argument can be used to prove that $\outx$ are $\Wkp^{3,p}_{\frac12}$\vspace{-.25em}-asymptotically flat coordinates.} $\outy:\outM\setminus\outsymbol K\to\R^3\setminus\overline{B_1(0)}$ for some compact set $\outsymbol K\subseteq\outM$ with $\mathcal{err}_{\oi\oj}:=((\pullback\outy\outg*)_{\oi\oj}-\eukoutg_{\oi\oj})\in\Wkp^{1,q}_{{-}\frac12}(\outM)$ for every $q\in\interval2\infty$ -- in particular for $q=\nicefrac{6p}{(3-p)}$. As $\outy$ are harmonic coordinates, we have
\begin{equation*} \outric_{\oi\oj} = {-}\frac12\laplace[\outg]\outg_{\oi\oj} + Q_{\oi\oj}(\outg*,\partial*\outg*), \labeleq{Laplace_g_ij} \end{equation*}
where $Q_{\oi\oj}(\outg*,\partial*\outg*)$ depends quadratic on the first coordinate derivatives of $(\outg_{kl})_{kl}$ and polynomial on $(\outg_{kl})_{kl}$ and $(\outg^{kl})_{kl}$. Therefore, using the Sobolev inequality on $\outric\in\Wkp^{1,p}_{{-}\frac52}(\outM$) implies
$\laplace[\outg]\outg_{\oi\oj} \in \Lp^{\!\frac{3p}{3-p}}_{-\frac52}(\outM)$
and therefore $\mathcal{err}_{\oi\oj}\in\Wkp^{2,\frac{3p}{3-p}}_{-\frac12}(\outM)$\vspace{-.25em}. Thus, the Sobolev inequality and $\nicefrac{3p}{(3-p)}>3$ imply $\mathcal{err}_{\oi\oj}\in\Wkp^{1,\infty}_{-\frac12}(\outM)$. Thus, \eqref{Laplace_g_ij} gives $\laplace[\outg]\outg_{\oi\oj}\in\Wkp^{1,p}_{{-}\frac52}(\outM)$ and therefore $\mathcal{err}_{\oi\oj}\in\Wkp^{3,p}_{-\frac12}(\outM)$. This proves the claim.
\end{proof}

\section{Alternative foliation properties and uniqueness of the foliation}\label{Foliation_property}%
In this section, we give alternative assumptions which imply that the elements of the cover $\mathcal M$ are pairwise disjoint and that $\mathcal M$ is globally unique, where we allow the assumptions of Theorem~\ref{Suff_ass_for_asymp_flat} or~\ref{Charac_of_Sobolev_asymp_flat}. Note that by \cite[Lemma3.5]{nerz2015CMCfoliation}, the following assumptions~\ref{Foliation_property__fol} and~\ref{Foliation_property__fol_area} do not only imply that the elements of $\mathcal M$ are disjoint, but are equivalent to it.
\begin{proposition}[Foliation property]\label{Foliation_property__prop}
Let the assumptions~\ref{Ass_fst}--\ref{Ass_lst} be satisfied for some family $\mathcal M=\{\M<\Hradius>\}_{\Hradius>\Hradius_0}$ with $\mathcal M\subseteq\regsphere*[\frac32+\ve]<M>\c[p]$ or with \eqref{assumptions}.\footnote{In the following, $\Hradius^{-\outve}$ has to be replaced by $\oc''(\Hradius)$ with $\oc''(\Hradius)\to0$ for $\Hradius\to\infty$ if $\mathcal M\subseteq\regsphere*[\frac32+\ve]<\c\,\vert\mass\vert>\c[p]$ is not satisfied.} Then the elements of $\mathcal M$ are pairwise disjoint \hbox{{\normalfont(}if} $\Hradius_0$ is sufficiently \hbox{large{\normalfont)}} if one of the following assumptions is satisfied\smallskip
\begin{enumerate}[nosep,series=foliation_properties,label={\normalfont (A\arabic{enumi})}]
\item $\displaystyle\left|\int\!\gauss<\Hradius>\,\eflap<\Hradius>\d\mug<\Hradius>\right| \le \sqrt{\frac\pi3}\,\frac{(1-\eta)\,\vert\mHaw\vert}{\Hradius^2}\left\Vert\eflap<\Hradius>\right\Vert_{\Lp^2(\M<\Hradius>)}$ for every $\M=\M<\Hradius>\in\mathcal M$, every function $\eflap=\eflap<\Hradius>\in\Wkp^{2,2}(\M<\Hradius>)$ with $\laplace<\Hradius>\eflap={-}\ewlap\,\eflap$ for some arbitrary $\ewlap=\ewlap<\Hradius>(\eflap<\Hradius>)\in\interval0{\nicefrac3{\Hradius^2}}$, where $\eta\in\interval01$ is some fixed constant;\label{Foliation_property__fol}\smallskip
\item $\displaystyle\left|\partial[\Hradius]@{\volume{\M<\Hradius>}}-\frac2\Hradius\volume{\M}\left(1-\frac{\mHaw}\Hradius\right)\right|\le\oc(\Hradius)\,\Hradius$ for some function $\oc\in\Ck(\R)$ with $\oc(\Hradius)\to0$ for $\Hradius\to\infty$; \label{Foliation_property__fol_area}\smallskip
\item $\displaystyle\left|\partial[\Hradius]@{\mHaw(\M<\Hradius>)}\right|\le\frac{\oc(\Hradius)}\Hradius$ for some function $\oc\in\Ck(\R)$ with $\oc(\Hradius)\to0$ for $\Hradius\to\infty$. \label{Foliation_property__fol_mass}
\end{enumerate}
Furthermore, \ref{Foliation_property__fol_area} and \ref{Foliation_property__fol_mass} are equivalent and imply \ref{Foliation_property__fol} for $\eta=\eta(\Hradius)\to1$ for $\Hradius\to\infty$.
\end{proposition}
We see directly that this proposition is equivalent to the lapse function $\rnu<\Hradius>$ of a cover map $\Phi$ satisfying \eqref{lemma__derivatives__fol_map_char} is in highest order constant, i.\,e.\ that~$\rnu<\Hradius>$ has constant sign (if $\Hradius$ is sufficiently large). By \eqref{Lemma__smooth_cover_eq}, the crucial part is to prove that~$\trans{\rnu<\Hradius>}$ is sufficiently bounded away from one (for sufficiently large $\Hradius$). In the following first lemma, we prove that this in fact equivalent to calculate the expanding part $\meanrnu=\fint\rnu\d\mug$, i.\,e.~its mean value, up to order $\mathcal o(\Hradius^{{-}1})$ and not only up to order $\mathcal O(\Hradius^{{-}\frac12-\outve})$ as it was done in~\eqref{Lemma__smooth_cover_eq}. Such a precise determination of~$\meanrnu$ is possible due to the knowledge of the area $\M<\Hradius>$ up to order $\mathcal O(\Hradius^{{-}1-\outve})$ using the Hawking mass, i.\,e.~\eqref{lemma__derivatives__dm}.
\begin{lemma}[Estimating \texorpdfstring{$\trans\rnu$}{the translating part of the lapse function} by \texorpdfstring{$\meanrnu$}{the rescaling part of the lapse function}]\label{lapse_function__transl_to_exp_part}
If the assumptions~\ref{Ass_fst}--\ref{Ass_lst} are satisfied, then there are constants $\Hradius_0'=\Cof{\Hradius'_0}[M][\ve][p][\c]$ and $C=\Cof[M][\ve][p][\c]$ such that $\Hradius_0>\Hradius_0'$ implies that the lapse function $\rnu$ of a $\Phi$ as in \eqref{lemma__derivatives__fol_map_char} satisfies
\begin{equation*} \vert\frac1\Hradius\Vert\trans\rnu\Vert_{\Lp^\infty(\M)}^2 - (\meanrnu-1-\frac{\mHaw}\Hradius)\vert \le \frac C{\Hradius^{1+\outve}} + \frac C{\Hradius^{1+\outve}}\Vert\trans\rnu\Vert_{\Lp^\infty(\M)}^2 \labeleq{compare_ut_and_bar_u} \end{equation*}
\end{lemma}
\begin{proof}
Using Proposition~\ref{Regularity_of_the_spheres}, $\int\laplace\rnu\d\mug=0$, the second inequality in~\eqref{Regularity_of_the_spheres__k}, and \eqref{lapse_func_charaction}, i.\,e.~$\jacobiext*\rnu \equiv \nicefrac{\H^2}2$, we get
\begin{align*}
 \vert\int_{\M<\Hradius>}\outric*(\nu,\nu)\rnu\d\mug - \frac{2(1-\meanrnu)}{\Hradius^2}\volume{\M} \vert
 ={}& \vert\int\jacobiext*\rnu-\trtr\zFund\zFund\rnu-\laplace\rnu\d\mug - \int\frac{\H^2}2(1-\meanrnu)\d\mug\vert \\
 \le{}& \frac C{\Hradius^{\frac12+\outve}}.
\end{align*}
We combine this with $\int\deform\rnu\,\trans\rnu\d\mug=0$, $\int\trans\rnu\d\mug=0$, $\jacobiext*\rnu\equiv\nicefrac2{\Hradius^2}$, \eqref{Lemma__smooth_cover_eq}, the second inequality in~\eqref{Regularity_of_the_spheres__k}, and the Gau\ss\ equation and get
\begin{align*}
 \frac{6\mHaw}{\Hradius^3}\Vert\trans\rnu\Vert_{\Lp^2(\M)}^2
	\le{}& \int\jacobiext*\trans\rnu\,\trans\rnu\d\mug + \frac C{\Hradius^{3+\outve}}\Vert\trans\rnu\Vert_{\Lp^2(\M)}^2 \\
	\le{}& \int\jacobiext\rnu\,\trans\rnu \d\mug
				 - \int\jacobiext\deform\rnu\,\trans\rnu \d\mug \\
			 & - \meanrnu\,\int (\outric*(\nu,\nu)+\trtr\zFund\zFund)\trans\rnu \d\mug
				 + \frac C{\Hradius^{3+\outve}}\Vert\trans\rnu\Vert_{\Lp^2(\M)}^2 \\
	\le{}& {-} \frac{2(1-\meanrnu)}{\Hradius^2}\meanrnu\volume{\M}
				 + \meanrnu\int \outric*(\nu,\nu)\deform\rnu \d\mug
				 + \meanrnu^2\,\int \outric*(\nu,\nu) \d\mug \\
			 & + \frac C{\Hradius^{1+\outve}}
				 + \frac C{\Hradius^{2+\outve}}\Vert\trans\rnu\Vert_{\Lp^2(\M)}
				 + \frac C{\Hradius^{3+\outve}}\Vert\trans\rnu\Vert_{\Lp^2(\M)}^2 \\
	\le{}& 8\pi\,(\meanrnu-1
				 - \frac{\mHaw}\Hradius)
				 + \frac C{\Hradius^{1+\outve}}
				 + \frac C{\Hradius^{3+\outve}}\Vert\trans\rnu\Vert_{\Lp^2(\M)}^2
\end{align*}
and equivalent for \lq$\ge$\rq.\pagebreak[1] Now, we note that $\frac3{4\pi}\Vert\sphtrans f\Vert_{\Lp^2(\sphere,\sphg*)}^2=\Vert\sphtrans f\Vert_{\Lp^\infty(\sphere^2)}^2$ on the Euclidean sphere and the translating part of a function with respect to the standard metric of the Euclidean unit sphere and therefore, Proposition~\ref{Regularity_of_the_spheres} implies
\[ \vert\Vert\trans\rnu\Vert_{\Lp^\infty(\M)}^2 - \frac3{4\pi\Hradius^2}\Vert\trans\rnu\Vert_{\Lp^2(\M)}^2 \vert \le \frac C{\Hradius^{\frac12+\outve}}\Vert\trans\rnu\Vert_{\Lp^\infty(\M)}^2. \]
Now, \eqref{compare_ut_and_bar_u} follows from the combination of the last two inequalities.
\end{proof}

\begin{proof}[Proof of Proposition~\ref{Foliation_property__prop}]
If \ref{Foliation_property__fol_area} or \ref{Foliation_property__fol_area} is satisfied, then Lemma~\ref{lapse_function__transl_to_exp_part} (and \eqref{lemma__derivatives__dm}) directly implies $\Vert\trans\rnu\Vert_{\Lp^\infty(\M<\Hradius>)}\to0$ for $\Hradius\to\infty$. In particular, $\rnu>0$ and therefore the elements of $\mathcal M$ are pairwise disjoint. By the same argument, these two properties are equivalent. The following argument also proves that both imply property~\ref{Foliation_property__fol}.\medskip

Per Definition of $\Phi$ and $\M<\Hradius>$, we know $\jacobiext*\rnu = \spartial[\Hradius]@{\H<\Hradius>} \equiv \nicefrac2{\Hradius^2}$. Using \eqref{Lemma__smooth_cover_eq} and Proposition~\ref{Regularity_of_the_spheres} imply $\vert\hspace{.05em}\jacobiext*(\rnu-1) + \outric*(\nu,\nu)\vert \le \nicefrac C{\Hradius^{3+\ve}}$ and $\Phi$ is therefore a diffeomorphism if $\Vert\trans{\rnu}\Vert_{\Lp^\infty(\M)}\le1-\eta$ with $\eta>0$ (independent of $\Hradius$). But again by Proposition~\ref{Stability} (for sufficiently large $\Hradius_0$), this is implied by
\[ \vert\sum_{i=1}^3 \int\outric*(\nu,\nu)\,\eflap_i\d\mug\,\eflap_i \vert \le (1-\eta)\,\frac{6\,\vert\mHaw(\M)\vert}{\Hradius^3}, \]
for some $\eta>0$ (independent of $\Hradius$). By comparing with the Euclidean sphere using Proposition~\ref{Regularity_of_the_spheres}, we see that this is the case if 
\[ \vert \int\outric*(\nu,\nu)\,\eflap_i\d\mug \vert \le (1-\eta)\,\sqrt{\frac{16\pi}3}\,\frac{\vert\mHaw(\M)\vert}{\Hradius^2}\qquad\forall\,i\in\{1,2,3\} \]
for some $\eta>0$ (independent of $\Hradius$) -- and by the same argument, this inequality is true if $\Vert\trans\nu\Vert_{\Lp^\infty(\M)}\le1-\eta'$ for some $\eta'>0$ (independent of $\Hradius$). Combining the decay of $\outsc$ and $\zFundtrf*$ with the fact that $\eflap_i$ is mean value free, this is true due to the Gau\ss\ equation and the assumed foliation property if $\Hradius$ is sufficiently large. Thus, the elements of $\mathcal M$ are pairwise disjoint if \ref{Foliation_property__fol} is satisfied.
\end{proof}

\begin{proposition}[Foliation property II]\label{Foliation_property__prop_2}
Let the assumptions~\ref{Ass_fst}--\ref{Ass_lst} be satisfied for some family $\mathcal M=\{\M<\Hradius>\}_{\Hradius>\Hradius_0}$ with $\mathcal M\subseteq\regsphere*[\frac32+\ve]<M>\c[p]$ or with \eqref{assumptions} is satisfied.\footnote{In the following, $\Hradius^{-\outve}$ has to be replaced by $\oc''(\Hradius)$ with $\oc''(\Hradius)\to0$ for $\Hradius\to\infty$ if $\mathcal M\subseteq\regsphere*[\frac32+\ve]<\c\,\vert\mass\vert>\c[p]$ is \emph{not} satisfied.} Then the elements of $\mathcal M$ are pairwise disjoint (if $\Hradius_0$ is sufficiently large) and \ref{Foliation_property__fol}--\ref{Foliation_property__fol_mass} are true for $\eta=\oc(\Hradius)\to0$ if one of the following assumptions is true\smallskip
\begin{enumerate}[resume*=foliation_properties]
\item $\Vert D_{\nu}\outsc\Vert_{\Lp^1(\M<\Hradius>)}\le \oc'(\Hradius)\,\Hradius^{{-}\frac32}$ for every $\Hradius>\Hradius_0$ and some  constant $\oc\relax<\infty$ independent of $\Hradius$;\label{Foliation_property__strong_sc}\smallskip
\item $\Vert\outsc\Vert_{\Lp^1(\M<\Hradius>)}+\Hradius\,\Vert D_{\nu}\outsc\Vert_{\Lp^1(\M<\Hradius>)}\le\oc'(\Hradius)$ for some function $\oc'\in\Lp^1(\R)$, \label{Foliation_property__strong_Dsc}\smallskip
\end{enumerate}
where $\oc(\Hradius)\le\Hradius^{{-}\frac\outve2}$ if $\mathcal M\subseteq\regsphere*[\frac32+\ve]<M>\c[p]$ and \ref{Foliation_property__strong_sc} is true for $\oc'(\Hradius)\le\Hradius^{{-}\outve}$.
\end{proposition}

\begin{lemma}[Derivatives of the lapse function]\label{lapse_function_control_II}
Let the assumptions~\ref{Ass_fst}--\ref{Ass_lst} be true and $s\in\interval*1*\infty$ be a constant. There exist constants $\Hradius_0'=\Cof{\Hradius'_0}[M][\ve][p][s][\c]$ and $C=\Cof[M][\ve][p][s][\c]$ such that $\Hradius_0>\Hradius_0'$ implies that the lapse function $\rnu$ of the parametrization $\Phi$ of $\mathcal M$ as in Lemma~\ref{lemma__derivatives} satisfies
\begin{subequations}\label{d_sigma_u__1}
\begin{align}
	\volume{\M}^{{-}1}\,\vert\int\partial[\Hradius]@{\rnu<\Hradius>}\d\mug\vert \le{}& C\,\Hradius^{2-\frac2s}\,(\Hradius^{\frac2s-\frac72-\outve}+\Vert D_{\nu}\outsc\Vert_{\Lp^s(\M)})\Vert\rnu\Vert_{\Lp^\infty(\M)}, \\
	\Vert\deform{(\partial[\Hradius]@{\rnu<\Hradius>})}\Vert_{\Wkp^{1,P}(\M)} \le{}& C\,\Hradius^{2+\frac2P-\frac{2}s}\,(\Hradius^{\frac2s-\frac72-\outve}+\Vert D_{\nu}\outsc\Vert_{\Lp^s(\M)})\Vert\rnu\Vert_{\Lp^\infty(\M)},
\end{align}
\end{subequations}
where $P:=\nicefrac{2s}{(2-s)}$ if $s\le\nicefrac{2p}{(2+p)}$ and $P:=p$ if $s\ge\nicefrac{2p}{(2+p)}$.
\end{lemma}
Here, we again used $\nicefrac12\le\Vert\rnu\Vert_{\Lp^\infty(\M)}\le 2 + \Vert\trans\rnu\Vert_{\Lp^\infty(\M)}$ to reduce notation.
\begin{proof}
Using $\jacobiext\rnu\equiv\nicefrac2{\Hradius^2}$, we see -- in normal coordinates with respect to $\g<\Hradius_0>$ for some fixed $\Hradius_0$ --
\begin{align*}
 \jacobiext(\partial[\Hradius]@{\,\rnu<\Hradius>})
		={}& \partial[\Hradius]@{(\jacobiext*\rnu<\Hradius>)}
				- \partial[\Hradius]@{\g<\Hradius>^{\ii\!\ij}}\partial^2_\ii_\ij@{\rnu<\Hradius>}
				+ \g<\Hradius>^{\ii\!\ij}\partial[\Hradius]@{\levi<\Hradius>_{\ii\!\ij}^\ik}\partial_\ik@{\rnu<\Hradius>}
				- \frac12\partial[\Hradius]@{\H^2}\rnu<\Hradius> \\
			& - \partial[\Hradius]@{\trtr{\zFundtrf<\Hradius>}{\zFundtrf<\Hradius>}}\rnu<\Hradius>
				- \partial[\Hradius]@{\outric*(\nu<\Hradius>,\nu<\Hradius>)}\rnu<\Hradius>
\end{align*}
and therefore \eqref{Lemma__smooth_cover_eq}, \eqref{lemma__derivatives__dg}, and \eqref{lemma__derivatives__dk} imply
\[
	\Vert\text{err}-\trans{\text{err}}\Vert_{\Lp^p(\M)}
	\le \frac C{\Hradius^{\frac72+\outve}}(\Hradius^{\frac2p}+\Vert\trans\rnu\Vert_{\Lp^p(\M)}),
\]
where
\[ \text{err}:=\jacobiext*(\partial[\Hradius]<\Hradius>@{\rnu<\Hradius>}) + \partial[\Hradius]@{(\outric*(\nu<\Hradius>,\nu<\Hradius>))}\,\rnu. \]
Hence, we get by Proposition~\ref{Stability} and the regularity of the weak Laplace operator
\[ \Vert\deform{(\partial[\Hradius]@{\rnu})}\Vert_{\Wkp^{1,P}(\M)} \le \frac C{\Hradius^{\frac32+\outve}}(\Hradius^{\frac2P}+\Vert\trans\rnu\Vert_{\Lp^P(\M)}) + C\,\Hradius^2\,\Vert\partial[\Hradius]@{(\outric*(\nu[\Hradius],\nu[\Hradius]))}\,\rnu\Vert_{\Wkp^{{-}1,Q}(\M)} \]
for $P\in\interval1*p$ with $\nicefrac1P+\nicefrac1Q=1$ arbitrary and we get the identical inequality for the rescaling part.\pagebreak[1] However, the Gau\ss\ equation, the assumed control for the $\Hradius$-derivative of $\outsc$, the inequalities \eqref{Lemma__smooth_cover_eq}, \eqref{lemma__derivatives__dk}, and \eqref{lemma__derivatives__dsc}, and the Sobolev inequalities imply
\begin{equation*}
 \Vert\partial[\Hradius]@{(\outric*(\nu[\Hradius],\nu[\Hradius]))}\,\rnu\Vert_{\Wkp^{{-}1,\frac{2s}{3s-2}}(\M)} \le \frac C\Hradius\,(\Hradius^{\frac2s-\frac72-\outve}+\Vert D_{\nu}\outsc\Vert_{\Lp^s(\M)})\Vert\trans\rnu\Vert_{\Lp^\infty(\M)} \labeleq{lemma__derivatives__sc_W1q}
\end{equation*}
for $P$ as in the claim of the Lemma, where we used $1+\Vert\trans\rnu\Vert_{\Lp^\infty(\M)}\le 2\Vert\rnu\Vert_{\Lp^\infty(\M)}$. All in all, we have proven~\eqref{d_sigma_u__1}.
\end{proof}

\begin{remark}[Pointwise and Sobolev setting]
In the setting of~\ref{Foliation_property__strong_sc}, we can replace the $D_{\nu}\outsc$-term on the right hand side of \eqref{d_sigma_u__1} by $\oc'(\Hradius)$. If~\ref{Foliation_property__strong_Dsc} is satisfied, \ref{Charac_of_Sobolev_asymp_flat} implies
\begin{align*}\hspace{3em}&\hspace{-3em}
 \vert\Hradius^{\frac{1+\eta}2}(\meanrnu<\Hradius>-1) - \Hradius^{\frac{1+\eta}2}(\meanrnu<(\Hradius+\sqrt\Hradius)>-1) \vert \\
 \le{}& \int_{\Hradius}^{\Hradius+\sqrt\Hradius}\vert\partial[\Hradius]@{(\Hradius^{\frac{1+\eta}2}(\meanrnu<\Hradius>-1))}\vert \d s + \vert(\Hradius+\sqrt\Hradius)^{\frac{1+\eta}2}-\Hradius^{\frac{1+\eta}2}\vert\,\vert\meanrnu<(\Hradius+\sqrt\Hradius)>-1\vert \\
 \le{}& \int_{\Hradius}^{\Hradius+\sqrt\Hradius}\frac{C\,\oc(s)}{s^{1-\frac\eta2}}(1+\Vert\trans{\rnu<s>}\Vert_{\Lp^\infty(\M<s>)}) \d s
				+ \int_{\Hradius}^{\Hradius+\sqrt\Hradius}C\,s^{\frac{1+\eta}2}\Vert D_{\nu}\outsc\Vert_{\Lp^1(\M<s>)} \d s \\
			& + \int_\Hradius^{\Hradius+\sqrt\Hradius} s^{\frac{\eta-1}2}\vert\meanrnu<s>-1\vert \d s + \frac{C\,\oc(\Hradius)}{\Hradius^{\frac{1-\eta}2}} \\
 \le{}& \int_{\Hradius}^{\Hradius+\sqrt\Hradius}\frac{C\,\oc(s)}{s^{1-\frac\eta2}}(1+\Vert\trans{\rnu<s>}\Vert_{\Lp^\infty(\M<s>)}) \d s
				+ \frac{C\,\oc(\Hradius)}{\Hradius^{\frac{1-\eta}2}} \\
 \le{}& C\,\oc(\Hradius)\,(1+\sup\lbrace\Vert\trans\rnu\Vert_{\Lp^\infty(\M<s>)}\ \middle|\ 0\le s-\Hradius\le\sqrt\Hradius\rbrace)
\end{align*}
for every $\eta\le1$, where $C$ does not depend on $\eta$ and where we assumed without loss of generality that $\oc$ is monotone decreasing and $\oc(\Hradius)\ge\int_\Hradius^\infty\oc'(s)\d s$.
\end{remark}

\begin{lemma}[The lapse function]\label{lapse_function_control_III_a}
If the assumptions~\ref{Ass_fst}--\ref{Ass_lst} are satisfied, then there are constants $\Hradius_0'=\Cof{\Hradius'_0}[M][\ve][p][\c]$ and $C=\Cof[M][\ve][p][\c]$ such that $\Hradius_0>\Hradius_0'$ implies that the lapse function $\rnu$ of the para\-me\-triza\-tion $\Phi$ of $\mathcal M$ as in Lemma~\ref{lemma__derivatives} satisfies\footnote{In fact, we can replace $\nicefrac\outve2$ by $\outve$, but this would need some additional arguments in the proof. As this is not important for this article and needs some additional, technical steps, we do not prove this here. In the Sobolev setting, we can replace $\Hradius^{{-}\frac\outve2}$ by $\outc''(\Hradius)$ with $\outc''(\Hradius)\to0$ as $\Hradius\to\infty$.}
that for every exponent $\eta\in\interval*{\nicefrac12+\nicefrac{\outve}2}*{1+\nicefrac{3\outve}4}$ and $\delta>0$ there holds
\begin{equation*}
 \forall\,\Hradius^{1-\eta}>\frac C{\Hradius^{\frac{\outve}2}}\ \;\exists\,s\in\interval{\Hradius}{\Hradius+\Hradius^{\frac{1-\outve}2}}:\ \vert\meanrnu<s>-1-\frac{\mHaw<s>}s\vert\notin\interval{\frac{3\delta}{4s^\eta}}{\frac\delta{s^\eta}}. \labeleq{lapse_function_control_III_a__eq}\pagebreak[2]
\end{equation*}
\end{lemma}
\begin{proof}
In order to calculate $\meanrnu$ up to order $\Hradius^{-1}$, we note
\begin{equation*}
	\vert\partial[\Hradius]@{\volume{\M<\Hradius>}} - 8\pi(\Hradius-2\mHaw)\meanrnu\vert
		= \vert\int\H\rnu\d\mug + 8\pi(\Hradius-2\mHaw)\meanrnu\vert
		\le \frac C{\Hradius^{\frac12+\outve}}
\end{equation*}
due to \eqref{lemma__derivatives__dm}, i.\,e.
\begin{equation*}
 \vert\partial[\Hradius]@{(\Hradius^{-2}\volume{\M<\Hradius>})} + \frac{8\pi}\Hradius(1-\meanrnu) \vert \le \frac C{\Hradius^{\frac52+\outve}} \labeleq{compare_M_d__with_mean_u},
\end{equation*}
Now, let us assume that \eqref{lapse_function_control_III_a__eq} is not true, i.\,e.~that
\[ \frac{\delta}{{\Hradius'}^\eta}\ge\pm({}_ {\Hradius'}\mean{\rnu}-1-\frac{\mHaw<\Hradius'>}{\Hradius'})\ge\frac{2\delta}{3{\Hradius'}^\eta} \qquad\forall\,\Hradius'\in\interval{\Hradius}{\Hradius+\Hradius^{\frac12-\frac14\outve}} \]
was true for some fixed sign ${\pm}\in\{{-}1,1\}$ and constants $\eta\in\interval*{\nicefrac12+\nicefrac{\outve}2}*{1+\nicefrac{3\outve}4}$, $\delta\in\R$, and $\Hradius\in\interval{\Hradius_0}{\Hradius_1}$, where $\mHaw<\Hradius'>:=\mHaw(\M<\Hradius>)$. As the other case proceeds in the same way, we assume the sign ${\pm}$ to be positive. Under this assumption, integrating \eqref{compare_M_d__with_mean_u} would give
\begin{align*}
 (\Hradius')^{-2}\volume{\M<\Hradius'>}
 \le{}& \Hradius^{-2}\volume{\M<\Hradius>}
			+ 8\pi\int_{\Hradius}^{\Hradius'}\frac{1-{}_s\meanrnu}s \d s
			+ \frac{C}{\Hradius^{2+\outve}} \\
 \le{}& \Hradius^{-2}\volume{\M<\Hradius>}
			+ 8\pi\int_{\Hradius}^{\Hradius'}\frac{\mHaw<s>}{s^2}+\frac\delta{s^{1+\eta}} \d s
			+ \frac{C}{\Hradius^{2+\outve}}. \labeleq{control_of_u_quer_step}
\end{align*}
Now, a integration by parts and~\eqref{lemma__derivatives__dm} would imply
\begin{align*}
 (\Hradius')^{-2}\volume{\M<\Hradius'>}
 \le{}& \Hradius^{-2}\volume{\M<\Hradius>}
			- 8\pi\left[\frac{\mHaw(\M<s>)}s+\frac{\delta}{\eta\,s^\eta}\right]_{\Hradius}^{\Hradius'} \\
			&	+ 8\pi\int_{\Hradius}^{\Hradius'}\frac1s\partial[\Hradius]<s>@{\mHaw(\M<\Hradius>)} \d s
				+ \frac{C}{\Hradius^{2+\outve}} \\
 \le{}& 4\pi(1 - \frac{2\mHaw(\M<\Hradius'>)}{\Hradius'} - \frac{2\delta}\eta\,(\frac1{{\Hradius'}^\eta}-\frac1{\Hradius^\eta})) \\
			&	+ 16\pi\int_{\Hradius}^{\Hradius'}\frac1s(1+\frac{\mHaw(\M<s>)}s-{}_s\meanrnu) \d s
				+ \frac{C}{\Hradius^{2+\outve}}.
\end{align*}
Again using \eqref{lemma__derivatives__dm} and the assumption on $\eta$, a simple integration gives
\begin{equation*}
 0 \le \frac{{-}\eta}{3\,\Hradius^{\eta+\frac12+\frac14\outve}} + \frac{C}{\Hradius^{\frac32+\outve}} + \frac{C}{\Hradius^{1+\eta+\frac\outve2}},
\end{equation*}
where we used the Taylor series for $s\mapsto s^{-\eta}$ at $s=1$. Note that the constants $C$ do \emph{not} depend on $\eta$ or $\delta$, but only on their range, i.\,e.~on $\outve$. In particular, we conclude
\[ \Hradius^{1-\eta+\frac34\outve} \le \frac C\eta \le 2\,C, \]
where again $C$ does not depend on $\eta$ or $\delta$. Therefore, we have proven the claim.
\end{proof}

\begin{lemma}[The lapse function]\label{lapse_function_control_III}
If the assumptions~\ref{Ass_fst}--\ref{Ass_lst} are satisfied and if\footnote{We can replace $\Hradius^{{-}\outve}$ by $\oc(\Hradius)$ for some function $\oc\in\Ck(\R)$ with $\oc(\Hradius)\to0$ for $\Hradius\to\infty$. In this case, we have to replace $\Hradius^{{-}\frac\outve2}$ and $\Hradius^{{-}\frac\outve4}$ in the claim by $\oc'(\Hradius)$ with the same property, but depending arbitrarily on $\oc$.}
\[ \Vert D_{\nu}\outsc\Vert_{\Lp^1(\M<\Hradius>)} \le \frac C{\Hradius^{\frac32+\outve}}\qquad\forall\,\Hradius>\Hradius_0, \]
then there are constants $\Hradius_0'=\Cof{\Hradius'_0}[M][\ve][p][p'][\c]$ and $C=\Cof[M][\ve][p][p'][\c]$ such that $\Hradius_0>\Hradius_0'$ implies that the lapse function $\rnu$ of the para\-me\-triza\-tion $\Phi$ of $\mathcal M$ as in Lemma~\ref{lemma__derivatives} satisfies\footnote{In fact, we can replace $\nicefrac\outve2$ by $\outve$, but this would need some additional arguments in the proof. As this is not important for this article and needs some additional, technical steps, we do not prove this here.}
\begin{equation*}
 \vert\meanrnu-1-\frac{\mHaw}\Hradius\vert \le \frac C{\Hradius^{1+\frac\outve2}}, \qquad
 \Vert\trans\rnu\Vert_{\Wkp^{3,p}(\M)} \le \frac C{\Hradius^{\frac\outve4-\frac2p}}.
\end{equation*}%
Furthermore, the elements of $\mathcal M$ are pairwise disjoint and the Hawking masses converge, i.\,e.~$\mHaw(\M<\Hradius>)\to:\mass\in\interval{c^{{-}1}\,M^{{-}1}}{c\,M}$.
\end{lemma}
\begin{proof}
Assume
\begin{equation*}
	\exists\,\eta\in\interval*{\frac{1+\outve}2}*{1+\frac{3\outve}4},\ \delta_0>0:\quad\forall\,\Hradius>\Hradius(\eta):\ \vert\meanrnu-1-\frac\mHaw\Hradius\vert\le \frac{\delta_0}{\Hradius^\eta} \labeleq{lapse_function_control_III_a__eq_2}.
\end{equation*}
By \eqref{compare_ut_and_bar_u} and \eqref{d_sigma_u__1}, this implies\footnote{If we replace $\Hradius^{{-}\outve}$ in the assumption by $\oc(\Hradius)$ for some function $\oc(\Hradius)\in\Ck(\R)$ with $\oc(\Hradius)\to0$ for $\Hradius\to\infty$, then $C_{\delta_0}$ can be chosen arbitrary small for sufficiently large $\Hradius$ and fixed $\delta_0$. In particular, we get $\Hradius\,\vert\mean\rnu-1-\frac{\mHaw}\Hradius\vert\to0$ for $\Hradius\to\infty$.}
\[ \Vert\trans\rnu\Vert_{\Lp^\infty(\M)} \le C_{\delta_0}\,\Hradius^{\frac{1-\eta}2}, \qquad
	 \vert\partial[\Hradius]@{\meanrnu<\Hradius>}\vert \le \frac{C_{\delta_0}}{\Hradius^{\min\{\frac\eta2+1+\outve,\frac32+\outve\}}}, \]
where again none of the constants depend on $\eta$ or $\Hradius(\eta)$. Thus, if $\meanrnu-1-\nicefrac\mHaw\Hradius=\nicefrac\delta{\Hradius^{\eta'}}$ was true for some $\Hradius>\Hradius_0$, $\delta\in\R$, and $\eta'(\eta):=\min\{1+\nicefrac{3\outve}4,\nicefrac{(\eta+1)}2+\outve\}$, then we would have
\[ \vert\meanrnu-1-\frac\mHaw\Hradius - \frac \delta{\Hradius^{\eta'}}\vert\le\frac{C_{\delta_0}}{\Hradius^{\eta'+\frac\outve4}} \qquad\forall\,\Hradius'\in\interval{\Hradius}{\Hradius+\Hradius^{\frac12-\frac34\outve}}, \]
where the constant does not depend on $\delta$, $\Hradius(\eta)$, $\eta$, or $\delta_0$. By \eqref{lapse_function_control_III_a__eq}, this implies $\Hradius^{1-\eta+\frac34\outve}<C$. Therefore, if \eqref{lapse_function_control_III_a__eq_2} is true for $\eta$, then it is also true $\eta'(\eta)$ instead of $\eta$ and $\Hradius(\eta'):=\max\{\Hradius(\eta),\Hradius')$ for some $\Hradius'=\Cof{\Hradius'}[\outve][M][p][\c][\delta_0]$ independent of~$\eta$. By \eqref{Lemma__smooth_cover_eq}, \eqref{lapse_function_control_III_a__eq_2} is true for $\eta=\frac12+\outve$ per iteration, we get
\[ \forall\,\eta\in\interval*{\frac{1+3\outve}2}*{1+\frac{3\outve}4},\ \delta_0>0,\ \Hradius^{1-\eta}>\frac{C_{\delta_0}}{\Hradius^{\frac{3\outve}4}}:\quad\vert\mean\rnu-1-\frac{\mHaw}\Hradius\vert\le\frac\delta{\Hradius^\eta}. \]
Thus, we can choose $\Hradius$ independent of $\eta\in\interval*{\nicefrac12+\nicefrac{\outve}2}{1+\nicefrac{3\outve}4}$ and this proves the claim.
\end{proof}

\begin{lemma}[The lapse function]\label{lapse_function_control_IIIb}
If the assumptions~\ref{Ass_fst}--\ref{Ass_lst} are satisfied and if
\[ \Vert\outsc\Vert_{\Lp^1(\M<\Hradius>)} + \Hradius\,\Vert D_{\nu}\outsc\Vert_{\Lp^1(\M<\Hradius>)} \le \oc'(\Hradius)\qquad\forall\,\Hradius>\Hradius_0 \]
for some $\oc'\in\Lp^1(\R)$, then there are constants $\Hradius_0'=\Cof{\Hradius'_0}[M][\ve][p][p'][\c]$ and $C=\Cof[M][\ve][p][p'][\c]$ such that $\Hradius_0>\Hradius_0'$ implies that the lapse function $\rnu$ of the para\-me\-triza\-tion $\Phi$ of $\mathcal M$ as in Lemma~\ref{lemma__derivatives} satisfies
\begin{equation*}
 \Hradius\,\vert\meanrnu-1-\frac{\mHaw}\Hradius\vert \xrightarrow{\Hradius\to\infty} 0, \qquad
 \Vert\trans\rnu\Vert_{\Wkp^{3,p}(\M)} \xrightarrow{\Hradius\to\infty} 0.
\end{equation*}%
Furthermore, the elements of $\mathcal M$ are pairwise disjoint and the Hawking masses converge, i.\,e.~$\mHaw(\M<\Hradius>)\to:\mass\in\interval{c^{{-}1}\,M^{{-}1}}{c\,M}$.
\end{lemma}
\begin{proof}
We see that \eqref{Charac_of_Sobolev_asymp_flat} implies
\begin{align*}\hspace{3em}&\hspace{-3em}
 \vert\Hradius^{\frac{1+\eta}2}(\meanrnu<\Hradius>-1) - \Hradius^{\frac{1+\eta}2}(\meanrnu<(\Hradius+\sqrt\Hradius)>-1) \vert \\
 \le{}& \int_{\Hradius}^{\Hradius+\sqrt\Hradius}\frac{C\,\oc(s)}{s^{1-\frac\eta2}}(1+\Vert\trans{\rnu<s>}\Vert_{\Lp^\infty(\M<s>)}) \d s
				+ \int_{\Hradius}^{\Hradius+\sqrt\Hradius}C\,s^{\frac{1+\eta}2}\Vert D_{\nu}\outsc\Vert_{\Lp^1(\M<s>)} \d s \\
			& + \int_\Hradius^{\Hradius+\sqrt\Hradius} s^{\frac{\eta-1}2}\vert\meanrnu<s>-1\vert \d s + \frac{C\,\oc(\Hradius)}{\Hradius^{\frac{1-\eta}2}} \\
 \le{}& \int_{\Hradius}^{\Hradius+\sqrt\Hradius}\frac{C\,\oc(s)}{s^{1-\frac\eta2}}(1+\Vert\trans{\rnu<s>}\Vert_{\Lp^\infty(\M<s>)}) \d s
				+ \frac{C\,\oc(\Hradius)}{\Hradius^{\frac{1-\eta}2}} \\
 \le{}& C\,\oc(\Hradius)\,(1+\sup\lbrace\Vert\trans\rnu\Vert_{\Lp^\infty(\M<s>)}\ \middle|\ 0\le s-\Hradius\le\sqrt\Hradius\rbrace)
\end{align*}
for every $\eta\le1$, where $C$ does not depend on $\eta$ and where we assumed without loss of generality that $\oc$ is monotone decreasing and $\oc(\Hradius)\ge\int_\Hradius^\infty\oc'(s)\d s$. Therefore, the proof of Lemma~\ref{lapse_function_control_III} can be applied to prove the claim.
\end{proof}

\begin{proof}[Proposition~\ref{Foliation_property__prop_2}]
If \ref{Foliation_property__strong_sc} (or \ref{Foliation_property__strong_Dsc}) is satisfied, then \eqref{compare_ut_and_bar_u}, \eqref{d_sigma_u__1} and the Lemmata~\ref{Lemma__smooth_cover} and \ref{lapse_function_control_III} (or Lemma~\ref{lapse_function_control_IIIb}) imply that \ref{Foliation_property__fol} and \ref{Foliation_property__fol_area} are true for $\eta=\oc(\Hradius)\to0$, where $\oc(\Hradius)\le\Hradius^{{-}\frac\outve2}$ if $\mathcal M\subseteq\regsphere*[\frac32+\ve]<M>\c[p]$ and if \ref{Foliation_property__strong_sc} is satisfied for $\oc'(\Hradius)\le\oc\,\Hradius^{{-}\outve}$.
\end{proof}

\begin{corollary}[Global uniqueness of the CMC-foliation]\label{Uniqueness_of_CMC_foliation}
Let the assumptions~\ref{Ass_fst}--\ref{Ass_lst} be satisfied for two families $\mathcal M_1$ and $\mathcal M_2$ with such that each of them satisfies one of the assumptions \ref{Foliation_property__fol_area}--\ref{Foliation_property__strong_Dsc} and either $\mathcal M\subseteq\regsphere*[\frac32+\ve]<M>\c[p]$ or its Sobolev version, i.\,e.~\eqref{assumptions}.
\end{corollary}
\begin{proof}
If $\mathcal M$ and $\mathcal M'$ are two families of regular spheres satisfying the assumptions of one of the Propositions~\ref{Foliation_property__prop} (with $\eta=\eta(\Hradius)\to0$ for $\Hradius\to\infty$) and~\ref{Foliation_property__prop_2}, then the construction explained in the proof of Theorem~\ref{Charac_of_Sobolev_asymp_flat} gives for each of them one $\Wkp^{1,p}_{\frac12}$\vspace{-.25em}-asymptotically flat coordinate systems $\outx$ and $\outx'$ outside of some compact set $\outsymbol K''\subseteq\outM$ with 
\[ \inf_{\M<\Hradius>} \vert\outx\vert \ge \Hradius - C\,\Hradius^{1-\outve}\quad\forall\,\M<\Hradius>\in\mathcal M, \qquad\quad
  \inf_{\M<\Hradius>'} \vert\outx'\vert \ge \Hradius - C\,\Hradius^{1-\outve}\quad\forall\,\M<\Hradius>'\in\mathcal M. \]
In particular, \cite[Thm~3.1]{bartnik1986mass} implies $\vert\outx'-\text{rot}\circ\outx\vert\le C\,\vert\outx\vert^{\frac12}$ for some rotation $\text{rot}$ of $\R^3$. Thus, \cite[Thm~5.3]{nerz2015CMCfoliation} implies $\M<\Hradius>=\M<\Hradius>'$ for every $\M<\Hradius>\in\mathcal M$ and $\M<\Hradius>'\in\mathcal M'$, i.\,e.~$\mathcal M\subseteq\mathcal M'$ or $\mathcal M'\subseteq\mathcal M$.
\end{proof}

\section{Characterizing other quantities: the linear momentum}\label{Charac_quan}
The results of Section~\ref{section_main_theorem} allow us to redefine other quantities without the use of coordinates. Here, we explain this by taking the example of the ADM-linear momentum $\impuls\in\R^3$ \cite{arnowitt1961coordinate} and $\Ck^2_{\frac12+\outve}$-asymptotically flat manifolds. These results can also be used for other quantities and in the setting of $\Wkp^{3,p}_{\frac12}$~asymp\-to\-tically flat manifolds, see \cite{bartnik1986mass} and Theorem~\ref{Charac_of_Sobolev_asymp_flat}.\pagebreak[2]

For this purpose, let us briefly recall the definition of ADM-linear momentum.
\begin{definition}[ADM-linear momentum]
Let $(\outM,\outg*,\outzFund*,\outmomden*,\outenden,\outx)$ be a \emph{$\Ck^2_{\frac12+\outve}$-asymptotically flat initial data set}, i.\,e.\ $(\outM,\outg*,\outx)$ is a $\Ck^2_{\frac12+\outve}$-asymptotically flat Riemannian manifold, the \emph{energy density} $\outenden$ satisfies the constraint equation $\outenden=\nicefrac12\,(\outsc-\vert\outzFund\vert^2-\outH^2)$, the \emph{exterior curvature} $\outzFund*$ decays sufficiently fast $\vert\outzFund*\vert\le\nicefrac{\oc}{\vert\outx\vert^{\outve+\nicefrac32}}$, and the \emph{momentum density} $\outmomden$ satisfying the constraint equation $\outmomden*=\outdiv(\outH*\,\outg*-\outzFund)$ is integrable, i.\,e.\ $\outmomden*\in\Lp^1(\outM)$. The \emph{ADM-linear moment} of $(\outM,\outg*,\outzFund*,\outmomden*,\outenden,\outx)$ is defined by\vspace{-.2em}
\[ \impuls_i := (8\pi)^{-1} \lim_{\rradius\to\infty}\int_{\sphere^2_\rradius(0)} \outH\,\nu_i-\outzFund*(\nu,e_i) \d\mug, \]
where $\outH:=\outtr\,\outzFund$ denotes the mean curvature of $(\outM,\outg*)$ \cite{arnowitt1961coordinate}.\pagebreak[3]
\end{definition}
It is well-known that the linear momentum of a $\Ck^2_{\frac12+\outve}$-asymptotically flat initial data set is always well-defined. This can be seen using the Gau\ss\ divergence theorem and the assumption on $\outmomden*$. In this representation, the linear momentum is interpreted as tangential vector (of $\outM$) at infinity. We see that this definition depends on the coordinate system, but it is well-known that it transforms correctly under coordinate changes \cite{chrusciel1988invariant}.\smallskip

Now, let us reinterpret the ADM-linear momentum as a \emph{function} on $\outM$ (well-de\-fined near infinity): define the \emph{ADM-linear momentum function} to be $\outg*(\impuls,\nu)$, where the vector field $\nu$ is characterized by its restricted to the CMC-leaves being the outer unit normal vector field of the corresponding leaf, i.\,e.\ $\nu|_{\M<\Hradius>}=\nu<\Hradius>$ for each mean curvature radius $\Hradius$ and the corresponding CMC-leaf $\M<\Hradius>$. Here, we identified the tangent vector near infinity $\impuls\in\R^3$ with the constant vector field $\pullback\outx\impuls$. Note that the restriction of this function to a CMC-leaf $\M<\Hradius>$ is (asymptotically as $\Hradius\to\infty$) an eigenfunction of the (negative) Laplace operator on $\M<\Hradius>$ with eigenvalue $\nicefrac2{\Hradius^2}$, i.\,e. $\outg*(\impuls,\nu)-\outg*(\impuls,\nu)\trans{}\to0$ on $\M<\Hradius>$ for $\Hradius\to0$. This means that this function (asymptotically) lays within a three-dimensional function space which is geometrically characterized.

As motivation for the above interpretation, we recall that CMC-foliations of $\Ck^2_{\frac12+\outve}$-asymp\-to\-tic\-ally flat initial data sets (asymptotically) evolve in time (under the Einstein equations) by a shift with lapse function asymptotically equal to the quotient of the ADM-linear momentum function and the Hawking mass of the leaf \cite{nerz2013timeevolutionofCMC}. Let us briefly explain this: assume that we have a temporal foliation $\{(\outM[t],\outg[t]*)\}_{t\in I}$ by $\Ck^2_{\frac12+\outve}$-asymptotically flat initial data sets of a space-time $(\uniM,\unig*)$ which satisfies the Einstein equations with respect to an asymptotically vanishing energy-momentum tensor and let $\{\M[t]<\Hradius>\}_\Hradius$ denote the corresponding CMC-foliations of one of these time-slices $(\outM[t],\outg[t]*)$. For every parametrization $\varphi<\Hradius>:I\times\sphere^2\to\uniM$ of the time evolution of one of these CMC-leaves $\{\M<\Hradius>[t]\}_t$, i.\,e.\ $\varphi<\Hradius>(t,\sphere^2)=\M<\Hradius>[t]$, we can split its derivative $\partial*_t(\varphi<\Hradius>)$ in a part $\tv[t]$ orthogonal to $\outM[t]$, which is a priori given by the temporal foliation, a part $\IndexSymbol X<\Hradius>[t]$ tangential to the CMC-leaf $\M<\Hradius>[t]$, which depends on the specific parametrization we have chosen, and a part $\IndexSymbol N<\Hradius>[t]=\rnu<\Hradius>[t]\,\nu<\Hradius>[t]$ tangential to the time-slice $\outM[t]$ but orthogonal to the CMC-leaf $\M<\Hradius>[t]$, which characterizes the evolution of the surface $\M<\Hradius>[t]$.
The result cited above means that the latter part is (asymptotically) characterized by the ADM-linear momentum function $\outg*(\impuls,\nu<\Hradius>)$, more precisely $\rnu<\Hradius>[t]$ is (asymptotically) equal to $\nicefrac{\outg*(\impuls,\nu)}{\mHaw<\Hradius>[t]}$ (on $\M<\Hradius>[t]$).\smallskip

As the part $\rnu<\Hradius>[t]\,\nu<\Hradius>[t]$ is given by the \emph{geometry} of $(\outM[t],\outg[t]*)$, this means that the ADM-linear momentum function is (asymptotically) characterized by a \emph{geometric} quantity~--~by a function on $\M<\Hradius>[t]$. Now, we define a (new) \emph{geometric} linear momentum function using this quantity.\footnote{Note that we do not use the exact function $\rnu<\Hradius>[t]$, because we want the linear momentum to be completely characterized by the data of one initial data set~--~as it is true for the ADM-linear momentum~--~and not by data of the temporal foliation.}
\begin{definition}[CMC-linear momentum]
Let $(\outM,\outg*,\outzFund*,\outmomden*,\outenden)$ be a three-dimensional initial data set, $\mathcal M:=\{\M<\Hradius>\}_{\Hradius>\Hradius_0}$ be a family of regular hypersurfaces satisfying the assumptions of Theorem~\ref{Suff_ass_for_asymp_flat} or \ref{Charac_of_Sobolev_asymp_flat}. If $\outzFund*$ and $\outmomden*$ are continuous, then the \emph{CMC-linear momentum} $\impulsf\in\Ck(\outM)$ is defined by
\[ \left.\impulsf\right|_{\M<\Hradius>} =: \impulsf<\Hradius> = \frac{\Hradius^2}6\trans{\Big(\Hradius\,\div<\Hradius>(\outzFund(\nu<\Hradius>,\cdot))-\Hradius\,\outmomden*(\nu<\Hradius>)+\tr<\Hradius>\,\outzFund*\,\Big)\!} \quad\forall\,\Hradius>\Hradius_1, \]
where $\Hradius_1:=\sup\big\{\Hradius_0'>\Hradius\,:\,\{\M<\Hradius>\}_{\Hradius>\Hradius_0'}\text{ is a foliation of its union}\big\}$ and $\div<\Hradius>$, $\nu<\Hradius>$, and $\tr<\Hradius>$ denote the two-dimensional divergence, the outer unit normal, and the two-dimensional trace with respect to $\M<\Hradius>\hookrightarrow(\outM,\outg*)$, respectively.\footnote{Note that the addend $\outmomden*(\nu<\Hradius>)$ in the definition of the CMC-linear momentum is~--~compared to the ADM-linear momentum~--~a correction term, adapting the CMC-linear momentum to the CMC-foliation, see below.}
\end{definition}
Note that $\trans{(\div(\outzFund(\nu<\Hradius>,\cdot)))}$ is a well-defined continues function even if $\outzFund*$ is only continuous. To see this, we have to understand $\div(\outzFund(\nu<\Hradius>,\cdot))$ weakly, i.\,e.\ use an integration by part in~\eqref{definition_trans_f}.
\begin{theorem}[Characterization of the ADM-linear momentum]
Let $(\outM,\outg*,\outzFund*,\outmomden*,\outenden)$ be a three-dimensional initial data set, $\mathcal M:=\{\M<\Hradius>\}_{\Hradius>\Hradius_0}$ be a family of regular hypersurfaces satisfying the assumptions of Theorem~\ref{Suff_ass_for_asymp_flat} or \ref{Charac_of_Sobolev_asymp_flat}. If the exterior curvature $\outzFund*$ and the momentum density $\outmomden*$ are continuous, then the CMC-linear momentum $\impulsf\in\Ck(\outM)$ is well-defined outside a compact set $\outsymbol K\subseteq\outM$. If $(\outM,\outg*,\outzFund*,\outmomden*,\outenden,\outx)$ is a $\Ck^2_{\frac12+\outve}$-asymptotically flat initial data set, then $\impulsf$ characterizes the ADM-linear momentum $\impuls=(\impuls^1,\impuls^2,\impuls^3)\in\R^3$, i.\,e.\
\[ \lim_{\Hradius\to\infty}\frac3{4\pi\Hradius^3}\int\impulsf<\Hradius>\,\outx_i \d\mug = \impuls^i \qquad\forall\,i\in\{1,2,3\}. \]
More precisely, there exists a constant $C_p=\Cof{C_p}[\mass][\outve][\oc][p]$ such that
\[ \Vert\impulsf<\Hradius> - \impuls^i\,\nu<\Hradius>_i\Vert_{\Wkp^{1,p}(\M<\Hradius>)} \le C_p\,\Hradius^{\outve-\frac2p} \qquad\forall\,p\in\interval*1\infty \]
where $\nu<\Hradius>_i$ denotes the components of the outer unit normal of $\M<\Hradius>\hookrightarrow(\outM,\outg*)$ with respect to $\outx$.
\end{theorem}
\begin{proof}
Using the results of Section \ref{Regularity_of_the_hypersurfaces}, this is true due to an integration by parts -- see \cite{nerz2013timeevolutionofCMC} for a detailed proof of this inequality.
\end{proof}
As explained above, the CMC-spheres (asymptotically) evolve in time (under the Einstein equations) by $\impulsf<\Hradius>[t]\,\mHaw(\M<\Hradius>[t])^{{-}1}\,\nu<\Hradius>[t]$\vspace{-.1em} and this is true in a pointwise sense. The corresponding error term between the lapse function $\rnu<\Hradius>[t]$ of the evolution of $\M<\Hradius>[t]$ in time direction $t$ and $\impulsf<\Hradius>[t]$ is of order $\Hradius^{{-}\ve}$ and this is true for $\vert\hspace{.05em}\outzFund*\vert\le\oc\Hradius^{{-}\frac32-\outve}$ and $\outmomden*\in\Lp^1(\outM)$. For the ADM-linear momentum this is only true with the additional assumption $\vert\outmomden*\vert\le C\,\Hradius^{{-}3-\ve}$ \cite{nerz2013timeevolutionofCMC}. 
Thus, this definition of linear momentum seems better adapted to the evolution of the CMC-surfaces in time -- the reason for this is the additional correction term $\outmomden*(\nu<\Hradius>)$ in the definition of the CMC-linear momentum.

\appendix
\section{Spheres with \texorpdfstring{$\Lp^p$}{Lp}-almost constant Gau\ss\ curvature}
\expandafter\def\csname @currentlabel\endcsname{\Alph{section}}
\label{Regularity_Calabi_energy}\gdef\thetheorem{\Alph{section}.\arabic{theorem}}%
The aim of this section is to prove a $\Wkp^{2,p}$-regularity of the Gau\ss\ curvature of the sphere, i.\,e.\ there exists a conformal parametrization with conformal factor $\Wkp^{2,p}$-close to $1$ if the Gau\ss\ curvature of a metric on the Euclidean sphere is in $\Lp^p$-close to $1$, where $p\in\interval1\infty$. We note that the same result is well-known if the Gau\ss\ curvature is \emph{pointwise} bounded away from zero and infinity, see for example \cite[Chap~2]{christodoulou1993global}. However, the author is not aware of a corresponding result in $\Lp^p$-spaces. Furthermore, we should note that we can not hope that \emph{every} conformal factor on the sphere is close to a constant if its Gau\ss\ curvature is close to a constant. The reason for this lays in the action of the M\"obius group -- for more information, we refer to \cite[Rem.~7]{christodoulou1993global}, \cite{struwe2002curvature}, and the citations therein. Note that besides the explained main result (Theorem~\ref{Lp_Reg_gauss_curv}), two intermediate results (Propositions~\ref{Conc_pts__vanash_Calabi_en} and \ref{Conc_pts__vanash_Calabi_en__resc}) are interesting for themselves.\pagebreak[1]

The scaling argument used in the proof of Proposition~\ref{Conc_pts__vanash_Calabi_en} was suggested to the author by Simon Brendle \cite{brendle2014communication}. Furthermore, the first part of the proof of Proposition~\ref{Conc_pts__vanash_Calabi_en__resc} is analog to Struwe's proof of \cite[Thm~3.2]{struwe2002curvature} (Theorem~\ref{struwe2002curvature}).\pagebreak[3]

Now, let us state the main result.
\begin{theorem}[\texorpdfstring{$\Wkp^{2,p}(\sphere^2)$}{W2}-regularity of the Gau\ss\ curvature\texorpdfstring\relax{ on the sphere}]\label{Lp_Reg_gauss_curv}
For each $p\in\interval1\infty$, $q\in\interval1\infty$ and $\delta\in\interval0{4\pi}$, there exist constants $C=\Cof[p][q][\delta]$ and $\ve=\Cof{\ve}[\delta][p]$ with the following property: If a metric $\g*$ of the Euclidean unit sphere $\sphere^2$ satisfies\vspace{-.25em}
\[ \confmug(\sphere^2)=(\int_{\sphere^2}\d\confmug)\, \in \interval\delta{8\pi-\delta}, \qquad\int_{\sphere^2}\vert\confgauss-1\vert^p\d\confmug \le \ve, \]
where \pagebreak[1]$\confgauss$ and $\confmug$ are the Gau\ss\ curvature and the measure on the sphere $\sphere^2$ with respect to $\g*$, respectively, then there exists a conformal parametrization $\varphi:\sphere^2\to\sphere^2$ with
\[ \Vert\conf\Vert_{\Wkp^{2,q}(\sphere^2,\sphg)} \le C\,\Vert\confgauss-1\Vert_{\Lp^q(\sphere^2,\g*)}.\vspace{-.25em} \]
Here, $\sphg*$ denotes the standard metric of the Euclidean unit sphere $\sphere^2$ and $\conf\in\Wkp^{2,p}(\sphere^2)$ is the corresponding conformal factor, i.\,e.\ $\varphi^*\g*=\exp(2\,\conf)\,\sphg*$.\pagebreak[3]
\end{theorem}
This theorem should also be compared with DeLellis-M\"uller's famous result that every closed hypersurface $\M\hookrightarrow\R^3$ with $\Lp^2$-small trace free part $\zFundtrf$ of the second fundamental form $\zFund$, i.\,e.\ $\Vert\zFundtrf\Vert_{\Lp^2(\M)}$ is small, is a sphere and has a conformal parametrization $\varphi:\sphere^2_r(\centerz)\to\M$ such that the conformal factor is in $\H(\sphere^2)$ and $\varphi-\id$ is in $\H^2(\sphere^2;\R^3)$ controlled, \cite{DeLellisMueller_OptimalRigidityEstimates}. Note that the latter control cannot be achieved in our setting, as we do not assume that $\M$ is a hypersurface in $\R^3$ and therefore there exists no canonical map $\id$ with which we can compare the conformal parametrization.
\begin{remark}[Characterization of the conformal map]
By analyzing the proof of Theorem~\ref{Lp_Reg_gauss_curv}, we immediately see that the conformal parametrization $\varphi$ is (up to rotation) characterized by satisfying the balancing condition~\eqref{Lp_Reg_gauss_curv__balanced}, i.\,e.~by
\begin{equation*}
(\pullback\varphi\mug)\big(\sphere^2\cap\{x_i\ge0\}\big) = (\pullback\varphi\mug)\big(\sphere^2\cap\{x_i\le0\}\big) \qquad\forall\,i\in\{1,2,3\}.
\tag{\ref{Lp_Reg_gauss_curv__balanced}'}
\end{equation*}
\end{remark}
\begin{remark}[Stability]
Note that Theorem~\ref{Lp_Reg_gauss_curv} implies an additional stability result: if $\g[1]*$ and $\g[2]*$ are two metrics satisfying the assumptions of Theorem~A.1, then we can choose the conformal parametrizations $\varphi[i]:\sphere^2\to\sphere^2$ such that
\[ \Vert\conformalf[i]\Vert_{\Wkp^{2,q}(\sphere^2,\sphg*)} \le C\,\Vert\confgauss[i]-1\Vert_{\Lp^p(\M)},\quad
	 \Vert\conformalf[1]-\conformalf[2]\Vert_{\Wkp^{2,q}(\sphere^2,\sphg*)} \le C\,\Vert\confgauss[1]-\confgauss[2]\Vert_{\Lp^p(\M)}, \]
where $\conformalf[i]\in\Wkp^{2,p}(\sphere^2,\sphg*)$ is the conformal factor, i.\,e.\ $\pullback\varphi{\g[i]*}=\exp(2\,\conformalf[i])\,\sphg*$. We can prove this corollary by repeating the last step of the proof of the theorem, i.\,e.~by linearization of the Gau\ss\ cuvature map around $\g[1]*$.
\end{remark}
\smallskip

As main tools for the proof of Theorem~\ref{Lp_Reg_gauss_curv}, we use Brezis-Merle's famous inequality \cite[Thm~1]{brezis1991uniform} (see Theorem~\ref{Brezis_Merle}) and Chen-Li's classification theorem \cite[Thm~1]{chen1991}: Every solution $\conft$ of
\[ {-}\laplace \conft = \exp(2\,\conft)\quad\text{in }\R^2,\qquad\int_{\R^2}\exp(2\,\conft)\d x < \infty \]
is given by $\conft(x)=\ln(\nicefrac{2\lambda}{(\lambda^2+\vert x-x_0\vert^2)})$ for some constant $\lambda>0$ and some point $x_0\in\R^2$. As an intermediate result, we prove a qualitative version of this characterization, Proposition~\ref{Conc_pts__vanash_Calabi_en__resc}: If the Gau\ss\ curvature $\confgauss<n>:={-}\exp({-}2\,\conft<n>)\,\laplace\conft<n>$ of a sequence of conformal factors converges in $\Lp^p_{\text{loc}}(\R^2)$ to $1$, the corresponding volumes $\int\exp(2\,\conft<n>)\d x$ are uniformly bounded, and they satisfy a non-concentration assumption, then $\conft<n>$ converges in $\Wkp^{1,q}_{\text{loc}}(\R^2)$ to the above $\conft$, where $q=\frac{2p}{2-p}$ for $p<2$ and $q\in\interval1\infty$ arbitrary for $p\ge2$.

One of the main ideas of the proof of Theorem~\ref{Lp_Reg_gauss_curv} is to prove that any sequence of conformal factors is bounded in $\Wkp^{2,p}(\sphere^2)$ or the mass (the area) of a subsequence is concentrated at some points if the corresponding Gau\ss\ curvatures converge in $\Lp^p(\sphere^2)$ to a constant. Here, concentration at a point means that the corresponding measures $\confmug<n>$ of the subsequence converge to a measure with non-trivial point measure at this point. Struwe proved the corresponding theorem in the context of the Calabi flow assuming only uniform boundedness of the Calabi energy of the sequence, however under our assumptions, we get a stronger control of the Dirac measures in the case of mass concentration \cite[Thm~3.2]{struwe2002curvature}. Before we cite Struwe's theorem, we recall that for a given closed Riemannian manifold $(\M,\g*)$ with constant Gau\ss\ curvature $\confgauss\equiv\confgauss_0$ and a conformal equivalent metric $\g'=\exp(2\,\conft')\,\g*$, the \emph{Calabi energy} of $\g*'$ (respectively $\conft'$) with respect to $\g*$ is defined by
\[ \text{Cal}_{\g*}(\conft') := \text{Cal}_{\g*}(\g*') := \int_{\sphere^2} \vert\confgauss'-\confgauss\vert^2\d\confmug = \Vert\confgauss'-\confgauss\Vert_{\Lp^2(\sphere^2,\g*)}^2, \]
where $\confgauss'$ denotes the Gau\ss\ curvature of $\M$ with respect to $\g*'$.\pagebreak[3]

\begin{theorem}[Concentration compactness {\cite[Thm~3.2]{struwe2002curvature}}]\label{struwe2002curvature}
Let $(\M,\g*)$ be a closed Riemannian manifold with constant Gau\ss\ curvature $\confgauss$ and let $\conf<n>$ be a sequence of conformal factors with uniformly bounded Calabi energy and unit volume $\confmug<n>(\M)=\mu(\M)=1$, where $\confmug<n>=\exp(2\,\conf<n>)\confmug$ and $\confmug$ is the measure induced on $\M$ by $\g$. Then either the sequence $\conf<n>$ is bounded in $\Hk^2(\sphere^2,\sphg*)$ or there exist points $x_1,\dots,x_L\in\M$ and a subsequence $k_n\to\infty$ such that
\[ \varrho_R(x_l):=\liminf_{n\to\infty}\int_{B_R(x_l)} \d\confmug<k_n> \ge 2\pi\qquad\forall\,R>0,\,l\in\{1,\dots,L\}. \]
Moreover, there holds
\[ 2\pi L\le\limsup_{n\to\infty}(\int_{\M}\confgauss<n>^2\d\confmug<n>)^{\frac12} < \infty \]
and either $\conf<k_n>\to{-}\infty$ as $n\to\infty$ locally uniformly on $\M\setminus\{x_1,\dots,x_L\}$ or $(\conf<k_n>)$ is locally bounded in $\H^2(\M\setminus\{x_1,\dots,x_L\},\g*)$.
\end{theorem}
Now, we strengthen this result in the case of the sphere and $\gauss\to1$ in $\Lp^p$ by proving that the amount $L$ of critical points $\{x_l\}_{l=1}^L$ satisfies in this setting
\begin{equation*} 4\pi L \le 4\pi = (\int_{\M}\confgauss^p\d\confmug)^{\frac1p} = \limsup_{n\to\infty}(\int_{\M}\confgauss<n>^p\d\confmug<n>)^{\frac1p}, \labeleq{Conc_pts__vanash_Calabi_en_eq} \end{equation*}
i.\,e.\ $L=1$. \pagebreak[1]
Let us therefore recall that any Riemannian surfaces is locally conformal equivalent to the plane, i.\,e. we can look at metrics on the Euclidean space $\R^2$ instead of $\sphere^2$ and get a new conformal factor $\conft$ satisfying ${-}\laplace\conft<n>=\confgauss<n>\,\exp(2\,\conft<n>)$.
\begin{proposition}[Concentration point for vanishing curvature error]\label{Conc_pts__vanash_Calabi_en}
Let $p\in\interval1\infty$ be a constant and $\conft<n>$ be a family of smooth conformal factors on $\R^2$. Assume that the corresponding volumes are bounded, i.\,e. 
\[ \limsup_{n\to\infty}\int\exp(2\,\conft<n>)\d x=:\cmu<\infty, \]
and that the Gau\ss\ curvature is $\Lp^p$-locally converging to one, i.\,e. 
\[ \lim_{n\to\infty} \int_K\vert\confgauss<n>-1\vert^p\d\confmug<n> = 0 \qquad\forall\,K\subseteq\R^2\text{ compact,} \]
where $\confmug<n>=\exp(2\,\conft<n>)\d x$ and $\confgauss<n>={-}\exp({-}2\,\conft<n>)\;\laplace\conft<n>$ denotes the measure and the Gau\ss\ curvature with respect to $\exp(2\,\conft<n>)\,\eukg*$, respectively. Then each point $x\in\R^2$ with $\varrho(x)>0$ satisfies $\varrho(x)\ge4\pi$, where
\[ \varrho(x) := \inf_{R>0} \liminf_{n\to\infty}\int_{B_R(x)} \exp(2\,\conft<n>)\d\,x. \]
In particular, there are not more than $\nicefrac{\cmu}{4\pi}$ many points $x\in\R^2$ with $\varrho(x)>0$.
\end{proposition}
We see that this implies \eqref{Conc_pts__vanash_Calabi_en_eq}. We will prove Proposition~\ref{Conc_pts__vanash_Calabi_en} by a blowup argument in a neighborhood of any point $x\in\R^3$ with $\varrho(x)>0$, i.\,e.\ we rescale $\conft<n>$ around $x$ by factors $R_n\to\infty$ to functions $\confth<n>$ such that there is a fixed radius $r$ with $\int_{B_r(x)}\exp(2\,\confth<n>)\d\,x=\nicefrac{\varrho(x)}2$ and then prove that this already implies that there is a fixed radius $R$ with $\int_{B_r(x)}\exp(2\,\confth<n>)\d\,x\approx4\pi$ (for sufficiently large $n$) -- note that the radius is fixed in the scaled image, i.\,e.\ this implies $\varrho(x)\approx4\pi$.\pagebreak[3]

Let us start by the last argument, i.\,e.\ we assume that we already scaled the metric. This result should be understood as a qualitative analog of Chen-Li's classification theorem \cite[Thm~1]{chen1991}.
\begin{proposition}[Concentration point for vanishing curvature-error -- rescaled]\label{Conc_pts__vanash_Calabi_en__resc}
Let $p\in\interval1\infty$ be a constant and $\confth<n>\in\Ck^2(\R^2)$ be a sequence of functions $\R^2$ with
\begin{align*}
 \limsup_{n\to\infty} \int_{\R^2} \exp(2\,\confth<n>) \d x =:{}& \cmu < \infty \labeleq{Conc_pts__vanash_Calabi_en__resc__en_ass}, \\
 \lim_{n\to\infty}\int_K\vert\confgauss<n>-1\vert^p\,\exp(2\,\confth<n>) \d x ={}& 0 \qquad\forall\,K\subseteq\R^2\text{ compact,} \labeleq{Conc_pts__vanash_Calabi_en__resc__Gauss_ass}
\end{align*}
where $\confgauss<n>:={-}\exp({-}2\,\confth<n>)\,\laplace\confth<n>$.\pagebreak[1] If there are constants $\ve_0\in\interval0{2\pi}$, $r>0$, and a sequence of positive numbers $S_n>0$ converging to infinity, i.\,e.\ $S_n\to\infty$ for $n\to\infty$, such that
\[ \int_{B_r(x)} \exp(2\,\confth<n>) \d x \le \int_{B_{r}(0)} \exp(2\,\confth<n>) \d x = \ve_0 \qquad\forall\,\vert x\vert\le S_n, \]
then
\[ \lim_{n\to\infty} \confth<n> = \confth := \ln(\frac{2\lambda}{\lambda^2+\vert x\vert^2})\quad\text{\normalfont in } \Wkp^{1,q}_{\text{loc}}(\R^2), \]
where $\lambda:=\sqrt{\frac{4\pi-\ve_0}{\ve_0}}\,r$ and $q:=\frac{2p}{2-p}$ for $p<2$ and $q\in\interval2\infty$ arbitrary for $p\ge2$.\pagebreak[3]
\end{proposition}
We see that this in particular implies the following corollary.
\begin{corollary}\label{Conc_pts__vanash_Calabi_en__cor}
Let $\confth<n>\in\Ck^2(\R^2)$ be a sequence of functions satisfying \eqref{Conc_pts__vanash_Calabi_en__resc__en_ass} and \eqref{Conc_pts__vanash_Calabi_en__resc__Gauss_ass}. If there exist constants $r>0$ and $\ve_0\in\interval0\pi$ with
\[ \int_{B_r(x)} \exp(2\,\confth<n>) \d x \le \int_{B_{r}(0)} \exp(2\,\confth<n>) \d x \in\interval*{\ve_0}*{2\pi-\ve_0} \qquad\forall\,\vert x\vert\le S_n \]
for some sequence of constants $S_n$ converging to infinity, then
\[ \liminf_{n\to\infty} \int_{\R^2} \exp(2\,\confth<n>) \d x \ge 4 \pi. \]
\end{corollary}
\begin{proof}[Proof of Corollary~\ref{Conc_pts__vanash_Calabi_en__cor}]
Assume $\confth<n>':=\confth_{k_n}$ is the area minimizing sequence, i.\,e. $\liminf_n\int_{\R^2}\exp(2\,\confth<n>)\d x=\lim_n\int_{\R^2}\exp(2\,\confth<n>') \d x$. Then there is a subsequence $\confth<n>'':=\confth_{l_n}'$ such that $\int_{B_{r}(0)} \exp(2\,\confth<n>'') \d x$ converges to some $\ve_0'\in\interval*{\ve_0}*{2\pi-\ve_0}$. Thus, Proposition~\ref{Conc_pts__vanash_Calabi_en__resc} implies $\confth<n>''\to \confth''=\ln(\nicefrac{2\lambda}{(\lambda^2+\vert x\vert^2)})$ in $\Wkp^{1,Q}_{\text{loc}}(\R^2)$, where $\lambda=\sqrt{\nicefrac{4\pi-\ve_0'}{\ve_0}}\,r$. Thus, for any $\ve>0$ there exist $N>0$ and $R>0$ such that
\[ \int_{B_R(0)}\exp(2\,\confth<n>'')\d x \ge \int_{B_R(0)}\exp(2\,\confth'') \d x - \ve \ge \int_{\R^2}\exp(2\,\confth'')\d x - 2 \ve = 4\pi - 2\ve \]
for any $n\ge N$, i.\,e.\ $\liminf_n\int_{\R^2}\exp(2u_n)\d x = \lim_n\int_{\R^2}\exp(2\,\confth<n>')\d x\ge4\pi$.
\end{proof}
The second part of the proof of Proposition~\ref{Conc_pts__vanash_Calabi_en__resc} uses Chen-Li's classification theorem \cite[Thm~1]{chen1991} for metrics of constant Gau\ss\ curvature in $\R^2$, while the first part of the proof is analog to Struwe's proof of \cite[Thm~3.2]{struwe2002curvature}. We repeat it nonetheless for the readers convenience. As Struwe, we need the following result of Brezis-Merle \cite[Thm~1]{brezis1991uniform}\footnote{Here, we state the same version of \cite[Thm~1]{brezis1991uniform} as Struwe \cite[Thm~3.1]{struwe2002curvature} which is a slight modification of the original theorem.}.
\begin{theorem}[{\cite[Thm~1]{brezis1991uniform}}]\label{Brezis_Merle}
Let $u$ be a distribution solution to the equation
\[ {-}\laplace u = f \quad\text{in }B_1(0),\qquad u\equiv0\quad\text{in }\partial*B_1(0)=\sphere^2_1(0), \]
where $f\in\Lp^1(B_1(0))$. For every $p<4\pi\,\Vert f\Vert_{\Lp^1(B_1(0))}^{{-}1}$, there exists a constant $C$ with
\[ \int_{B_1(0)}\exp(p\,\vert u\vert) \d x \le C\,(4\pi-p\,\Vert f\Vert_{\Lp^1(B_1(0))})^{-1}. \]
\end{theorem}
\begin{proof}[Proof of Proposition~\ref{Conc_pts__vanash_Calabi_en__resc}]
We start by proving that the sequence $\confth<n>$ is bounded in $\Wkp^{2,p}_{\text{loc}}(\R^2)$. As mentioned above, the proof of this part is analog to Struwe's proof of \cite[Thm~3.2]{struwe2002curvature}. Let $N\gg1$ be a constant, define $R:=R_N:=\inf_{n\ge N} S_n$, and note $R_N\to\infty$ for $N\to\infty$. Let $n\ge N$ be arbitrary and suppress the corresponding index $n$. Furthermore, let $x\in B_R(0)$ be arbitrary and choose functions $\confth^0,\confth^h\in\Ck^2(B_r(x))$ with $\confth=\confth^0+\confth^h$ such that
\[ \left\{\quad\begin{aligned} {-}\laplace\confth^h ={}& 0,& {-}\laplace\confth^0 ={}& {-}\laplace\confth &&\text{in } B_r(x) \\
	\confth^h ={}& u,& \confth^0 ={}& 0 &&\text{in } \partial* B_r(x) \end{aligned} \right., \]
i.\,e.\ $\confth^h$ is the harmonic part of $\confth$ and $\confth^0$ is the rest having boundary value $0$. We see that
\[ \int_{B_r(x)}\vert\laplace\confth^0\vert\d x
		\le \int_{B_r(x)}\exp(2\,\confth)\d x + \frac{2\pi - \ve_0}2
		\le \frac{2\pi+\ve_0}2 < 2\pi \]
if $N$ is so large that $\int_{B_{R+r}(0)}\exp(2\,\confth)\,\vert\confgauss-1\vert\d x \le \pi-\nicefrac{\ve_0}2$ holds for every $n\ge N$. In particular, we can choose $N$ independently of $x\in B_R(0)$. Fix a $p'\in\interval1{\nicefrac{4\pi}{(2\pi+\ve_0)}}$. Brezis-Merle's result, Theorem~\ref{Brezis_Merle}, implies
\[ \int_{B_r(x)} \exp(2p'\,\vert \confth^0\vert) \d x \le C, \]
where the constant $C$ depends on $q$, $r$, and $\ve_0$. In the following, we do not distinguish between constants $C$ depending on $q'$, $r$, $\ve_0$, $\cmu$ and $R$. The above implies
\begin{equation*}
	\int_{B_r(x)} \vert\confth^0\vert \d x \le \ln(\int_{B_r(x)}\exp\,\vert\confth^0\vert \d x) \le C \labeleq{Conc_pts__vanash_Calabi_en__resc__est_v}
\end{equation*}
and we already know
\[  2\int_{B_r(x)} \confth \d x \le \ln(\int_{B_r(x)} \exp(2\,\confth) \d x) \le \ln \cmu \le C. \]
Thus, the mean value property of harmonic functions implies for $y\in B_{\nicefrac r2}(x)$
\[ \confth^h(y) = \fint_{B_{\frac r2}(y)} \confth^h \d x \le C, \]
i.\,e. $\confth^h\le C$ in $B_{\nicefrac r2}(x)$. We conclude
\begin{align*}
 \int_B \vert\laplace\confth^0\vert^{q'} \d x
	\le{}& \int_B \exp(\frac{2\,q'}p\confth)\,\vert\confgauss\vert^{q'}\,\exp(2\,q'\,\frac{p-1}p\,\confth) \d x \\
	\le{}& (\int_B \exp(2\,\confth)\,\vert\confgauss\vert^p \d x)^{\frac{q'}p}\,(\int_B \exp(2\,q'\,\frac{p-1}{p-q'}\,\confth) \d x)^{\frac{p-q'}p} \\
	\le{}& C,
\end{align*}
where $B:=B_{\nicefrac12}(x)$ and where $q'=\frac{p'\,p}{p+p'-1}$, i.\,e.\ $\nicefrac{(p-1)}{(p-q')}=\nicefrac{p'}{q'}$ and $q'\in\interval1p$.\pagebreak[1] Hence, \eqref{Conc_pts__vanash_Calabi_en__resc__est_v} implies $\Vert\confth^0\Vert_{\Wkp^{2,q'}(B)}\le C$. In particular, the Sobolev inequalities imply $\confth\le \confth^h + \vert\confth^0\vert \le C$ in $B=B_{\nicefrac r2}(x)$ and we therefore get
\[ \int_B\vert\laplace\confth\vert^p \d x = \int_B\vert\confgauss\vert^p\,\exp(2\,p\,\confth) \d x \le C. \]
As $x$ was arbitrary in $B_R(0)$, we conclude with the above estimates and the regularity of the Laplace operator
\[ \max_{B_R(0)} \confth\le C + \min_{B_R(0)} \confth \le C + \fint_{B_R(0)} \confth \d x \le C + \frac12\ln(\frac{\cmu}{4\pi R^2}) \le C. \]
Assuming without loss of generality that $R>r$, we conclude
\[ 4\pi\exp(2\,\max_{B_R(0)}\confth)\,R^2 \ge \int_{B_R(0)}\exp(2\,\confth) \d x \ge \int_{B_{r}(0)} \exp(2\,\confth)\d x \ge \ve_0 \]
implying
\[ {-}C \le \max_{B_R(0)} \confth \le C + \min_{B_R(0)} \confth \le C, \]
i.\,e.\ $\vert\confth\vert\le C$ in $B_R(0)$ -- note that the constant $C=\Cof[R]$ depends on $R$. Thus, we get the desired uniformly bound, as we proved
\begin{equation*} \Vert \confth<n> \Vert_{\Wkp^{2,p}(B_R(0))}\le \Cof[R],\qquad\lim_{n\to\infty}\Vert\laplace \confth<n> -\exp(2\,\confth<n>)\Vert_{\Lp^p(B_R(0))} = 0. \labeleq{Brezis_Merle__bounded}\smallskip\pagebreak[2]\end{equation*}

By the compactness of the Sobolev embeddings, it is now sufficient that any in $\Wkp^{1,q}_{\text{loc}}(\R^2)$ converging subsequence of $\confth<n>$ converges to $\ln(\nicefrac{2\lambda}{(\lambda^2+r^2)})$ for the fixed constant $\lambda:=\sqrt{\frac{(4\pi-\ve_0)}{\ve_0}}\,r$. Thus, we can assume that $\confth<n>$ converges in $\Wkp_{\text{loc}}^{1,q}(\R^2)$ to some function $\confth\in\Wkp_{\text{loc}}^{1,q}(\R^2)$, i.\,e.
\[ \lim_{n\to\infty}\Vert \confth<n> - \confth \Vert_{\Wkp^{1,q}(B_R(0))} = 0 \qquad\forall\,R>0. \]
In particular, we know $\exp(2\,\confth<n>)\to\exp(2\,\confth)$ locally uniformly and that $\confth$ is locally bounded. By the second inequality in \eqref{Brezis_Merle__bounded}, this implies $\laplace \confth<n>\to{-}\exp(2\,\confth)$ in $\Lp^p_{\text{loc}}(\R^2)$.
Hence, we know that $\laplace\confth={-}\exp(2\,\confth)$ in the $\Lp^2$-weak sense in $B_R(0)$ (for every $R>0$). The regularity of the Laplace operator and the locally boundedness of $\confth<n>$ (see above) therefore implies $\confth\in\Ck^\infty(\R^2)$ and $\laplace\confth={-}\exp(2\,\confth)$ pointwise everywhere. As we furthermore know
\[ \int_{B_R(0)} \exp(2\,\confth)\,\d x \le \limsup_{n\to\infty} \int_{B_R(0)} \exp(2\,\confth<n>) \d x \le \limsup_{n\to\infty} \int_{\R^2} \exp(2\,\confth<n>) \d x = \cmu, \]
we can use the Chen-Li's classification theorem \cite[Thm~1]{chen1991} to conclude that there exist a point $x_0$ and a factor $\kappa>0$ such that
\[ \confth(x) = \ln(\frac{2\kappa}{\kappa^2+\vert x-x_0\vert^2}). \]
We see that $x_0$ is uniquely determined by the fact that for any $R>0$
\[ \int_{B_R(x)} \exp(2\,\confth) \d x \le \int_{B_R(x_0)} \exp(2\,\confth) \d x \qquad\forall\,x\in\R^2. \]
We therefore conclude $x_0=0$ by
\[ \ve_0 \ge \lim_{n\to\infty} \int_{B_{r}(x)} \exp(2\,\confth<n>)\d x = \int_{B_{r}(x)} \exp(2\,\confth) \d x \qquad\forall\,x\in\R^2 \]
and
\[ \ve_0 = \lim_{n\to\infty} \int_{B_{r}(0)} \exp(2\,\confth<n>)\d x = \int_{B_{r}(0)} \exp(2\,\confth)\d x, \]
due to $\exp(2\,\confth<n>)\to\exp(2\,\confth)$ in $\Lp^1_{\text{loc}}(\R^2)$. In particular, we get
\[ \ve_0 = \int_{B_r(0)}\exp(2\,\confth)\d x = \frac{4\pi r^2}{\kappa^2+r^2} \]
implying $\kappa=\lambda$.\pagebreak[3]
\end{proof}

\begin{proof}[Proof of Proposition~\ref{Conc_pts__vanash_Calabi_en}]
Let $x_0\in\R^2$ be a point such that $\ve_1:=\varrho(x_0)>0$ -- without loss of generality $x_0=0$. Now, we want to rescale $\conft<n>$ around some center point $y_n$ near $0$ by a factor $\varrho_n\to\infty$ (for $n\to\infty$) and use Corollary~\ref{Conc_pts__vanash_Calabi_en__cor} on the rescaled functions $\confth<n>(x):=\conft<n>(\varrho_n\,x+y_n)+\ln\varrho_n$. However, we cannot choose $0$ as this center point (for all $n$) as it could be that the \emph{mass} $\exp(2\,\conft<n>)\d x$ is only large around some point $y_n$ \emph{close to} but not (necessarily) \emph{equal to} $0$ and this point $y_n$ could be \lq scaled away\rq\ if we scaled around $0$. Thus, we have to choose the center more carefully.

Now, we want to use the Brendle's scaling argument \cite{brendle2014communication}. As explained, we want to find a center point for the scaling done later and this center point should lay \lq near\rq\ $0$. Thus, we have to fix a neighborhood of $0$ such that the only \lq center\rq\ $y_n$ (see below) of the mass (measure) $\exp(2\,\conft<n>)$ ($\exp(2\,\conft<n>)\d x$) within this neighborhood approximates $0$. If $L\in\N$ is an integer, $r>0$ is a constant, and $y_1,\dots,y_L\in\R^2$ are points with $B_r(y_i)\cap\{y_1,\dots,y_L\}=\{y_i\}$ and $\varrho(y_i)\ge\frac12\,\min\{\pi,\ve_1\}=:\ve_0$ for every $i$, then
\[ \cmu \ge \limsup_{n\to\infty} \int_{\R^2} \exp(2\,\conft<n>) \d x
		\ge \limsup_{n\to\infty} \sum_{i=1}^L \int_{B_r(y_i)} \exp(2\,\conft<n>) \d x
		\ge L\,\ve_0. \]
In particular, we know $L\le\nicefrac{\cmu}{\ve_0}<\infty$. Thus, there is a radius $R>0$ such that every $x\in B_{2R}(0)$ with $x\neq 0$ satisfies $\varrho(x)<\ve_0$. Now, we define the center points $y_n\in B_R(0)\subseteq\R^2$ as one of the points in which $\varrho_n$ is minimal, where
\[ \varrho_n(x) := \inf\lbrace r\in\interval0{2R}\ \middle|\ \int_{B_r(x)}\exp(2\,\conft<n>) \d x \ge \ve_0 \rbrace \qquad\forall\,x\in\R^2,\,\vert x\vert\le R. \]
Note that the minimum of $\varrho_n$ within $\{y\,:\,\vert y\vert\le R\}=:B_R$ exists as $\varrho_n$ is continuous, i.\,e.\ we can choose such a (not necessarily uniquely defined) $y_n$ for every $n\gg1$. As $(y_n)$ is a bounded sequence and every cluster point $y$ of it satisfies $\varrho(y)\ge\ve_0$ and (per definition of $R$) therefore $y=0$ or $2R\le\vert y\vert=\lim_n\vert y_n\vert \le R$, we know $y_n\to 0$ for $n\to\infty$. Furthermore, we know $\varrho_n:=\varrho_n(y_n)\le\varrho_n(0)\to 0$ for $n\to\infty$.

We rescale around $y_n$ with the scaling factor $\varrho_n^{-1}$, i.\,e.\ we define
\[ \confth<n>(x) := \conft<n>(\varrho_n\,x+y_n) + \ln \varrho_n \qquad\forall\,x\in\R^2 \]
and have to check the preliminaries of Corollary~\ref{Conc_pts__vanash_Calabi_en__cor} in order to use it. First, we see that
\[ \int_{\R^2} \exp(2\,\confth<n>) \d x = \int_{\R^2} \varrho^2\,\exp(2\,\conft<n>(\varrho_n\,x+y_n)) \d x = \int_{\R^2} \exp(2\,\conft<n>) \d x \]
implying $\limsup_n\int \exp(2\,\confth<n>) \d x\le\cmu<\infty$. Furthermore, we get
\[ \laplace \confth<n> = \varrho^2 (\laplace \conft<n>)(\varrho_n\,x+y_n) = {-}\confgauss<n>(\varrho_n\,x+y_n)\,\exp(2\,\confth<n>) =: {-}\confgauss<n>'(x)\,\exp(2\,\confth<n>) \]
and
\[ \int_{B_{\varrho_n^{-1}}(0)} \vert\confgauss<n>'-1\vert^p\,\exp(2\,\confth<n>) \d x = \int_{B_1(y_n)} \vert\confgauss<n>-1\vert^p\,\exp(2\,\conft<n>) \d x \xrightarrow{n\to\infty} 0. \]
As $\varrho_n\to0$ for $n\to\infty$, this implies that $\int_K \vert\confgauss<n>'-1\vert^p\,\exp(2\,\confth<n>) \d x\to0$ for any compact set $K\subseteq\R^2$. For the last preliminary, we note that every $x\in B_{\nicefrac1{\sqrt{\varrho_n}}}(0)$ and $x_n:=\varrho_nx+y_n$ satisfies
\begin{align*}
 \int_{B_1(x)} \exp(2\,\confth<n>)\d x
	={}& \int_{B_{\varrho_n}(x_n)} \exp(2\,\conft<n>)\d x \le \int_{B_{\varrho_n(x_n)}(x_n)} \exp(2\,\conft<n>)\d x = \ve_0, \\
 \int_{B_1(0)} \exp(2\,\confth<n>)\d x
	={}& \int_{B_{\varrho_n}(y_n)} \exp(2\,\conft<n>)\d x = \ve_0.
\end{align*}
Thus, we can use Corollary~\ref{Conc_pts__vanash_Calabi_en__cor} for $S_n=\varrho^{\nicefrac{{-}1}2}$ and $r=1$ to conclude that $\confth<n>\to \confth$ in $\Wkp^{1,q}_{\text{loc}}(\R^2)$ with $\int_{\R^2}\exp(2\,\confth)\d x=4\pi$ and $q=\frac{2p}{2-p}>2$ for $p<2$ and $q\in\interval2\infty$ for $p\ge2$. This means for every $\ve>0$ and $\ve'>0$ there is a radius $R<\infty$ with
\begin{align*}
 4\pi \le{}& \int_{B_R(0)} \exp(2\,\confth) \d x + \frac\ve2 \le \int_{B_R(0)} \exp(2\,\confth<n>) \d x + \ve \\
	&= \int_{B_{R\varrho_n}(y_n)} \exp(2\,\conft<n>) \d x + \ve
	\le \int_{B_{\ve'}(0)} \exp(2\,\conft<n>) \d x + \ve
\end{align*}
for every $n\ge N$, where $N$ is so large that $\int_{B_R(0)}\vert\exp(2\,\confth)-\exp(2\,\confth<n>)\vert \d x\le\nicefrac\ve2$ and $\varrho_n\,R + \vert y_n\vert \le \ve'$. As $\ve>0$ was arbitrary, we get $\varrho(x_0)=\varrho(0)\ge4\pi$, as this implies
\[ \liminf_{n\to\infty} \int_{B_\ve'(0)} \exp(2\,\conft<n>) \d x \ge 4\pi \qquad\forall\,\ve'>0.\qedhere \]
\end{proof}

\begin{proof}[Proof of Theorem~\ref{Lp_Reg_gauss_curv}]
Using conformal maps, we can assume
\begin{equation*} \mug(\sphere^2\cap\{x_i\ge0\}) = \mug(\sphere^2\cap\{x_i\le0\}) \in\interval{\frac\delta2}{4\pi-\frac\delta2} \qquad\forall\,i\in\{1,2,3\}, \labeleq{Lp_Reg_gauss_curv__balanced}\end{equation*}
see for example \cite[Lemma~3.4]{DeLellisMueller_OptimalRigidityEstimates} for a proof of an analog claim. Let $\conf$ denote a corresponding conformal factor, i.\,e.\ $\g*=\exp(2\,\conf)\,\sphg*$, where $\sphg*$ denotes the standard metric of the Euclidean unit sphere. We first prove the implication
\begin{equation*} \forall\,\ve>0\quad\exists\,\ve'=\Cof{\ve'}[\ve][\delta][p]>0:\quad \Vert \confgauss-1\Vert_{\Lp^p(\sphere^2,\g*)} \le \ve' \ \Longrightarrow\ \Vert\conf\Vert_{\Wkp^{2,p}(\sphere^2,\sphg*)} \le \ve \labeleq{Lp_Reg_gauss_curv__Lp_infty}\end{equation*}
by contradiction,\pagebreak[1] i.\,e.\ we assume the existence of a constant $\ve>0$ and a sequence $\conf<n>\in\Wkp^{2,p}(\sphere^2)$ such that \eqref{Lp_Reg_gauss_curv__balanced} is satisfied for $\confmug<n>:=\exp(2\,\conf<n>)\,\confmug$, where $\confmug$ is the measure with respect to $\sphg*$, and that
\begin{equation*} \int_{\sphere^2} \vert\confgauss<n>-1\vert^p\d\confmug<n> \le \frac1n, \qquad \Vert \conf<n>\Vert_{\Wkp^{2,p}(\sphere^2,\sphg*)} \ge \ve \qquad\quad\forall\,n\in\N, \labeleq{Lp_Reg_gauss_curv__Lp_infty_not}\end{equation*}
where \pagebreak[1]$\confgauss<n>$ and $\confmug<n>$ denote the Gau\ss\ curvature and the measure with respect to the metric $\exp(2\,\conf<n>)\,\sphg*$. As first step, we prove that there is no concentric point, i.\,e.
\begin{equation*} \varrho(p):=\inf_{r>0}\limsup_{n\to\infty}\confmug<k_n>(B_r^{\sphere}(p)) = 0 \qquad\forall\,p\in\sphere^2,\,k_n\nearrow\infty \labeleq{Lp_Reg_gauss_curv_curv__conc}\end{equation*}
where $B_r^{\sphere}(p)=\{q\in\R^3\,:\,\vert q\vert=1,\,\vert p-q\vert\le r\}$ denotes the ball of radius $r$ around $p$ with respect to the metric of the surrounding Euclidean space\pagebreak[2]. Again, we prove this by contradiction and therefore assume the existence of a point $p\in\sphere^2$ with $\varrho(p)>0$. Let $\conf<k_n>\to\infty$ denote a corresponding subsequence and $\phi^p:\sphere^2\setminus\{p\}\to\R^2$ and $\conft<n>^p$ denote the stereographic projection through $p$ and the corresponding conformal factor on the Euclidean plane, respectively, i.\,e.\ ${\phi^p}_*\g*=\exp(2\,\conft<n>^p)\,\eukg*$, where $\eukg*$ denotes the Euclidean standard metric. We see
\[ \int_{\R^2}\exp(2\,\conft<n>^p) \d x \le 8\pi - 2\delta, \qquad \lim_{n\to\infty}\int_{\R^2}\vert\confgauss<n>-1\vert^p\d\confmug<n> = 0 \]
and by the diffeomorphism invariance of $\varrho$, we see $\varrho'(0)=\varrho(p)>0$ and therefore Proposition~\ref{Conc_pts__vanash_Calabi_en} implies
\[ \varrho(p)=\varrho'(0) := \inf_{r>0}\limsup_{n\to\infty}\confmug<n>^p(B_r^2(0)) \ge 4\pi, \]
where $B_r^2(x)=\{y\in\R^2\,:\,\vert x\vert\le r\}$ and $\confmug<n>^p:={\phi^p}_*\confmug<n>$ denote the Euclidean ball of radius $r$ in $\R^2$ and the corresponding measure on $\R^2$, respectively. However, there is a direction $i\in\{1,2,3\}$ and a fixed sign $\pm$ such that $B_{\nicefrac14}^{\sphere}(p)\subseteq\{\pm x_i\ge 0\}$, i.\,e.\ we get the contradiction
\[ 4\pi \le \limsup_{n\to\infty}\confmug<k_n>(U) \le \limsup_{n\to\infty}\confmug<k_n>(\sphere^2\cap\{\pm x_i\ge0\}) \le 4\pi - \delta. \]
Thus, there is no concentration point, i.\,e.\ \eqref{Lp_Reg_gauss_curv_curv__conc} holds.

Now, we prove a quantitative version of \eqref{Lp_Reg_gauss_curv_curv__conc}, i.\,e.
\begin{equation*} \forall\,\ve'>0\quad\exists\,r>0:\quad\forall\,p\in\sphere^2:\quad\limsup_{n\to\infty}\confmug<n>(B_r^{\sphere}(p)) \le \ve'\vspace{-.5em} \labeleq{Lp_Reg_gauss_curv_curv__conc_quan}. \end{equation*}
If such a radius did not exists, then there would exist a constant $\ve>0$ and a sequence $y_n\in\sphere^2$ such that $\confmug<k_n>(B_{\nicefrac1n}(y_n))>\ve'$ for some subsequence of $\confmug<k_n>$. By the compactness $\sphere^2$, we can assume that $y_n$ converges to some $y\in\sphere^2$ for which therefore
\[ \confmug<k_n>(B_r^{\sphere}(y)) \ge \confmug<k_n>(B_{\frac1n}^{\sphere}(y_n)) \ge \ve' \qquad\forall\,n>N \]
holds if $N$ is so large that $r>\vert y_n-y\vert+\nicefrac1n$ for every $n\ge N$ and where $r>0$ is arbitrary. This implicates $\varrho(y)\ge\ve'$ contradicting \eqref{Lp_Reg_gauss_curv_curv__conc}. Hence, there exists such a radius for every $\ve'>0$, i.\,e.\ \eqref{Lp_Reg_gauss_curv_curv__conc_quan} holds. Therefore, we can without loss of generality assume that $N:=(1,0,0)\in\sphere^2$ satisfies
\begin{equation*}
 \confmug<n>(B_r^{\sphere}(P)) \le \confmug<n>(B_r^{\sphere}(N)) \le \pi \qquad\forall\,n\in\N,\,P\in\sphere^2 \labeleq{Lp_Reg_gauss_curv__Ass_on_N}
\end{equation*}
for some fixed radius $r>0$. \smallskip

Let $P\in\{N,S\}$ be one of the poles, $N:=(1,0,0)$ and $S:=({-}1,0,0)$, and again denote by $\phi^P$ and $\conft<n>^P$ the corresponding stereographic projection and conformal factor, respectively. By the existence of the above uniform radius $r>0$, there exists a $s>0$ such that $\int_{B_s^2(x)}\exp(2 \conft<n>^P) \d x \le \pi $ holds for every $x\in\R^2$. With $\int_{\R^2}\exp(2\,\conft<n>^P)\d x\le8\pi-2\delta<\infty$, we deduce the existence of a sequence $y_n^P$ with
\[ \int_{B^2_s(x)} \exp(2\,\conft<n>^P)\d x \le \int_{B^2_s(y_n^P)} \exp(2\,\conft<n>^P) \d x \le \pi\qquad\forall\,x\in\R^2 \]
and by \eqref{Lp_Reg_gauss_curv__Ass_on_N} $y_n^N=0$. Let us know prove that $B_s^2(y_n^P)$ contains some positive mass bounded away from $0$, i.\,e. 
\begin{equation*} \exists\,\ve'>0:\quad\forall\,n\in\N:\quad\confmug<n>(B_s^2(y_n^P)) \ge \ve'. \labeleq{Lp_Reg_gauss_curv__min_mass}\end{equation*}
Again, we prove this by a contradiction argument. Therefore, we assume that $\liminf_n\int_{B_s(y_n^P)} \exp(2\,\conft<n>^P) \d x=0$. This implies
\[ \frac\delta2 \le \int_{B_1(0)} \exp(2\,\conft<n>^P) \d x \le C\,\int_{B_s(y_{k_n}^P)} \exp(2\,\conft<k_n>^P) \d x \xrightarrow{n\to\infty} 0 \]
for some subsequence $k_n$, where $C$ depends only on $s$. Here, we used that we can cover $B_1(0)$ with finite many balls $B_s(p)$ and that each of these balls contains less mass than $B_s(y_n^P)$. As this is again a contradiction, we know
\[ \liminf_{n\to\infty}\int_{B_s(y_n^P)} \exp(2\,\conft<n>^P) \d x>0. \]

Now, we prove that $\conf<n>$ is bounded in $\Wkp^{2,p}(\sphere^2,\sphg*)$. Let therefore $\conf<k_n>\to\infty$ be a $\Wkp^{2,p}$-norm maximizing subsequence, i.\,e.\ 
\[ \sup_{n\in\N} \Vert \conf<n>\Vert_{\Wkp^{2,p}(\sphere^2,\sphg*)} = \sup_{n\in\N} \Vert \conf<k_n>\Vert_{\Wkp^{2,p}(\sphere^2,\sphg*)}. \]
With the same argument as in the proof of Corollary~\ref{Conc_pts__vanash_Calabi_en__cor}, we can use Proposition~\ref{Conc_pts__vanash_Calabi_en__resc} to conclude that a subsequence $\conft<n>'^P:=\conft<k_n'>^P(x+y_{k_n}^P)$ of $\conft<k_n>^P(x+y_n^P)$ converges in $\Wkp^{1,q}_{\text{loc}}(\R^2)$ to $\conft'^P$, where $\conft'^P$ is as $\confth$ in Proposition~\ref{Conc_pts__vanash_Calabi_en__resc} and $q=\frac{2-p}{2p}$ if $p<2$ and $q\in\interval2\infty$ arbitrary if $p\ge2$. We see that this implies that $\Vert \conft<k_n'>'^P\Vert_{\Wkp^{2,p}(B_R(0))}$ is bounded for every $R>0$, see the proof of Proposition \ref{Conc_pts__vanash_Calabi_en__resc}. This implies $\Vert\conf<k_n'>\Vert_{\Hk^2(\sphere^2,\sphg*)}\le C$ for some constant $C\ge0$ if $\sup_n\vert y_{k_n}^P\vert<\infty$. But if there existed a subsequence $\vert y_{l_n}^S\vert\ge n+s$, then
\[ \int_{B_{\frac1n}(0)} \exp(2\,\conft<l_n>'^N) \d x = \int_{\R^2\setminus B_n(0)} \exp(2\,\conft<l_n>'^S) \d x \ge \int_{B_s(y_{l_n}^P)} \exp(2\,\conft<l_n>'^S)\d x \ge \ve' \]
would again contradict $\varrho(N)=0$, where we again used \eqref{Lp_Reg_gauss_curv__min_mass}. Thus, all in all we know
\[ \sup_{n\in\N} \Vert \conf<n>\Vert_{\Wkp^{2,p}(\sphere^2,\sphg*)} = \sup_{n\in\N} \Vert \conf<k_n'>\Vert_{\Wkp^{2,p}(\sphere^2,\sphg*)} =: c < \infty. \]
By the compactness of the Sobolev embeddings, we can therefore assume that $\conf<n>$ converges in $\Wkp^{1,q}(\sphere^2,\sphg*)$ to a function $\conf\in\Wkp^{1,q}(\sphere^2,\sphg*)$ and only have to prove $\conf\equiv0$, where $q$ is as above. However, we know that
\[ {-}\laplace \conf<n> = \confgauss<n> - \exp({-}2\conf<n>) \xrightarrow[\text{in } \Lp^p(\sphere^2,\sphg*)]{n\to\infty} 1 - \exp({-}2\,\conf), \]
where we used that $\frac1C\,\confmug\le\confmug<n>\le C\,\confmug$ due to the boundedness of $\conf<n>$. We get
\[ {-}\laplace\conf = 1 - \exp({-}2\,\conf) \quad\text{weakly in }\sphere^2 \]
by combining the two above convergences. The convergence of $\conf<n>$ furthermore implies $\Vert\conf\Vert_{\Wkp^{1,q}(\sphere^2)}\le C$. Using the regularity of the Laplace operator, we conclude $\conf\in\Ck^\infty(\sphere^2)$ and ${-}\laplace\conf = 1 - \exp(2\,\conf)$ pointwise everywhere in $\sphere^2$. Chen-Li's classification theorem \cite[Thm~1]{chen1991} therefore implies that there is a point $x_0$ and a factor $\lambda>0$ such that
\[ \conft(x) = \ln(\frac{2\lambda}{\lambda^2+\vert x-x_0\vert^2}), \]
where $\conft$ is the conformal factor after the stereographic projection, i.\,e.\ $\conft$ is defined by ${\varphi^N}_*(\exp(2\,\conf)\sphg)=:\exp(2\,\conft)\outg*$. In particular, we conclude $4\pi=\int_{\sphere^2} \exp(2\,\conf) \d\confmug$.
Correspondingly, \eqref{Lp_Reg_gauss_curv__balanced} implies
\begin{align*}
	\int_{B_1^2(0)}\exp(2\,\conft) \d x ={}& \int_{\sphere^2\cap\{x_1\le0\}} \exp(2\,\conf) \d\confmug = \int_{\sphere^2\cap\{x_1\ge0\}} \exp(2\,\conf) \d\confmug \\
		={}& \int_{\R^2\setminus B_1^2(0)}\exp(2\,\conft) \d x,
\end{align*}
i.\,e.\ $\lambda=1$ and $x_0=0$. Hence, $\conft(x) = \ln(\nicefrac2{(1+\vert x\vert^2)})$ and therefore $\exp(2\,\conf(x))\equiv1$, i.\,e.\ $\conf\equiv0$. This means $\Vert\conf<n>\Vert_{\Wkp^{1,q}(\sphere^2,\sphg*)}\to0$ for $n\to\infty$. Thus, $\conf<n>$ uniformly converges to $0$. Thus, the assumption on $\confgauss<n>={-}\exp(2\,\conf<n>)\,\laplace\conf<n>$ implies that $\Vert\conf<n>\Vert_{\Wkp^{2,p}(\sphere^2,\sphg*)}\to0$ for $n\to\infty$. This contradicts \eqref{Lp_Reg_gauss_curv__Lp_infty_not} and therefore, we have finally proven \eqref{Lp_Reg_gauss_curv__Lp_infty}.\smallskip

Now, let $\ve>0$ be arbitrary and let $\g*$ be a metric on the sphere $\sphere^2$ satisfying $\confmug(\sphere^2)\in\interval\delta{8\pi-\delta}$ and $\Vert \confgauss-1\Vert_{\Lp^p(\sphere^2,\g*)} \le \ve'$, where $\ve'$ is as in \eqref{Lp_Reg_gauss_curv__Lp_infty}. Using conformal maps, we can again assume that the balancing condition \eqref{Lp_Reg_gauss_curv__balanced} is satisfied and therefore \eqref{Lp_Reg_gauss_curv__Lp_infty} implies $\Vert\conf\Vert_{\Wkp^{2,p}(\sphere^2,\sphg*)} \le \ve$, where $\conf$ is the corresponding conformal factor, i.\,e.\ $\g*=\exp(2\,\conf)\,\sphg*$. Furthermore, we know that the Fr\'echet derivative of the Gau\ss\ curvature map $\boldsymbol{\confgauss}:\conf'\mapsto {-}\laplace \conf' + \exp({-}2\,\conf')$ is 
\[ D\confgauss(0):\Wkp^{2,p}(\sphere^2,\sphg*)\to\Lp^p(\sphere^2,\sphg*):\conf'\mapsto{-}\laplace\conf' - 2\relax\conf' \]
and that $D\confgauss(\conf)$ is continuous in $\conf\in\Wkp^{2,p}(\sphere^2,\sphg*)$. As ${+}2$ is not a Eigenvalue of $\laplace$ on the Euclidean standard sphere, this derivative is invertible. Thus, there is are constants $\eta>0$ and $\eta'>0$ such that for every $\confgauss'\in\Lp^p(\sphere^2,\sphg*)$ with $\Vert\confgauss'-1\Vert_{\Lp^p(\sphere^2,\sphg*)}\le\eta'$ there exists exactly one conformal factor $\conf'\in\Wkp^{2,p}(\sphere^2,\sphg*)$ with $\Vert\conf'\Vert_{\Wkp^{2,p}(\sphere^2,\sphg*)}\le\eta$ and $\boldsymbol{\confgauss}(\conf')=\confgauss'$. Furthermore, we see $\Vert\conf''\Vert_{\Wkp^{2,q}(\sphere^2,\sphg*)}\le \Vert L_{\g*'}\conf''\Vert_{\Lp^q(\sphere^2,\sphg*)}$ for every $q\in\interval1\infty$ and for the Fr\'echet derivative $L_{\g*'}$ of the Gau\ss\ curvature map for metrics $\g*'$ in a small $\Wkp^{2,p}$-neighborhood of $\sphg*$. Thus, we conclude 
\[ \Vert\conf'\Vert_{\Wkp^{2,q}(\sphere^2,\sphg*)} \le C\,\Vert\boldsymbol{\confgauss}(\conf')-1\Vert_{\Lp^q(\sphere^2,\sphg*)} \qquad\forall\,\Vert\conf'\Vert_{\Wkp^{2,p}(\sphere^2,\sphg*)} < \eta,\;q\in\interval1\infty \]
for some constant $C$ depending only on $p\in\interval1\infty$. By choosing a sufficiently small $\ve'\in\interval0{\eta'}$, we conclude
\[ \Vert\conf\Vert_{\Wkp^{2,p}(\sphere^2,\sphg*)} \le C\,\Vert\confgauss-1\Vert_{\Lp^p(\sphere^2,\sphg*)} \le C\,\Vert\confgauss-1\Vert_{\Lp^p(\sphere^2,\g*)}.\qedhere \]
\end{proof}%

\clearpage
\bibliography{bib}
\bibliographystyle{alpha}\vfill
\end{document}